\documentclass[12pt]{amsart}

\usepackage{amssymb,latexsym}
\usepackage{bm}
\usepackage{mathrsfs}
\usepackage{graphicx}
\usepackage{psfrag}
\usepackage{wrapfig}
\usepackage{color}
\usepackage{enumitem}
\setlist[enumerate]{parsep=0pt plus 4pt,topsep=0pt plus 4pt}
\usepackage{url}
\usepackage{hyperref}
\usepackage{tikz}
\definecolor{darkblue}{RGB}{0,0,160}
\hypersetup{colorlinks,citecolor=black,filecolor=black,%
            linkcolor=darkblue,urlcolor=darkblue}
\allowdisplaybreaks

\newcommand{\excise}[1]{}

\newtheorem{thm}{Theorem}[section]
\newtheorem{lemma}[thm]{Lemma}

\newtheorem{cor}[thm]{Corollary}
\newtheorem{prop}[thm]{Proposition}

\theoremstyle{definition}
\newtheorem{example}[thm]{Example}
\newtheorem{remark}[thm]{Remark}
\newtheorem{defn}[thm]{Definition}

\numberwithin{equation}{section}

\newcounter{separated}

\newcommand{\Ring}[1]{\ensuremath{\mathbb{#1}}}

\renewcommand\>{\rangle}

\newcommand\0{\mathbf{0}}

\newcommand\NN{\Ring{N}}

\newcommand\QQ{{\mathbb Q}}
\newcommand\RR{{\mathbb R}}

\newcommand\ZZ{{\mathbb Z}}
\newcommand\bb{{\mathbf b}}

\newcommand\ee{{\mathbf e}}

\newcommand\ii{{\mathbf i}}
\newcommand\jj{{\mathbf j}}
\newcommand\kk{\Bbbk}
\newcommand\mm{{\mathfrak m}}

\newcommand\qq{{\mathbf q}}

\newcommand\xx{{\mathbf x}}

\newcommand\cH{H}

\newcommand\cM{M}
\newcommand\cN{N}
\newcommand\cP{P}
\newcommand\cQ{Q}

\newcommand\cZ{Z}

\newcommand\ob{\hspace{.3ex}^{\raisebox{.8ex}{$\scriptscriptstyle\circ$}}%
	\hspace{-1.22ex}b}
\newcommand\vC{\check{\mathcal C}}

\renewcommand\aa{{\mathbf a}}
\renewcommand\phi{\varphi}

\newcommand\too{\longrightarrow}

\newcommand\into{\hookrightarrow}
\newcommand\otni{\hookleftarrow}

\newcommand\onto{\twoheadrightarrow}

\newcommand\spot{{\hbox{\raisebox{1pt}{\tiny$\scriptscriptstyle\bullet$}}}}

\newcommand\minus{\smallsetminus}

\newcommand\cupdot{\ensuremath{\mathbin{\mathaccent\cdot\cup}}}
\newcommand\goesto{\rightsquigarrow}
\newcommand\dirlim{\varinjlim}

\newcommand\nothing{\varnothing}

\newcommand\filleftmap{\mathord\leftarrow \mkern-6mu
	\cleaders\hbox{$\mkern-2mu \mathord- \mkern-2mu$}\hfill
	\mkern-6mu \mathord-}
\newcommand\fillrightmap{\mathord- \mkern-6mu
	\cleaders\hbox{$\mkern-2mu \mathord- \mkern-2mu$}\hfill
	\mkern-6mu \mathord\rightarrow}

\newcommand{\lhookdownarrow}{\rotatebox[origin=c]{-90}{$\into$}}
\newcommand{\twoheaddownarrow}{\rotatebox[origin=c]{-90}{$\onto$}}
\renewcommand\iff{\Leftrightarrow}
\renewcommand\epsilon{\varepsilon}
\renewcommand\implies{\Rightarrow}

\newcommand\dx[1][]{\delta^{\hspace{.1ex}\xi}}

\newcommand\ol[1]{{\overline{#1}}}

\definecolor{lightred}{rgb}{1,.3,.3}
\newcommand\red{\color{lightred}}
\newcommand\blu{\color{blue}}

\newcommand{\aoverb}[2]{{\genfrac{}{}{0pt}{1}{#1}{#2}}}
\def\twoline#1#2{\aoverb{\scriptstyle {#1}}{\scriptstyle {#2}}}

\DeclareMathOperator\Hom{Hom} 
\DeclareMathOperator\Mor{Mor} 

\DeclareMathOperator\coker{coker} 

\newcommand\monomialmatrix[3]{{
\begin{array}{@{}r@{\:}r@{}c@{}l@{}}
  \begin{array}{@{}c@{}}		
	\begin{array}{@{}r@{}}
	\\
	#1
	\end{array}\!
  \end{array}						
&
  \begin{array}{@{}c@{}}		
	\begin{array}{@{}l@{}}\\				
	\end{array}						
	\\							
	\left[\begin{array}{@{}l@{}}				
	#3							
	\end{array}\!						
	\right.							
  \end{array}							
&
  #2					
&
  \begin{array}{@{}c@{}}		
	\begin{array}{@{}l@{}}\\				
	\end{array}						
	\\							
	\left.\!\begin{array}{@{}l@{}}				
	#3							
	\end{array}						
	\right]							
  \end{array}							
\end{array}
}}

\begin{document}

\mbox{}
\vspace{-4.3ex}
\title[Homological algebra of modules over posets]%
      {Homological algebra of modules over posets}
\author{Ezra Miller}
\address{Mathematics Department\\Duke University\\Durham, NC 27708}
\urladdr{\url{http://math.duke.edu/people/ezra-miller}}

\makeatletter
  \@namedef{subjclassname@2020}{\textup{2020} Mathematics Subject Classification}
\makeatother
\subjclass[2020]{Primary: 05E40, 13E99, 06B15, 13D02, 55N31, 06A07,
32B20, 14P10, 52B99, 13A02, 13P20, 68W30, 13P25, 62R40, 06A11, 06F20,
06F05, 68T09;
Secondary: 13C99, 05E16, 32S60, 14F07, 62R01, 62H35, 92D15, 92C15,
13F99, 20M14, 14P15, 06B35, 22A25}

\date{10 August 2020}

\begin{abstract}
Homological algebra of modules over posets is developed, as closely
parallel as possible to that of finitely generated modules over
noetherian commutative rings, in the direction of finite presentations
and resolutions.  Centrally at issue is how to define finiteness to
replace the noetherian hypothesis which fails.  The tameness condition
introduced for this purpose captures finiteness for variation in
families of vector spaces indexed by posets in a way that is
characterized equivalently by distinct topological, algebraic,
combinatorial, and homological manifestations.  Tameness serves both
theoretical and computational purposes: it guarantees finite
presentations and resolutions of various sorts, all related by a
syzygy theorem, amenable to algorithmic manipulation.  Tameness and
its homological theory are new even in the finitely generated discrete
setting of $\NN^n$-gradings, where tame is materially weaker than
noetherian.  In the context of persistent homology of filtered
topological spaces, especially with multiple real parameters, the
algebraic theory of tameness yields topologically interpretable data
structures in terms of birth and death of homology classes.
\end{abstract}
\maketitle

\vspace{-2ex}
\setcounter{tocdepth}{2}
\tableofcontents

\mbox{}\vspace{-5.6ex}
\section{Introduction}\label{s:intro}

\subsection*{Overview}\label{sub:overview}

A module over a poset is a family of vector spaces indexed by the
poset elements with a homomorphism for each poset relation.  The setup
is inherently commutative: the homomorphism for a poset relation $p
\preceq q$ is the composite of homomorphisms for the relations $p
\preceq r$ and $r \preceq q$ whenever $r$ lies between $p$ and~$q$.
This paper lays the foundation for an extensive theory of modules over
arbitrary posets, with a view toward abstract mathematical theory,
algorithmic challenges, and statistical implications.  The mathematics
includes commutative and homological algebra as they interact with
topological, analytic, algebraic, or polyhedral geometric structure on
the poset, if any is given.  The algorithmic challenges involve
effectively encoding and manipulating arbitrary poset modules.  The
statistical considerations stem from applied topology, where modules
over posets arise from persistent homology.

This installment covers initial homological aspects: the extent to
which modules over posets behave like multigraded modules over
polynomial rings when it comes to finite presentations and
resolutions.  The long-term investigation tests the frontier of
multigraded algebra regarding how far one can get without a ring and
with no hypotheses on the multigrading other than a partial order.
The syzygy theorem for poset modules here vastly generalizes the one
for finitely generated modules over polynomial rings, along the way
introducing finite data structures to enable algorithmic computation.

The poset of utmost interest is the real vector space~$\RR^n$, with
its usual componentwise partial order.  A~module over~$\RR^n$ is
equivalently an $\RR^n$-graded module over the polynomial ring whose
exponents are allowed to be nonnegative real numbers instead of
integers.  In this setting, the noetherian hypothesis fails
spectacularly, and essentially nothing is known about homological
behavior of its category of modules.  The infrastructure developed
here meets the lack of noetherian hypotheses head~on, to open
the~possibility of working directly with modules over~$\RR^n$ and,
with no additional difficulty, arbitrary posets.

The focus, and the most subtle point, is the nature of a suitable
finiteness condition to replace the noetherian hypothesis.  The
\emph{tame} condition introduced here is the natural candidate because
it captures equivalent topological, algebraic, combinatorial, and
homological manifestations of finiteness for variation of vector
spaces parametrized by a poset.  Tameness serves both theoretical and
computational purposes: it guarantees various finite presentations and
resolutions all related by a syzygy theorem, and the data structures
thus produced are amenable to algorithmic manipulation.  Tameness, its
syzygy theorem, and its data structures are new and theoretically as
well as computationally valuable even in the discrete setting over the
poset~$\ZZ^n$, which is ordinary commutative algebra of polynomial
rings, where tame is much weaker than noetherian.

No restriction on the underlying poset is required.  For example, the
lack of local finiteness of~$\RR^n$ is immaterial.  Moreover, in that
particular setting, if the partial orderings and the modules possess
supplementary geometry, be it subanalytic, semialgebraic, or
piecewise-linear, for instance, then the data structures and
transitions between the topological, algebraic, combinatorial, and
homological perspectives take advantage of and preserve the geometry.

Beyond the abstract route to graded module theory over real-exponent
polynomial rings and arbitrary posets, one impetus for these
developments lies in data science applications, where the poset
consists of ``parameters'' indexing a family of topological subspaces
of a fixed topological space.  Taking homology of the subspaces in
this topological filtration yields a poset module, called the
persistent homology of the filtration, referring to how homology
classes are born, persist for a while, and then die as the parameter
moves up in the poset.  In ordinary persistent homology, the poset is
totally ordered---usually the real numbers~$\RR$, the integers~$\ZZ$,
or a subset $\{1,\dots,m\}$.  This case is well studied (see
\cite{edelsbrunner-harer2010}, for example), and the algebra is
correspondingly simple \cite{crawley-boevey2015}.  Persistence with
multiple totally ordered parameters, introduced by Carlsson and
Zomorodian \cite{multiparamPH}, has been developed in various ways,
often assuming that the poset is $\NN^n$.  That discrete framework has
been preferred in part because it arises frequently when filtering
finite simplicial complexes, but also because settings involving
continuous parameters unavoidably produce modules that fail to be
finitely presented in several fundamental ways.  Tameness, with its
data structures and syzygy theorem, circumvent these~limitations.

Multigraded algebra can be expressed equivalently in terms of modules,
or sheaves, or functors, or derived categories, and the literature
exhibits all of these.  The exposition throughout this paper is
intentionally kept at the most elementary level, with posets instead
of thin skeletal categories, for instance, and with modules instead of
sheaves or functors on posets.  At the risk of masking the depth of
the content in these enriched contexts, this choice of elementary
language makes the exposition accessible to a wide audience, including
statisticians applying persistent homology in addition to topologists,
combinatorialists, algebraists, geometers, and programmers.

The power of the foundations here is demonstrated by
\cite{strat-conical}, for example, which proves conjectures made by
Kashiwara and Schapira concerning the relationship between subanalytic
and piecewise-linear stratifications of vector spaces and
constructibility of sheaves on real vector spaces in the derived
category with microsupport restricted to a cone; see
\cite[Conjecture~3.17]{kashiwara-schapira2017}
	\footnote{Bibliographic note: this conjecture appears in~v3
	(the version cited here) and earlier versions of the cited
	arXiv preprint.  It does not appear in the published version
	\cite{kashiwara-schapira2018}, which is~v6 on the arXiv.}
%
and \cite[Conjecture~3.20]{kashiwara-schapira2019}.  The theory here
as well as in \cite{prim-decomp-pogroup, essential-real} was developed
simultaneously and independently from that of Kashiwara and Schapira
\cite{kashiwara-schapira2018},
cf.~\cite{qr-codes}.  The conical-microsupport theory is roughly
equivalent to the subanalytic special case of poset module theory for
partially ordered real vector spaces, and similarly for the later PL
theory \cite{kashiwara-schapira2019}; this is essentially the content
of \cite{strat-conical}.  A~detailed comparison of the two viewpoints,
including key differences, is left to \cite{strat-conical}, where the
derived sheaf background is reviewed.

\setcounter{tocdepth}{-1}
\subsection*{Acknowledgements}\label{sub:acknowledgements}
\setcounter{tocdepth}{2}

First, a special acknowledgement goes to Ashleigh Thomas, who has been
a long-term collaborator on this project.  She was listed as an author
on earlier drafts of \cite{qr-codes} (of which this is roughly the
first quarter), but her contributions lie more properly beyond these
preliminaries (see \cite{primary-distance}, for example), so she
declined in the end to be named as an author on this installment.
Early in the development of the ideas here, Thomas put her finger on
the continuous rather than discrete nature of multiparameter
persistence modules for fly wings.  She computed the first examples
explicitly, namely those in Example~\ref{e:toy-model-fly-wing}, and
produced the biparameter persistence diagrams there as well as some of
the figures in Example~\ref{e:flange-switching}.

Justin Curry pointed out connections from the combinatorial viewpoint
taken here, in terms of modules over posets, to higher notions in
algebra and category theory, particularly those involving
constructible sheaves, which are in the same vein as Curry's proposed
uses of them in persistence \cite{curry-thesis}; see
Remarks~\ref{r:curry}, \ref{r:indicator}, \ref{r:lurie},
and~\ref{r:kan-extension}.

The author is indebted to David Houle, whose contribution to this
project was seminal and remains ongoing; in particular, he and his lab
produced the fruit fly wing images \cite{houle03}.  Paul Bendich and
Joshua Cruz took part in the genesis of this project, including early
discussions concerning ways to tweak persistent (intersection
\cite{bendich-harer2011}) homology for investigations of fly wings.
Ville Puuska discovered several errors in an early version of
Section~\ref{s:encoding}, resulting in substantial correction and
alteration; see Examples~\ref{e:puuska-nonconstant-isotypic}
and~\ref{e:puuska-nontransitive}.  Banff International Research
Station provided an opportunity for valuable feedback and suggestions
at the workshop there on Topological Data Analysis (August, 2017) as
parts of this research were being completed; many participants,
especially the organizers, Uli Bauer and Anthea Monod, as well as
Michael Lesnick, shared important perspectives and insight.  Thomas
Kahle requested that Proposition~\ref{p:determined} be an equivalence
instead of merely the one implication it had stated.  Hal Schenck gave
helpful comments on an earlier version of the Introduction.  Passages
in Examples~\ref{e:fly-wing-filtration} and~\ref{e:toy-model-fly-wing}
are based on or taken verbatim from~\cite{fruitFlyModuli}.  Portions
of this work were funded by NSF grant~DMS-1702395.

\subsection{Modules over posets}\label{sub:modules}

There are many essentially equivalent ways to think of a poset module.
The definition in the first line of this Introduction is among the
more elementary formulations; see Definition~\ref{d:poset-module} for
additional precision.  Others include a
\begin{itemize}
\item%
representation of a poset \cite{nazarova-roiter};

\item%
functor from a poset to the category of vector spaces (e.g., see
\cite{curry2019});

\item%
vector-space valued sheaf on a poset (e.g., see
\cite[\S4.2]{curry-thesis} or \cite[\S3.3]{strat-conical});

\item%
representation of a quiver with (commutative) relations (e.g., see
\cite[\S A.6]{oudot2015});


\item%
representation of the incidence algebra of a poset
\cite{doubilet-rota-stanley1972}; or

\item%
module over a directed acyclic graph \cite{chambers-letscher2018}.
\end{itemize}
The premise here is that commutative algebra provides an elemental
framework out of which flows corresponding structure in these other
contexts, in which the reader is encouraged to interpret all of the
results.  \cite{strat-conical} provides an example of how that can
look, in that case from sheaf perspectives.  Expressing the
foundations via commutative algebra is natural for its infrastructure
surrounding resolutions.  And as the objects are merely graded vector
spaces with linear maps among them---there are no rings to act---it is
also the most elementary language available.

Some of the formulations of poset modules are only valid when the
poset is assumed to be locally finite (see
\cite{doubilet-rota-stanley1972}, for instance), or when the object
being acted upon satisfies a finitary hypothesis
\cite{khripchenko-novikov2009} in which the algebraic information is
nonzero on only finitely many points in any interval.  This is not a
failing of any particular formulation, but rather a signal that the
theory has a different focus.  Combinatorial formulations are built
for enumeration.  Representation theories are built for decomposition
into and classification of irreducibles.  While commutative algebra
appreciates a direct sum decomposition when one is available, such as
over a noetherian ring of dimension~$0$, its initial impulse is to
relate arbitrary modules to simpler ones by less restrictive
decomposition, such as primary decomposition, or by resolution, such
as by projective or injective modules.  That is the tack taken here.

\subsection{Topological tameness}\label{sub:tame}

The \emph{tame} condition (Definitions~\ref{d:constant-subdivision}
and~\ref{d:tame}) on a module~$M$ stipulates that the poset admit a
partition into finitely many domains of constancy for~$M$.  This
finiteness generalizes topological tameness for persistent homology in
a single parameter (see \cite[\S3.8]{chazal-deSilva-glisse-oudot2016},
for example), reflecting the intuitive notion that given a filtration
of a topological space from data, only finitely many topologies should
appear.  Tameness is thus a topological concept, designed to control
the size and variation of homology groups of subspaces in a fixed
topological~space.

\begin{example}\label{e:fly-wing-filtration}
\enlargethispage*{2.2ex}%
Let $\cQ = \RR_- \times \RR_+$ with the coordinatewise partial order,
so $(r,s) \in \cQ$ for any nonnegative real numbers~$-r$ and~$s$.  Let
$X = \RR^2$ be the plane containing an embedded planar graph.  Define
$X_{rs} \subseteq X$ to be the set of points at distance at least~$-r$
from every vertex and within~$s$ of some edge.  Thus $X_{rs}$ is
obtained by removing the union of the balls of radius~$r$ around the
vertices from the union of $s$-neighborhoods of the edges.  In the
following portion of an embedded graph, $-r$ is approximately
twice~$s$:
$$%
\includegraphics[height=23mm]{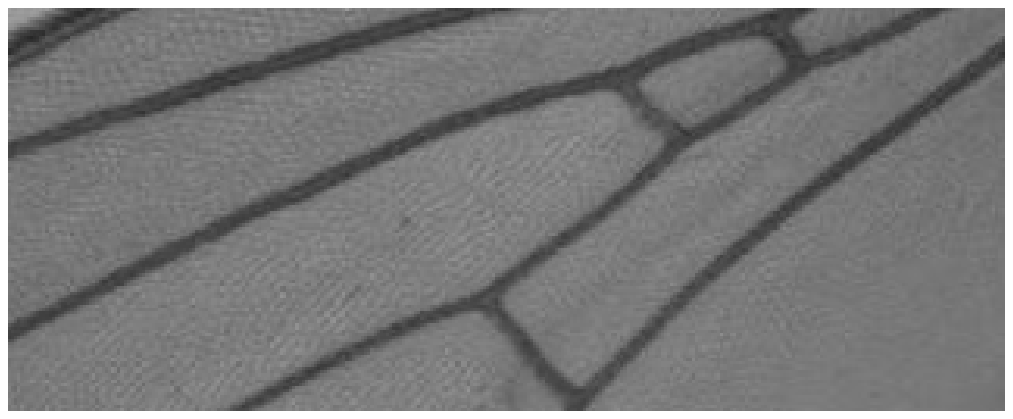}
\quad
\raisebox{10mm}{$\goesto$}
\quad
\includegraphics[height=23mm]{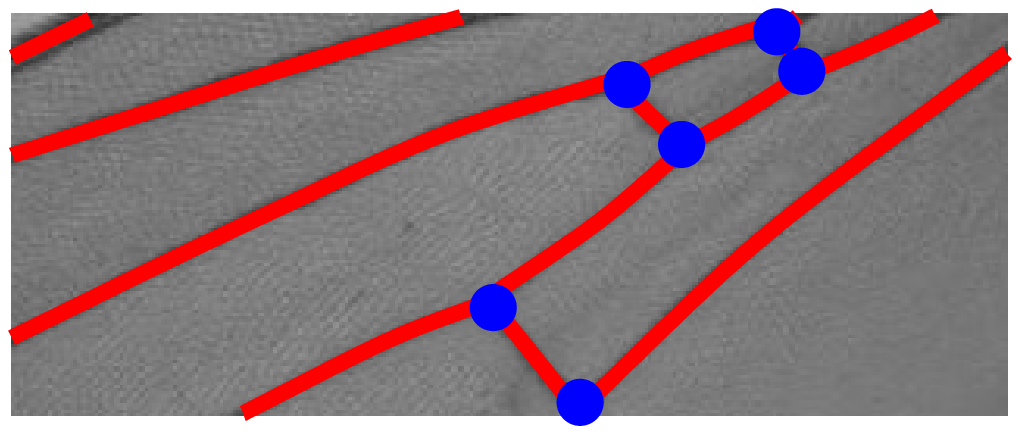}
\vspace{-1ex}
$$
The biparameter persistent homology module $\cM_{rs} = H_0(X_{rs})$
summarizes the geometry of the embedded planar graph.
\end{example}

Relevant properties of these modules are best highlighted in a
simplified setting.

\begin{example}\label{e:toy-model-fly-wing}
Using the setup from Example~\ref{e:fly-wing-filtration}, the zeroth
persistent homology for the toy-model embedded graph at left in
Figure~\ref{f:toy-model-fly-wing} is the $\RR^2$-module $\cM$ shown at
\begin{figure}[ht]
\vspace{-8.5ex}
$$%
\begin{array}[b]{@{}c@{}}
\mbox{}\\[28.5pt]
\includegraphics[height=30mm]{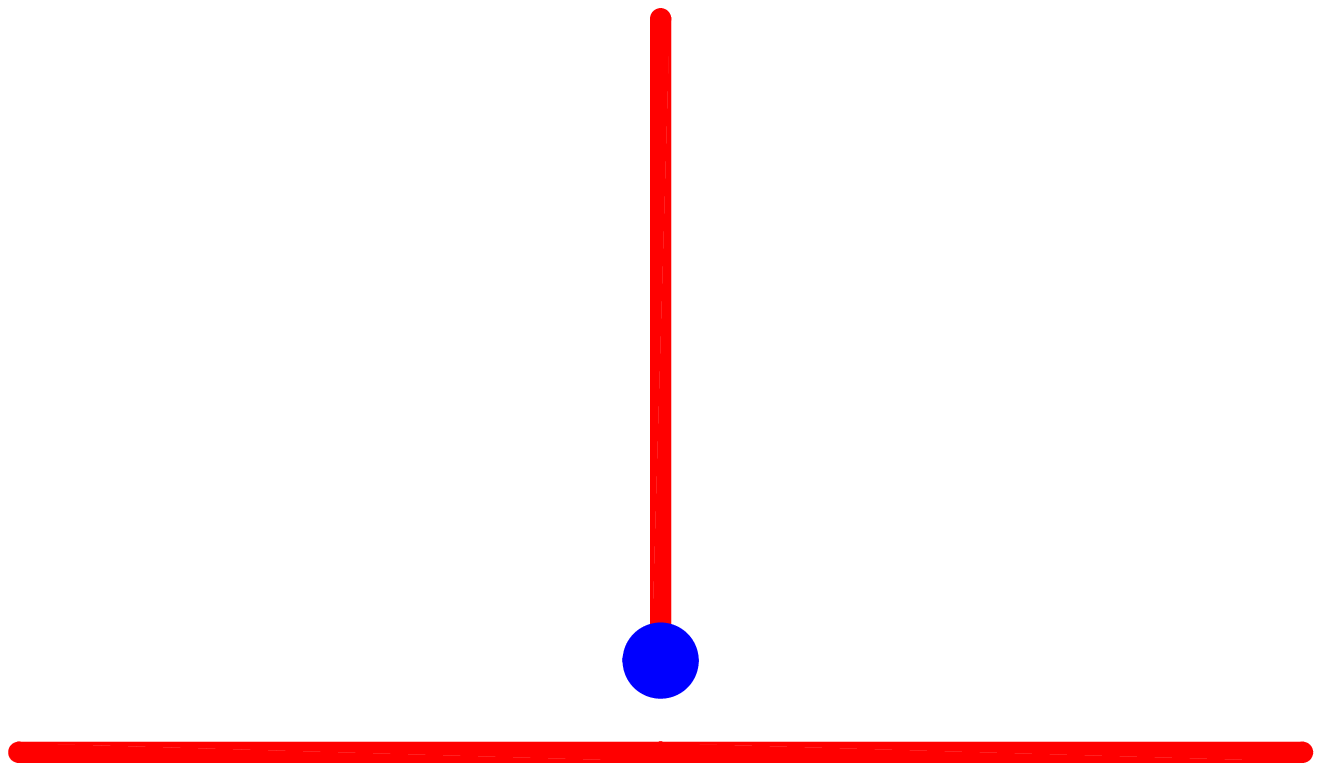}
\\[-22.5pt]
\end{array}
\begin{array}[b]{@{\ \ }c@{}}
\hspace{-3pt}\goesto\\[7mm]\mbox{}
\end{array}
\qquad
\begin{array}[b]{@{\hspace{-10pt}}r@{\hspace{-10pt}}|@{}l@{}}
\begin{array}{@{}c@{}}
\psfrag{r}{\tiny$r \to$}
\psfrag{s}{\tiny$\begin{array}{@{}c@{}}\uparrow\\[-.5ex]s\end{array}$}
\includegraphics[height=30mm]{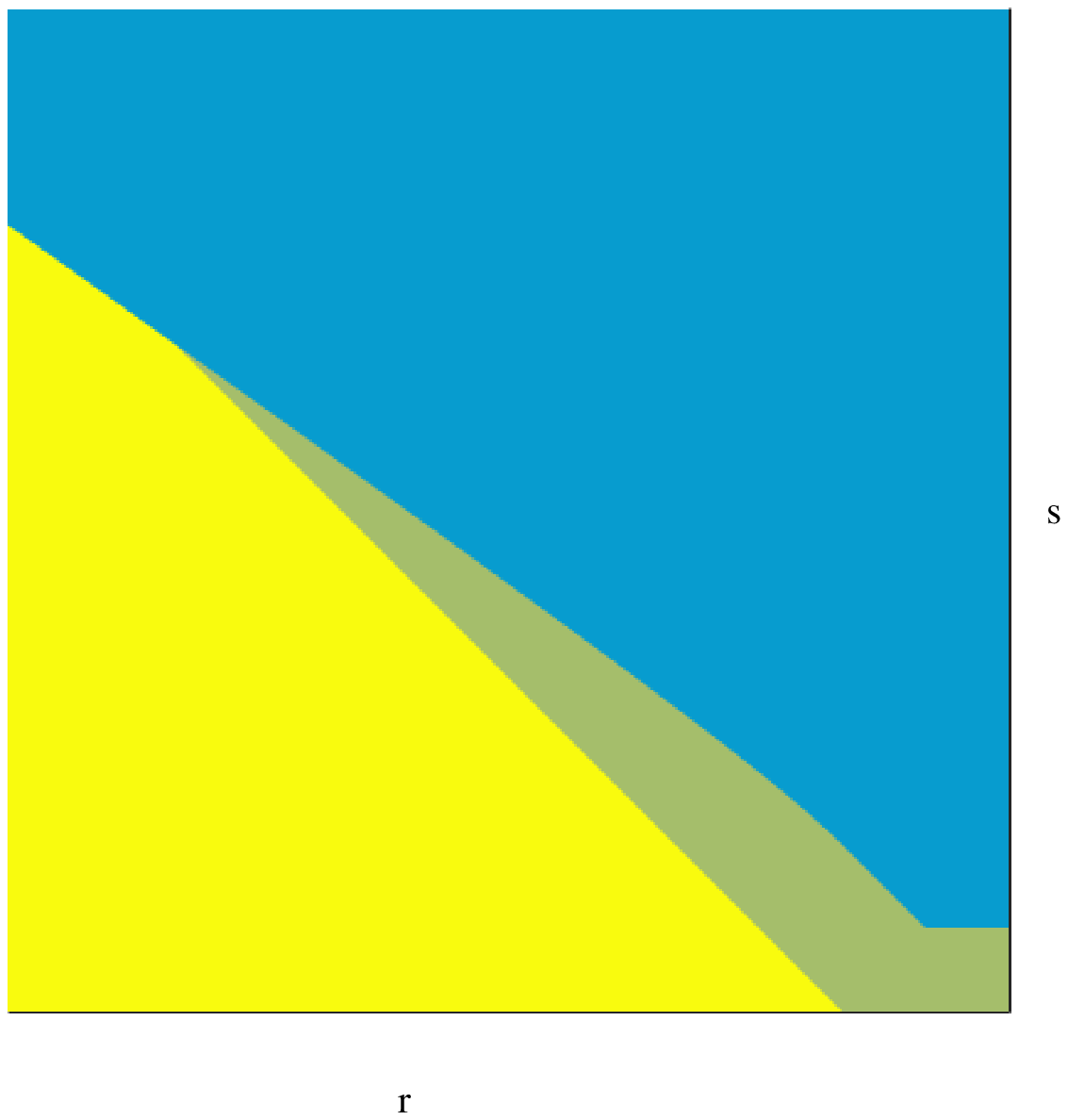}\\[-7.7pt]
\end{array}
&\,\,\,\\[-6pt]\hline
\end{array}
\qquad
\begin{array}[b]{@{\ \ }c@{}}
\hspace{-3pt}\goesto\\[7mm]\mbox{}
\end{array}
\qquad
\begin{array}[b]{@{\hspace{-10pt}}r@{\hspace{-10pt}}|@{}l@{}}
\begin{array}{@{}c@{}}
\psfrag{1}{\tiny$\kk$}
\psfrag{2}{\tiny$\kk^2$}
\psfrag{3}{\tiny$\kk^3$}
\includegraphics[height=30mm]{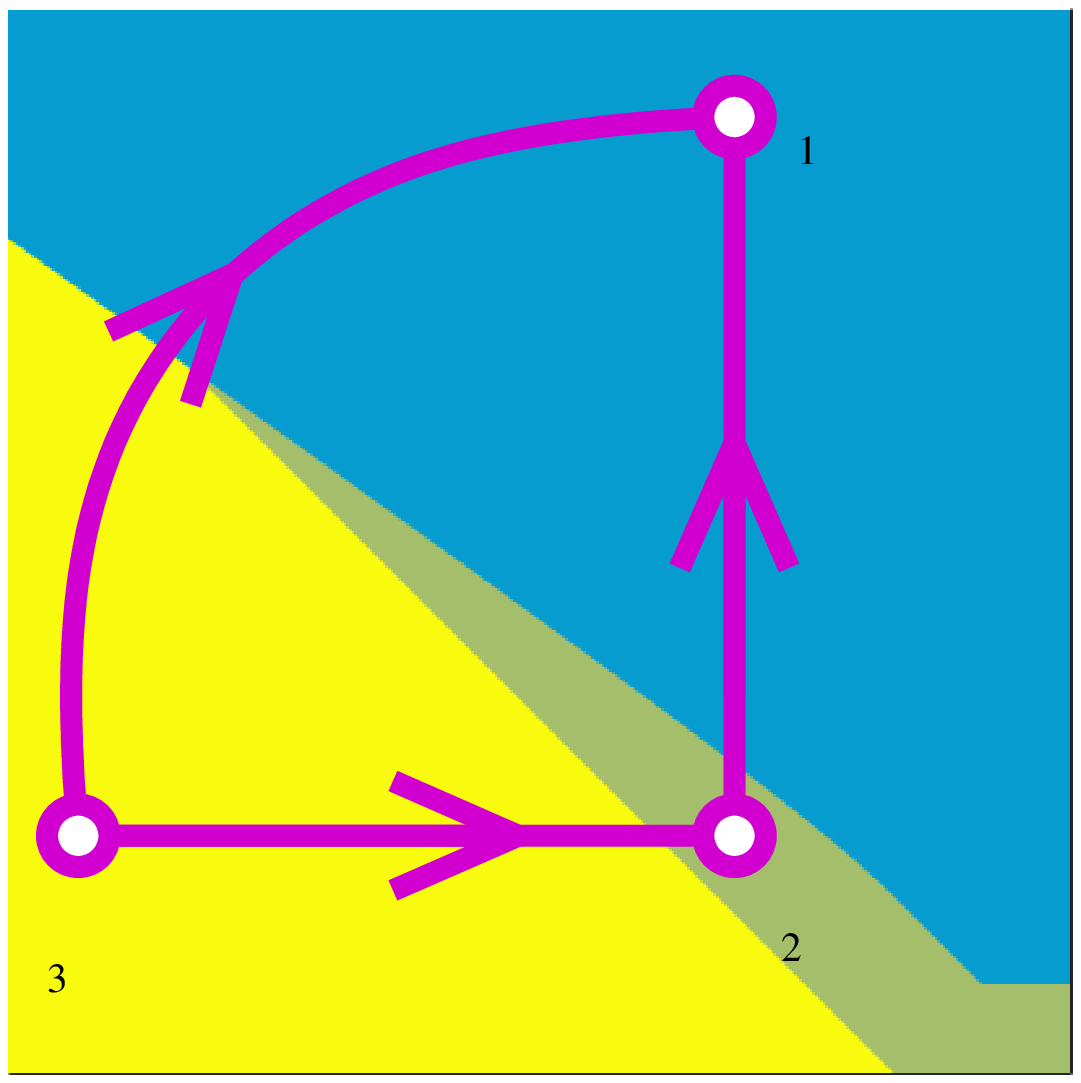}\\[-7.7pt]
\end{array}
&\,\,\,\\[-6pt]\hline
\end{array}
\vspace{-2ex}
$$
\caption{$\RR^2$-module and finite encoding\vspace{-1ex}}
\label{f:toy-model-fly-wing}
\end{figure}
center.  Each point of $\RR^2$ is colored according to the dimension
of its associated vector space in~$\cM$, namely $3$, $2$, or~$1$
proceeding up (increasing~$s$) and to the right (increasing~$r$).  The
structure homomorphisms $\cM_{rs} \to \cM_{r's'}$ are all surjective.

This $\RR^2$-module fails to be finitely presented for three
fundamental reasons.  First, the three generators sit infinitely far
back along the $r$-axis.  (Fiddling with the sign on~$r$ does not
help: the natural maps on homology proceed from infinitely large
radius to~$0$ regardless of how the picture is drawn.)  Second, the
relations that specify the transition from vector spaces of
dimension~$3$ to those of dimension~$2$ or~$1$ lie along a real
algebraic curve, as do those specifying the transition from
dimension~$2$ to dimension~$1$.  These curves have uncountably many
points.  Third, even if the relations are discretized---restrict~$\cM$
to a lattice $\ZZ^2$ superimposed on~$\RR^2$, say---the relations
march off to infinity roughly diagonally away from the origin.  (See
Example~\ref{e:encoding} for the right-hand~image.)

Nonetheless, the $\RR^2$-module here is tame, with four constant
regions: over the bottom-left region (yellow) the vector space
is~$\kk^3$; over the middle (olive) region the vector space
is~$\kk^2$; over the upper-right (blue) region the vector space
is~$\kk$; and over the remainder of~$\RR^2$ the vector space is~$0$.
The homomorphisms between these vector spaces do not depend on which
points in the regions are selected to represent them.  For instance,
$\kk^3 \to \kk^2$ always identifies the two basis vectors
corresponding to the connected components that are the left and right
halves of the horizontally~infinite~red~strip.
\end{example}

In principle, tameness can be reworked to serve directly as a data
structure for algorithmic computation, especially in the presence of
an auxiliary hypothesis to regulate the geometry of the constant
regions---when they are semialgebraic or piecewise linear
(Definition~\ref{d:auxiliary-hypotheses}.\ref{i:semialgebraic}
or~\ref{d:auxiliary-hypotheses}.\ref{i:PL}), for example.  The
algorithms would generalize those for polyhedral ``sectors'' in the
discrete case \cite{injAlg} (or see~\cite[Chapter~13]{cca}).

\subsection{Combinatorial tameness: finite encoding}\label{sub:combin-tame}

Whereas the topological notion of tameness requires little more than
an arbitrary subdivision of the poset into regions of constancy
(Definition~\ref{d:constant-subdivision}), the combinatorial
incarnation imposes additional structure on the constant regions,
namely that they should be partially ordered in a natural way.  More
precisely, it stipulates that the module~$\cM$ should be pulled back
from a $\cP$-module along a poset morphism $\cQ \to \cP$ in which
$\cP$ is a finite poset and the $\cP$-module has finite dimension as a
vector space over the field~$\kk$ (Definition~\ref{d:encoding}).

\begin{example}\label{e:encoding}
The right-hand image in Example~\ref{e:toy-model-fly-wing} is a finite
encoding of~$\cM$ by a three-element poset $\cP$ and the $\cP$-module
$H = \kk^3 \oplus \kk^2 \oplus \kk$ with each arrow in the image
corresponding to a full-rank map between summands of~$H$.
Technically, this is only an encoding of~$\cM$ as a module over $\cQ =
\RR_- \times \RR_+$.  The poset morphism $\cQ \to \cP$ takes all of
the yellow rank~$3$ points to the bottom element of~$\cP$, the olive
rank~$2$ points to the middle element of~$\cP$, and the blue rank~$1$
points to the top element of~$\cP$.  (To make this work over all
of~$\RR^2$, the region with vector space dimension~$0$ would have to
be subdivided, for instance by introducing an antidiagonal extending
downward from the origin, thus yielding a morphism from~$\RR^2$ to a
five-element poset.)  This encoding is semialgebraic
(Definition~\ref{d:auxiliary-hypotheses}): its fibers are real
semialgebraic sets.
\end{example}

In general, constant regions need not be situated in a manner that
makes them the fibers of a poset morphism (Example~\ref{e:subdivide}).
Nonetheless, over arbitrary posets, modules that are tame by virtue of
admitting a finite constant subdivision (Definition~\ref{d:tame})
always admit finite encodings (Theorem~\ref{t:tame}), although given
constant regions are typically subdivided by the constructed encoding
poset morphism.  This implication is what demands precision in the
definition of tame via constant subdivision; it makes subtle use of
the no-monodromy condition in Definition~\ref{d:constant-subdivision}.
In the case where the poset is a real vector space, if the constant
regions have additional geometry
(Definition~\ref{d:auxiliary-hypotheses}), then a similarly geometric
finite encoding is possible
(Theorem~\ref{t:tame}.\ref{i:auxiliary-uptight}).

\enlargethispage{.1ex}
\begin{remark}\label{r:filter}
Filtrations of finite simplicial complexes by products of intervals
yield persistent homology modules that are not naturally modules over
a polynomial ring in~$n$ (or any finite number of) variables.  This is
for the same reason that single-parameter persistent homology is not
naturally a module over a polynomial ring in one variable: though
there might only be finitely many topological transitions, they can
(and often do) occur at incommensurable real numbers.  That said,
filtering a finite simplicial complex automatically induces a finite
encoding.  Indeed, the parameter space maps to the poset of simplicial
subcomplexes of the original simplicial complex by sending a parameter
to the simplicial subcomplex it indexes.
\end{remark}

\begin{remark}\label{r:exit-path}
The framework of poset modules arising from filtrations of topological
spaces is more or less an instance of MacPherson's exit path category
\cite[\S1.1]{treumann2009}.  In that context, Lurie defined a notion
of constructibility in the Alexandrov topology \cite[Definitions~A.5.1
and~A.5.2]{lurie2017}, independently of the developments here and for
different purposes.  It would be reasonable to speculate that tameness
should correspond to Alexandrov constructibility, given that encoding
of a poset module is defined by pulling back along a poset morphism
(in Lurie's language, a continuous morphism of posets), but it does
not; see Remark~\ref{r:lurie}.  The difference between constant in the
sense of tameness via constant subdivision
(Section~\ref{sub:constant}) and locally constant in the
sheaf-theoretic sense with the Alexandrov topology makes tameness---in
the equivalent finitely encoded formulation---rather than Alexandrov
constructibility the right notion of finiteness for the syzygy theorem
as well as for algorithmic computation and data analysis applications
of persistent homology.  That contrasts with the comparison between
tameness and subanalytic constructibility in the usual topology on
real vector spaces, which are essentially the same notion for the
relevant sheaves; see \cite[{\rm\S}4]{strat-conical}.
\end{remark}

\subsection{Algebraic tameness:~fringe presentation}\label{sub:alg-tame}%

To compute with poset modules algebraically, in theoretical as well as
algorithmic senses, requires presentations.  When the poset
is~$\ZZ^n$, so the modules are multigraded over the usual polynomial
ring in~$n$~variables, free presentations are available.  But over
arbitrary posets, there are no free modules, and even when there are,
requiring finite presentation is unreasonably restrictive,
cf.~Example~\ref{e:toy-model-fly-wing}.  Furthermore, there is nothing
special about generators (in topological language, ``births'') as
opposed to cogenerators (``deaths'').  These issues are all resolved
by (i)~using arbitrary upsets instead the right-angled principal
upsets that give rise to free modules and (ii)~symmetrically involving
downsets.  The resulting notion of \emph{fringe presentation}
(Definition~\ref{d:fringe}) is a homomorphisms from a direct sum of
interval modules for upsets to a direct sum of interval modules for
downsets.

Fringe presentation is expressed by a \emph{monomial matrix}
(Definition~\ref{d:monomial-matrix-fr}): an array of scalars with rows
labeled by upsets and columns labeled by downsets.

\begin{example}\label{e:fringe}
Over the poset~$\RR^2$, the monomial matrix
\vspace{7ex}
$$%
\monomialmatrix
	{\begin{array}[t]{@{}r@{\hspace{-3.1pt}}|@{}l@{}}
	 \includegraphics[height=15mm]{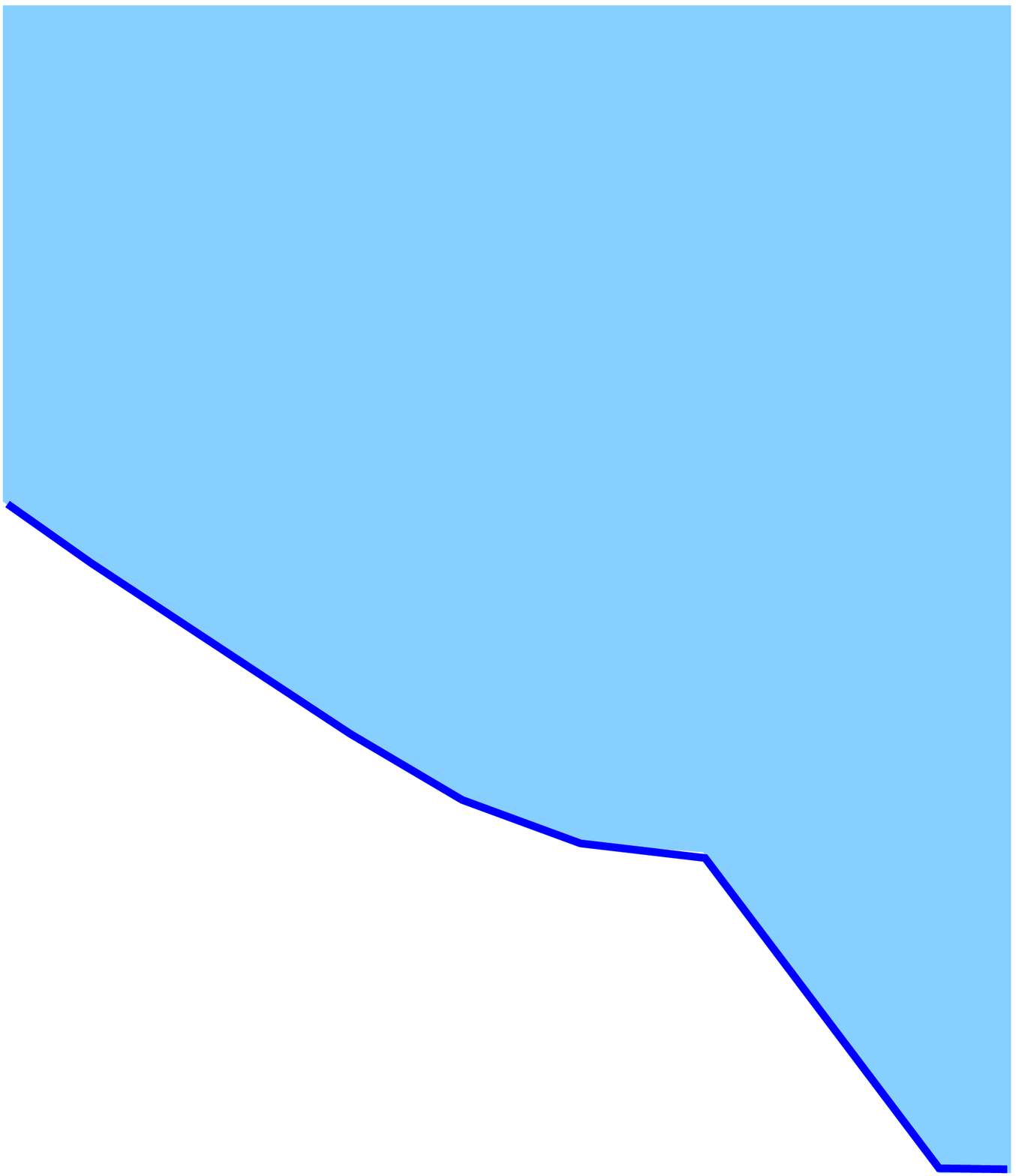}&\ \hspace{-.2pt}\\[-4.3pt]\hline
	\end{array}}
	{\!\!
	 \begin{array}{c}
	 \\[-10ex]
	 \begin{array}[b]{@{}r@{\hspace{-.2pt}}|@{}l@{}}
		\raisebox{-5mm}{\includegraphics[height=17mm]{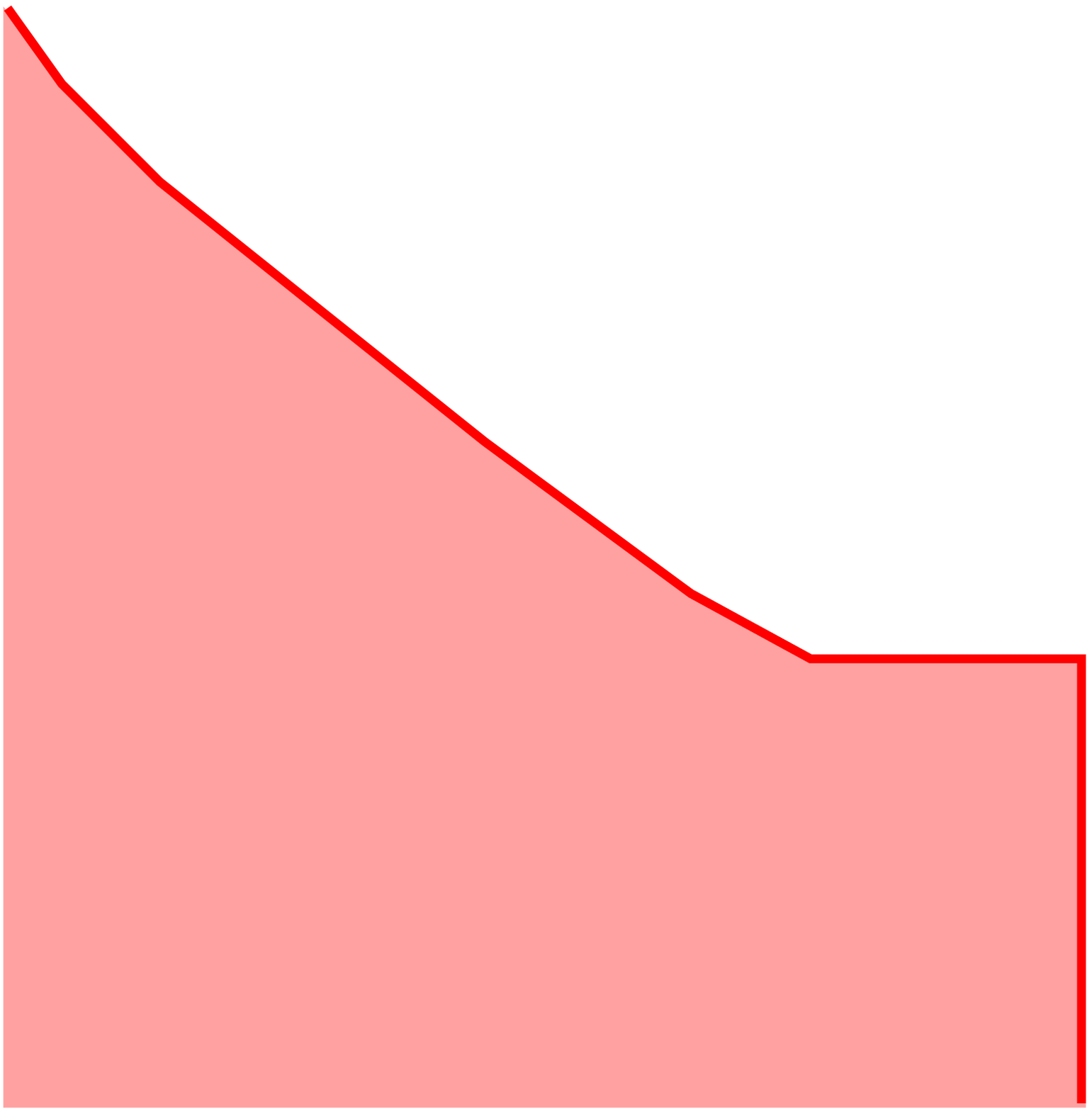}}
		&\ \,\\[-6.3pt]\hline
	 \end{array}
	 \!\!
	 \\[4ex]
	 \phi_{11}
	 \\[3ex]
	 \end{array}}
	{\\\\\\}
\begin{array}{@{}c@{}}
\hspace{.1pt}\ \text{represents a fringe presentation of}\ \hspace{.2pt}
\cM = \kk\!\!
\left[
\begin{array}{@{\ }c@{\,}}
\\[-2.2ex]
\begin{array}{@{}r@{\hspace{-.4pt}}|@{}l@{}}
\includegraphics[height=15mm]{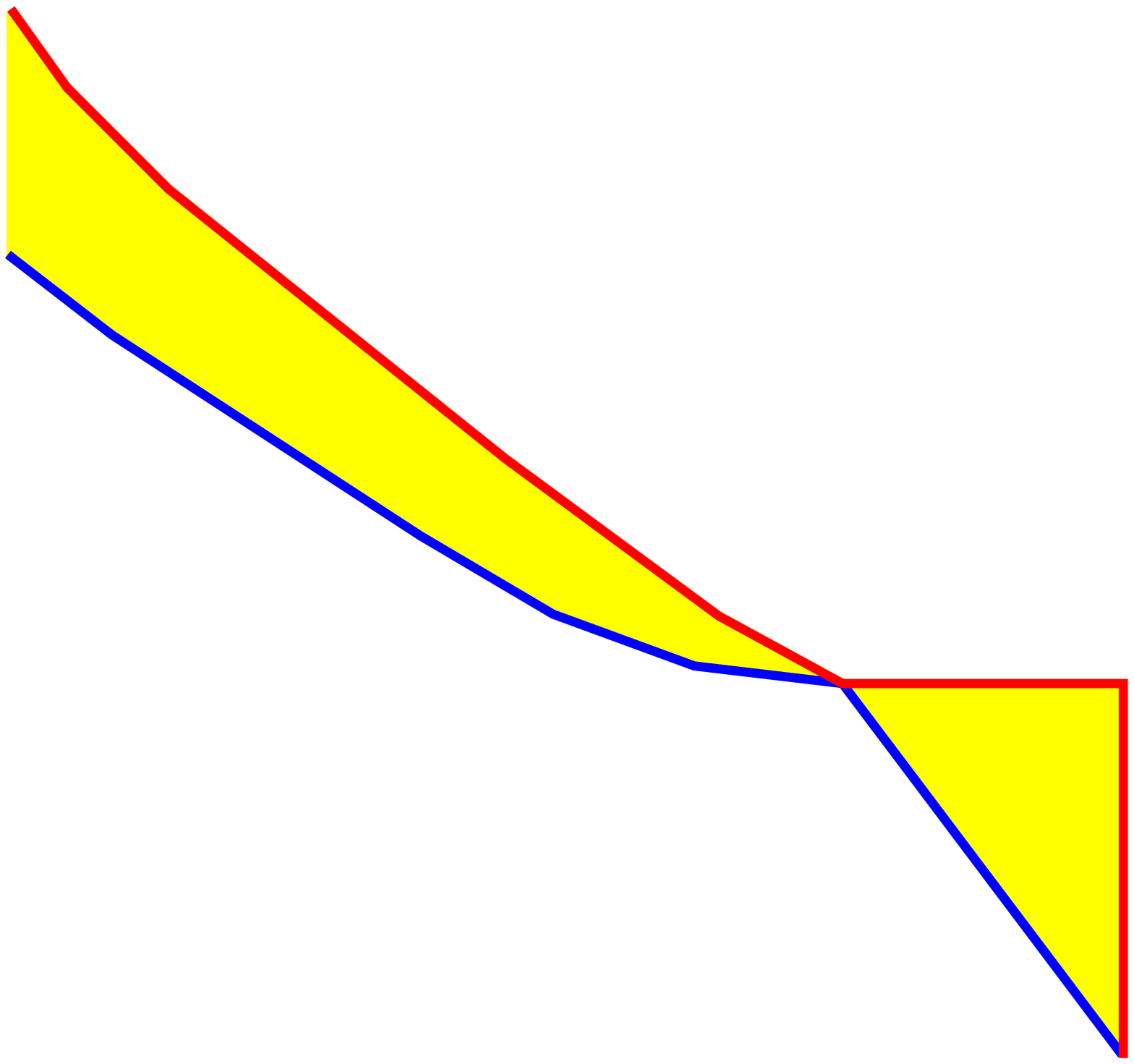}&\ \,\hspace{-.3pt}\\[-4.7pt]\hline
\end{array}
\\[-1ex]\mbox{}
\end{array}
\right]
\\[-3ex]
\end{array}
\vspace{.5ex}
$$
as long as $\phi_{11} \in \kk$ is nonzero.  That is, the monomial
matrix specifies a homomorphism \( \kk\bigl[\,
  \begin{array}{@{}c@{}}
  \\[-3ex]
  \begin{array}{@{}r@{\hspace{-1.4pt}}|@{}l@{}}
  \raisebox{-2.2pt}{\includegraphics[height=5mm]{blue-upset}}
  &\hspace{1.7pt}\\[-2pt]\hline
  \end{array}\,
  \end{array}
  \bigr]\!
\to
  \kk\bigl[\,
  \begin{array}{@{}c@{}}
  \\[-3ex]
  \begin{array}{@{}r@{\hspace{-.2pt}}|@{}l@{}}
  \raisebox{-6.2pt}{\includegraphics[height=5mm]{red-downset}}
  &\hspace{1.5pt}\\[-2pt]\hline
  \end{array}
  \end{array}
  \,\bigr]
\)
with image~$\cM$, which has $\cM_\aa = \kk$ over the yellow
parameters~$\aa$ and~$0$ elsewhere.  The blue upset specifies the
generators (births) at the lower boundary of~$\cM$; unchecked, these
persist all the way up and to the right.  But the red downset
specifies the cogenerators (deaths) along the upper boundary of~$\cM$.
This example illustrates how fringe presentations are topologically
interpretable in terms of birth and death of homology classes, when
the modules in question are persistent homology.
\end{example}

When birth upsets and death downsets are semialgebraic, or piecewise
linear, or otherwise manageable algorithmically, monomial matrices
render fringe presentations effective data structures for
multiparameter persistent homology (\emph{multipersistence}).

It is evidence for the naturality of the definitions that the
algebraic condition of admitting a finite fringe presentation is
equivalent to the topological and combinatorial notions of tameness;
this equivalence is part of the syzygy theorem
(Theorem~\ref{t:syzygy}).

Although the data structure of fringe presentation is aimed at
$\RR^n$-modules, it is new and lends insight already for finitely
generated $\NN^n$-modules (even when $n = 2$), where monomial matrices
have their origins \cite[Section~3]{alexdual}.  The context there is
more or less that of finitely determined modules; see
Definition~\ref{d:monomial-matrix-fl} in particular, which is really
just the special case of fringe presentation in which the upsets are
localizations of~$\NN^n$ and the downsets are duals---that is,
negatives---of those.

It may be helpful to understand the relaxation from free presentation
to fringe presentation step by step over $\ZZ^n$ or~$\RR^n$.  First,
tame modules can have generators that sit infinitely far back along
various axes, as in Example~\ref{e:toy-model-fly-wing}.  This issue is
solved by allowing flat modules instead of free ones.  Over~$\ZZ^n$,
for instance, this means that (multigraded translates of)
localizations of the polynomial ring by inverting variables should be
used instead of only (translates of) the ring itself; see
Remark~\ref{r:flat}.

The next relaxation concerns cogenerators (deaths) as opposed to
generators (births).  Switching these means injective hulls and
copresentations instead of flat covers and presentations.  Commutative
algebra has considered multigraded injectives for decades
\cite{goto-watanabe1978} (see \cite[Chapter~11]{cca} for an
exposition), even algorithmically \cite{irredRes,injAlg}.

The goal, however, is to place flat presentations and injective
copresentations on equal footing, so as to incorporate births and
deaths simultaneously.  These Matlis dual concepts (see
Remark~\ref{r:flat}) are combined by composing a flat cover $F \onto
\cM$ with an injective hull $\cM \into E$ to get a homomorphism $F \to
E$ whose image is~$\cM$.  This homomorphism is a \emph{flange
presentation} of~$\cM$ (Definition~\ref{d:flange}), which splices a
flat resolution to an injective one in the same way that Tate
resolutions (see \cite{coanda2003}, for example) transition from a
free resolution to an injective one over a Gorenstein local ring of
dimension~$0$.  Flange presentation is the most direct generalization
to multiple parameters of the presentation corresponding to a bar code
or persistence diagram.  The key realization is that with multiple
parameters, while births still correspond to generators, deaths
correspond to cogenerators rather than to relations
among~\mbox{generators}.

The final relaxation, from summands that are flat or injective to
arbitrary upset or downset modules, provides finite data structures
for tame modules even when they have infinte numbers of generators or
cogenerators.

\begin{example}\label{e:fringe'}
The module~$\cM$ in Example~\ref{e:fringe} is tame but has uncountably
many generators, uncountably many cogenerators, and an even worse set
of relations.  The fringe presentation in Example~\ref{e:fringe}
gathers the lower boundary points into a single {\color{blue}upset
module} and all upper boundary points into a single
{\color{red}downset module} (Definition~\ref{d:indicator}).  In
contrast, a free $\RR^n$-module of rank~$1$ has its nonzero components
on a principal upset, which has exactly one lower corner.  Thus fringe
presentation sacrifices flatness and injectivity for finiteness and
flexibility to serve over arbitrary posets.
\end{example}

\begin{remark}\label{r:approx}
Any $\RR^n$-module~$\cM$ can be approximated by a~$\ZZ^n$-module, the
result of restricting $\cM$ to, say, the rescaled
lattice~$\epsilon\ZZ^n$.  Suppose, for the sake of argument,
that~$\cM$ is bounded, in the sense of being zero at parameters
outside of a bounded subset of~$\RR^n$; think of
Example~\ref{e:toy-model-fly-wing}, ignoring those parts of the module
there that lie outside of the depicted square.
\begin{figure}[h]
\vspace{-2ex}
$$%
\begin{array}[b]{@{\hspace{-10pt}}r@{\hspace{-10pt}}|@{}l@{}}
\begin{array}{@{}c@{}}
\psfrag{r}{}
\psfrag{s}{}
\includegraphics[height=30mm]{toy-model}\\[-7.7pt]
\end{array}
&\,\,\,\\[-6pt]\hline
\end{array}
\quad\
\begin{array}[b]{@{\ \ }c@{}}
\hspace{-3pt}\goesto\\[7mm]\mbox{}
\end{array}
\qquad
\begin{array}[b]{@{\hspace{-10pt}}r@{\hspace{-10pt}}|@{}l@{}}
\begin{array}{@{}c@{}}
\includegraphics[height=30mm]{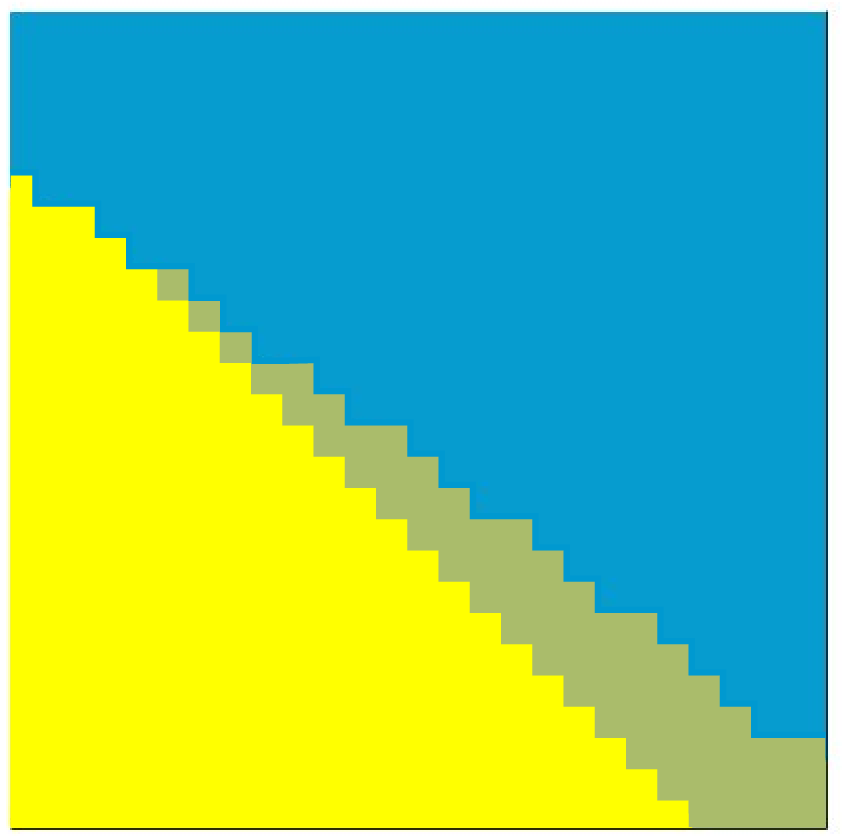}\\[-7.7pt]
\end{array}
&\,\,\,\\[-6pt]\hline
\end{array}
$$\vspace{-4ex}
\end{figure}
Ever better approximations, by smaller~$\epsilon \to 0$, yield sets of
lattice points ever more closely hugging an algebraic curve.
Neglecting the difficulty of computing where those lattice points lie,
how is a computer to store or manipulate such a set?  Listing the
points individually is an option, and perhaps efficient for
particularly coarse approximations, but in~$n$ parameters the
dimension of this storage problem is~$n-1$.  As the approximations
improve, the most efficient way to record such sets of points is
surely to describe them as the allowable ones on one side of the given
algebraic curve.  And once the computer has the curve in memory, no
approximation is required: just use the (points on the) curve itself.
In this way, even in cases of multipersistence where the entire
topological filtration setup can be approximated by finite simplicial
complexes, understanding the continuous nature of the un-approximated
setup can be at the same time more transparent and more~efficient.
\end{remark}

\begin{remark}\label{r:history}
$\ZZ^n$-graded commutative algebra is decades old
\cite{goto-watanabe1978}, but the perspective arising from their
equivalence with multipersistence is relatively new
\cite{multiparamPH}.  Initial steps have included descriptions of the
set of isomorphism classes \cite{multiparamPH}, presentations
\cite{csv17} and algorithms for computing \cite{computMultiPH,csv12}
or visualizing \cite{lesnick-wright2015} them, as well as interactions
with homological algebra of modules, such as persistence invariants
\cite{knudson2008} and certain notions of multiparameter noise
\cite{scolamiero-chacholski-lundman-ramanujam-oberg16}.  That
multipersistence modules can fail to be finitely generated
(cf.~Example~\ref{e:toy-model-fly-wing}) in situations reflecting
reasonably realistic data analysis was observed by Patriarca,
Scolamiero, and Vaccarino
\cite[Section~2]{patriarca-scolamiero-vaccarino2012}.  Their algorithm
for discrete parameters keeps track of generators not individually but
gathered together as generators of monomial ideals.  Huge numbers of
predictable syzygies among generators are swallowed and hence are
present only implicitly.  And that is good, as nothing topologically
new about persistence of homology classes is taught by the well known
syzygies of monomial ideals (see \cite{sylvan-resolution}, for
example), which in this setting are merely an interference pattern
from the merging of separate birth points of the same class.
\end{remark}

\subsection{Homological tameness: the syzygy theorem}\label{sub:homalg}

Just as upsets and downsets can be used to present poset modules, they
can be used to resolve them.  As in polynomial algebra, this line of
thinking culminates in a syzygy theorem (Theorem~\ref{t:syzygy} for
modules; Theorem~\ref{t:syzygy-complexes} for complexes) to the effect
that, remarkably, the topological, algebraic, combinatorial, and
homological notions of tameness available respectively~via
\begin{itemize}
\item%
constant subdivision (Definition~\ref{d:tame}),
\item%
fringe presentation (Definition~\ref{d:fringe}),
\item%
poset encoding (Definition~\ref{d:encoding}), and
\item%
indicator resolution (Definition~\ref{d:resolutions})
\end{itemize}
are equivalent.  The moral is that the tame condition over arbitrary
posets appears to be the right notion to stand in lieu of the
noetherian hypothesis over~$\ZZ^n$: the tame condition is robust, has
multiple characterizations from different mathematical perspectives,
and enables algorithmic computation in principle.  The syzygy theorem
is the main take-away from the paper.  It engulfs the statements of
its stepping stones, most notably Theorems~\ref{t:constant-uptight}
and~\ref{t:tame}, whose proofs isolate crucial ideas.

The syzygy theorem directly reflects the more usual syzygy theorem for
finitely determined $\ZZ^n$-modules
(Theorem~\ref{t:finitely-determined}), with upset and downset
resolutions being the arbitrary-poset analogues of free and injective
resolutions, respectively, and fringe presentation being the
arbitrary-poset analogue of flange presentation.  Indeed, the proof of
the syzygy theorem works by reducing to the finitely determined case
(Section~\ref{s:ZZn}) over~$\ZZ^n$.  The main point is that given a
finite encoding of a module over an arbitrary poset~$\cQ$, the
encoding poset can be embedded in~$\ZZ^n$.  The proof is completed by
pushing the data forward to~$\ZZ^n$, applying the more usual syzygy
theorem to finitely determined modules there, and pulling back
to~$\cQ$.

It bears mentioning that even if one is interested in ring-theoretic
situations where the poset is $\ZZ^n$ or~$\RR^n$, one can and should
do homological algebra of tame modules over a finite encoding poset
rather than (only) over the original parameter space.

\begin{remark}\label{r:weaker}
Topological tameness via constant subdivision is a~priori weaker (that
is, more inclusive) than combinatorial tameness via finite encoding,
and algebraic tameness via fringe presentation is a~priori weaker than
homological tameness via upset or downset resolution.  Thus the syzygy
theorem leverages relatively weak topological structure into powerful
homological structure.  In particular, it provides concrete,
computable, combinatorially describable representatives for objects in
the derived category.  The proof \cite{strat-conical} of two
conjectures due to Kashiwara and Schapira
(\cite[Conjecture~3.17]{kashiwara-schapira2017}
and~\cite[Conjecture~3.20]{kashiwara-schapira2019}) relies on the fact
that, although the tameness characterizations require no additional
structure on the underlying poset, any additional structure that is
present---subanalytic, semialgebraic, or piecewise-linear---is
preserved by the transitions among tameness characterizations in the
syzygy~\mbox{theorem}.
\end{remark}

\subsection{Bar codes and further developments}\label{sub:bar-codes}

Tame modules over the totally ordered set of integers or real numbers
are, up to isomorphism, the same as ``bar codes'': finite multisets of
intervals.  The most general form of this bijection between algebraic
objects and essentially combinatorial objects over totally ordered
sets is due to Crawley-Boevey \cite{crawley-boevey2015}.  At its root
this bijection is a manifestation of the tame representation theory of
the type~$A$ quiver; that is the context in which bar codes were
invented by Abeasis and Del Fra, who called them ``diagrams of boxes''
\cite{abeasis-delFra1980,abeasis-delFra-kraft1981}.  Subsequent
terminology for objects equivalent to these diagrams of boxes include
bar codes themselves (see \cite{ghrist2008}) and planar depictions
discovered effectively simultaneously in topological data analysis,
where they are called persistence diagrams
\cite{edelsbrunner-letscher-zomorodian2002} (see
\cite{cohenSteiner-edelsbrunner-harer2007} for attribution) and
combinatorial algebraic geometry, where they are
called~lace~\mbox{arrays}~\cite{quivers}.

No combinatorial analogue of the bar code can classify modules over an
arbitrary poset because there are too many indecomposable modules,
even over seemingly well behaved posets like~$\ZZ^n$
\cite{multiparamPH}: the indecomposables come in families of positive
dimension.  Over arbitrary posets, every tame module does still admit
a decomposition of the Krull--Schmidt sort, namely as a direct sum of
indecomposables \cite{botnan-crawley-boevey2018}, but again, there are
too many indecomposables for this to be useful in general.  Instead of
decomposing modules as direct sums of elemental pieces, which be
arbitrarily complicated \cite{buchet-escolar2020}, the commutative
algebra view advocates expressing poset modules in terms of intervals,
especially indicator modules for upsets and downsets, by way of less
rigid constructions like fringe presentation (Section~\ref{s:fringe}),
primary decomposition \cite{prim-decomp-pogroup, primary-distance}, or
resolution (Section~\ref{s:syzygy}).  This relaxes the direct sum in a
$K$-theoretic way, allowing arbitrary complexes instead of split short
exact sequences.

Various aspects of bar codes are reflected in the equivalent concepts
of tameness.  The finitely many regions of constancy are seen in
topological tameness by constant subdivision.  The matching between
left and right endpoints is seen in algebraic tameness by fringe
presentation, where the left endpoints form lower borders of birth
upsets and the right endpoints form upper borders of death downsets.
The expressions of modules in terms of bars is seen, in its relaxed
form, in homological tameness, where modules become ``virtual'' sums,
in the sense of being formal alternating combinations rather than
direct sums of intervals.  Primary decomposition
\cite{prim-decomp-pogroup} isolates elements that would, in a bar
code, lie in bars unbounded in fixed sets of directions (see
also~\cite{harrington-otter-schenck-tillmann2019}).

Bar codes rely on some concept of minimality: left endpoints must
correspond to \emph{minimal} generators, and right endpoints to
\emph{minimal} cogenerators.  These are not available over arbitrary
posets and are subtle to define and handle properly even for partially
ordered real vector spaces \cite{essential-real}.  When minimality is
available, instead of a bijection (perfect matching) from a multiset
of births to a multiset of deaths, the best one can settle for is a
linear map from a filtration of the birth multiset to a filtration of
the death multiset \cite{functorial-multiPH}.

\section{Tame poset modules}\label{s:tame}

\subsection{Modules over posets}\label{sub:persistence}

\begin{defn}\label{d:poset-module}
Let $\cQ$ be a partially ordered set (\emph{poset}) and~$\preceq$ its
partial order.  A \emph{module over~$\cQ$} (or a \emph{$\cQ$-module})
is
\begin{itemize}
\item%
a $Q$-graded vector space $\cM = \bigoplus_{q\in Q} \cM_q$ with
\item%
a homomorphism $\cM_q \to \cM_{q'}$ whenever $q \preceq q'$ in~$Q$
such that
\item%
$\cM_q \to \cM_{q''}$ equals the composite $\cM_q \to \cM_{q'} \to
\cM_{q''}$ whenever $q \preceq q' \preceq q''$.
\end{itemize}
A \emph{homomorphism} $\cM \to \cN$ of $\cQ$-modules is a
degree-preserving linear map, or equivalently a collection of vector
space homomorphisms $\cM_q \to \cN_q$, that commute with the structure
homomorphisms $\cM_q \to \cM_{q'}$ and $\cN_q \to \cN_{q'}$.
\end{defn}

The last bulleted item is \emph{commutativity}.  In important
instances (e.g.,~\mbox{Example}~\ref{e:multifiltration}), it reflects
that inclusions of topological subspaces induce functorial maps on
homology.

\begin{example}\label{e:RR^n-graded}
A module over the poset $\RR^n$ whose partial order is componentwise
comparison is the same thing as an $\RR^n$-graded module over the
monoid algebra $\kk[\RR^n_+]$, where $\RR_+ = \{r \in \RR \mid r \geq
0\}$ is the additive monoid of nonnegative real numbers.  (This is
immediate from the definitions, see
\cite[\S2.1]{lesnick-interleav2015}, for instance.)  This case
generalizes that of $\ZZ^n$-modules, which are $\ZZ^n$-graded modules
over polynomial rings~$\kk[\NN^n]$:
elements of $\kk[\RR^n_+]$ are polynomials with real exponents.
\end{example}

\begin{example}\label{e:multifiltration}
Let $X$ be a topological space and $\cQ$ a poset.
\begin{enumerate}
\item\label{i:filtration}%
A \emph{filtration of~$X$ indexed by~$\cQ$} is a choice of subspace
$X_q \subseteq X$ for each $q \in \cQ$ such that $X_q \subseteq
X_{q'}$ whenever $q \preceq q'$.

\item\label{i:PH}%
The \emph{$i^\mathrm{th}$ persistent homology} of the \emph{filtered
space}~$X$ is the associated homology module, meaning the $\cQ$-module
$\bigoplus_{q \in \cQ} H_i X_q$.
\end{enumerate}
\end{example}

\begin{remark}\label{r:curry}
There are a number of abstract, equivalent ways to phrase
Example~\ref{e:multifiltration}.  For example, a filtration is a
functor from $\cQ$ to the category $\mathcal{S}$ of subspaces of~$X$
or an $\mathcal{S}$-valued sheaf on~$\cQ$ with its \emph{Alexandrov
topology}, whose base is the set of principal upsets (dual
order~ideals with unique minimal element).  For background on and
applications of many of these perspectives, see Curry's dissertation
\cite{curry-thesis}, particularly \S4.2 there.  See also
\cite{strat-conical}, which details the transition from modules to
constructible~sheaves.
\end{remark}

\begin{example}\label{e:RR+}
A \emph{real multifiltration} of~$X$ is a filtration indexed
by~$\RR^n$, with its partial order by coordinatewise comparison.
Example~\ref{e:fly-wing-filtration} is a real multifiltration of $X =
\RR^2$ with $n = 2$.  The persistent homology of a real $n$-filtered
space~$X$ is a \emph{multipersistence module}, which is an
$\RR^n$-module.
\end{example}

\subsection{Constant subdivisions}\label{sub:constant}

\begin{defn}\label{d:constant-subdivision}
Fix a $\cQ$-module~$\cM$.  A \emph{constant subdivision} of~$\cQ$
\emph{subordinate} to~$\cM$ is a partition of~$\cQ$ into
\emph{constant regions} such that for each constant region~$I$ there
is a single vector space~$M_I$ with an isomorphism $M_I \to M_\ii$ for
all $\ii \in I$ that \emph{has no monodromy}: if $J$ is some (perhaps
different) constant region, then all comparable pairs $\ii \preceq
\jj$ with $\ii \in I$ and $\jj \in J$ induce the same composite
homomorphism $M_I \to M_\ii \to M_\jj \to M_J$.
\end{defn}

\begin{example}\label{e:puuska-nonconstant-isotypic}
Consider the poset module (kindly provided by Ville Puuska
\cite{puuska18})
$$%
\psfrag{1}{\tiny$1$}
\psfrag{2}{\tiny$2$}
\psfrag{kk}{\tiny$\kk$}
\begin{array}{@{}c@{}}\includegraphics[height=30mm]{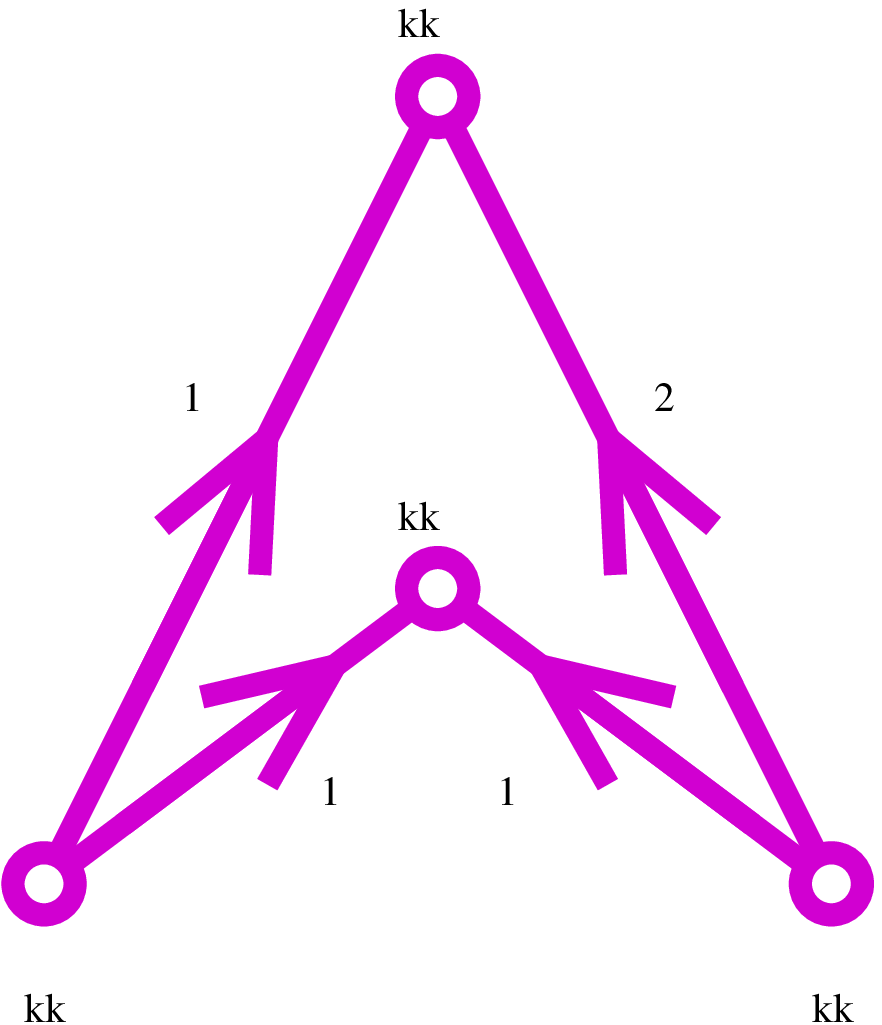}\end{array}
$$
in which the structure morphisms $M_\aa \to M_\bb$ are all identity
maps on~$\kk$, except for the rightmost one.  This example
demonstrates that module structures need not be recoverable from their
\emph{isotypic subdivision}, in which elements of~$\cQ$ lie in the
same region when their vector spaces are isomorphic via a poset
relation.  In cases like this, refining the isotypic subdivision
appropriately yields a constant subdivision.  Here, the two minimum
elements must lie in different constant regions
and the two maximum elements must lie in different constant regions.
Any partition accomplishing these separations---that is, any
refinement of a partition that has a region consisting of precisely
one maximum and one minimum---is a constant subdivision.  Of course, a
finite poset always admits a constant subdivision with finitely many
regions, since~the~\mbox{partition}~into~\mbox{singletons}~works.
\end{example}

\begin{example}\label{e:antidiagonal}
Constant subdivisions need not refine the isotypic subdivision in
Example~\ref{e:puuska-nonconstant-isotypic}, one reason being that a
single constant region can contain two or more incomparable isotypic
regions.  For a concrete instance with a single constant region
comprised of uncountably many incomparable isotypic regions, let $\cM$
be the $\RR^2$-module that has $\cM_\aa = 0$ for all $\aa \in \RR^2$
except for those on the antidiagonal line spanned by $\left[\twoline
1{-1}\right] \in \RR^2$, where $\cM_\aa = \kk$.  There is only one
such $\RR^2$-module because all of the degrees of nonzero graded
pieces of~$\cM$ are incomparable, so all of the structure
homomorphisms $\cM_\aa \to \cM_\bb$ with $\aa \neq \bb$ are zero.
Every point on the line is a singleton isotypic region.  This
conclusion reverses entirely when the line is thickened to a strip of
positive width, where the single isotypic region comprising the
support yields a constant subdivision.
\end{example}

The direction of the line in Example~\ref{e:antidiagonal} is
important: an antidiagonal line, whose points form an antichain
in~$\RR^2$, behaves radically differently than diagonal lines.

\begin{example}\label{e:diagonal}
Let $\cM$ be an $\RR^2$-module with $\cM_\aa = \kk$ whenever $\aa$
lies in the closed diagonal strip between the lines of slope~$1$
passing through any pair of points.  The structure homomorphisms
$\cM_\aa \to \cM_\bb$ could all be zero, for instance, or some of them
could be nonzero.  But the length $|\aa - \bb|$ of any nonzero such
homomorphism must in any case be bounded above by the Manhattan (i.e.,
$\ell^\infty$) distance between the two points, since every longer
structure homomorphism factors through a sequence that exits and
re-enters the strip.
$$%
\psfrag{0}{\tiny$0$}
\begin{array}{@{}c@{}}\includegraphics[height=45mm]{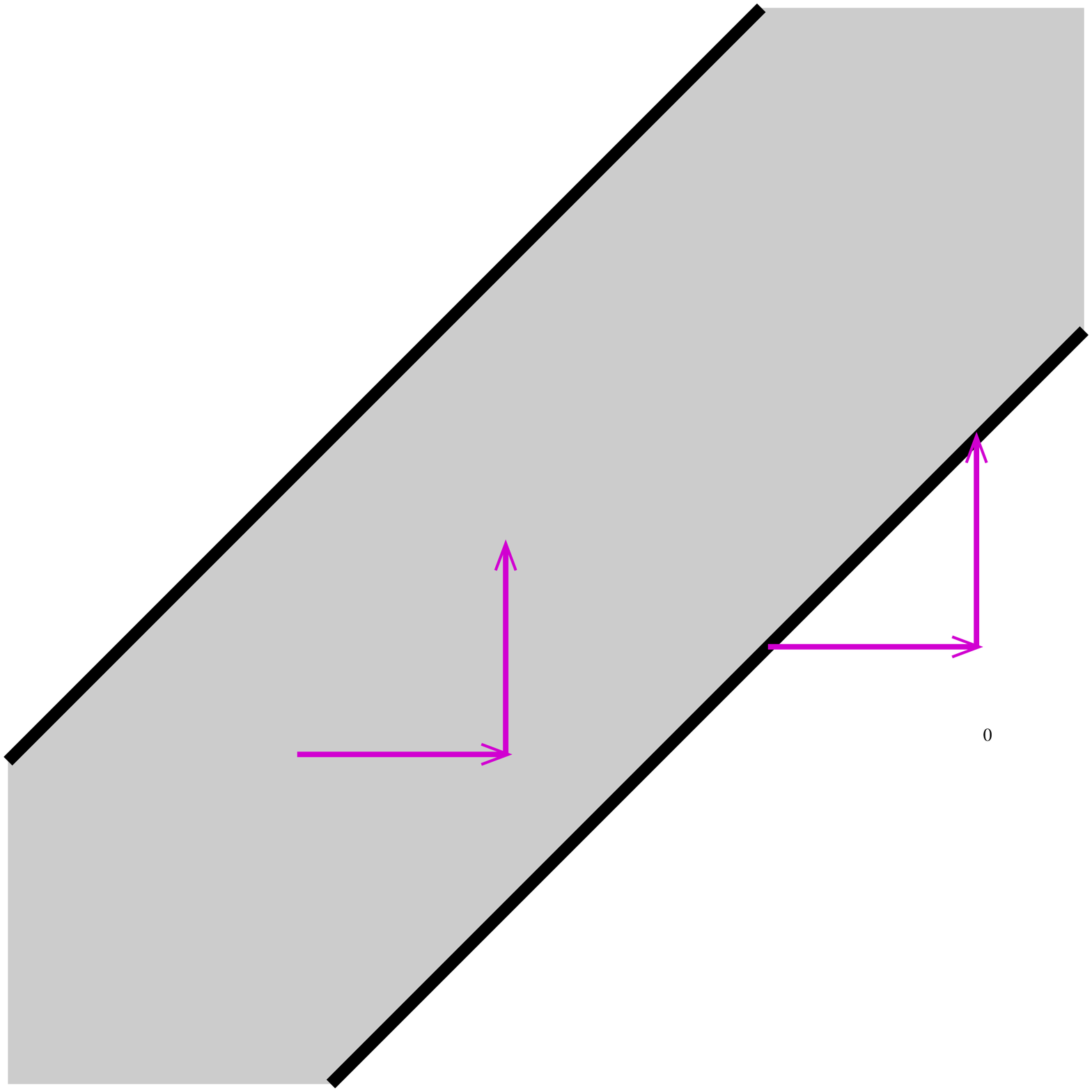}\end{array}
$$
In particular, the structure homomorphism between any pair of points
on the upper boundary line of the strip is zero because it factors
through a homomorphism that points upward first; therefore such pairs
of points lie in distinct regions of any constant subdivision.  The
same conclusion holds for pairs of points on the lower boundary line
of the strip.  When the strip has width~$0$, so the upper and lower
boundary coincide, the module is supported along a diagonal line whose
uncountably many points must all lie in distinct constant regions.
\end{example}

The reference to ``no monodromy'' in
Definition~\ref{d:constant-subdivision} agrees with the usual notion.

\begin{lemma}\label{l:no-monodromy}
Fix a constant region~$I$ subordinate to a poset module~$\cM$.  The
composite isomorphism $M_I \to M_\ii \to \dots \to M_{\ii'} \to M_I$
is independent of the path from~$\ii$ to~$\ii'$ through~$I$, if one
exists.  In particular, when $\ii = \ii'$ the composite is the
identity on~$M_I$.
\end{lemma}
\begin{proof}
The second claim follows from the first.  When the path has
length~$0$, the claim is that $M_I \to M_\ii \to M_I$ is the identity
on~$M_I$, which follows by definition.  For longer paths the result is
proved by induction on path length.
\end{proof}

Constant subdivision is the subtle part of the central finiteness
concept of the paper.

\begin{defn}\label{d:tame}
Fix a poset~$\cQ$ and a $\cQ$-module~$M$.
\mbox{}
\begin{enumerate}
\item\label{i:finite-constant-subdiv}%
A constant subdivision of~$\cQ$ is \emph{finite} if it has finitely
many constant regions.

\item\label{i:Q-finite}%
The $\cQ$-module~$M$ is \emph{$\cQ$-finite} if its components $M_q$
have finite dimension over~$\kk$.

\item\label{i:tame}%
The $\cQ$-module~$M$ is \emph{tame} if it is $\cQ$-finite and $\cQ$
admits a finite constant subdivision subordinate to~$M$.
\end{enumerate}
\end{defn}

\begin{remark}\label{r:tame}
\mbox{}
\begin{enumerate}
\item\label{i:ordinary-tame}%
In ordinary totally ordered persistent homology, tameness means simply
that the bar code (see Section~\ref{sub:bar-codes}) has finitely many
\pagebreak
bars, or equivalently, the persistence diagram has finitely many
off-diagonal dots: finiteness of the set of constant regions precludes
infinitely many non-overlapping bars (the bar code can't be ``too
long''), while the vector space having finite dimension precludes a
parameter value over which lie infinitely many bars (the bar code
can't be ``too wide'').
\item%
The tameness condition here includes but is much less rigid than the
compact tameness condition in
\cite{scolamiero-chacholski-lundman-ramanujam-oberg16}, the latter
meaning more or less that the module is finitely generated over a
scalar multiple of~$\ZZ^n$ in~$\QQ^n$.

\item%
Some literature calls Definition~\ref{d:tame}.\ref{i:Q-finite}
\emph{pointwise finite dimensional (PFD)}.  The terminology here
agrees with that in \cite{alexdual}, on which Section~\ref{s:ZZn} here
is based.
\end{enumerate}
\end{remark}

\begin{excise}{%
  \begin{remark}\label{r:isotypic-expected}
  Data analysis should always produce tame persistence modules.  Indeed,
  from limited experience, in data analysis the isotypic subdivision
  (Example~\ref{e:puuska-nonconstant-isotypic}) appears consistently to
  be a finite constant subdivision.  It is unclear what kind of data
  situation might produce a nonconstant isotypic subdivision, but it
  seems likely that in such cases a constant subdivision can always be
  obtained by subdividing---into contractible pieces---isotypic
  components that have nontrivial topology and then gathering
  incomparable contractible pieces into single constant regions.
  Exploring these assertions is left open.
  \end{remark}
}\end{excise}%

\begin{lemma}\label{l:subdiv}
Any refinement of a constant subdivision subordinate to a
$\cQ$-module~$\cM$ is a constant subdivision subordinate to~$\cM$.
\end{lemma}
\begin{proof}
Choosing the same vector space~$M_I$ for every region of the
refinement contained in the constant region~$I$, the lemma is
immediate from Definition~\ref{d:constant-subdivision}.
\end{proof}

\subsection{Auxiliary hypotheses}\label{sub:auxiliary}\mbox{}

\noindent
Effectively computing with real multifiltered spaces
(Example~\ref{e:RR+}) requires keeping track of the shapes of various
regions, such as constant regions.  (In later sections, other regions
along these lines include upsets, downsets, and fibers of poset
morphisms.)  The fact that applications of persistent homology often
arise from metric considerations, which are semialgebraic in nature,
or are approximated by piecewise linear structures suggests the
following auxiliary hypotheses for algorithmic developments.  The
subanalytic hypothesis is singled out for theoretical purposes
surrounding conjectures of Kashiwara and Schapira
(\cite[Conjecture~3.17]{kashiwara-schapira2017},
\cite[Conjecture~3.20]{kashiwara-schapira2019};~cf.~\cite{strat-conical}).

\begin{defn}\label{d:pogroup}
An abelian group~$Q$ is \emph{partially ordered} if it is generated
by a submonoid~$Q_+$, called the \emph{positive cone}, that has
trivial unit group.  The partial order is: $q \preceq q' \iff q' - q
\in Q_+$.  A partially ordered group is
\begin{enumerate}
\item\label{i:real}%
\emph{real} if the underlying abelian group is a real vector space of
finite dimension;
\item\label{i:discrete}%
\emph{discrete} if the underlying abelian group is free of finite
rank.
\end{enumerate}
\end{defn}

\begin{defn}\label{d:auxiliary-hypotheses}
Fix a subposet~$\cQ$ of a real partially ordered group.  A partition
of~$\cQ$ into subsets~is
\begin{enumerate}
\item\label{i:semialgebraic}%
\emph{semialgebraic} if the subsets are real semialgebraic varieties;

\item\label{i:PL}%
\emph{piecewise linear (PL)} if the subsets are finite unions of
convex polyhedra, where a \emph{convex polyhedron} is an intersection
of finitely many closed or open half-spaces;

\item\label{i:subanalytic}%
\emph{subanalytic} if the subsets are subanalytic varieties;

\item\label{i:classX}%
\emph{of class~$\mathfrak X$} if the subsets lie in a
family~$\mathfrak X$ of subsets of~$\cQ$ closed under complement,
finite intersection, negation, and Minkowski sum with the
positive~cone~$\cQ_+$.
\end{enumerate}
A module over~$\cQ$ is \emph{semialgebraic}, or \emph{PL},
\emph{subanalytic}, or \emph{of class~$\mathfrak X$} if $\cQ_+$ is and
the module is tamed by a subordinate finite constant subdivision of
the corresponding~type.
\end{defn}

\begin{remark}\label{r:auxiliary-hypotheses}
Subposets of real partially ordered groups are allowed in
Definition~\ref{d:auxiliary-hypotheses} to be able to speak of, for
example, piecewise linear sets in rational vector spaces, or
semialgebraic subsets of~$\ZZ^n$, such as the set of lattice points in
a right circular cone (e.g.~\cite[Example~5.9]{prim-decomp-pogroup}).
When $\cQ$ is properly contained in the ambient real vector space,
subsets of~$\cQ$ are semialgebraic, PL, or subanalytic when they are
intersections with~$\cQ$ of the corresponding type of subset of the
ambient real vector space.
\end{remark}

\begin{prop}\label{p:auxiliary-hypotheses}
Fix a partially ordered real vector space~$\cQ$.
\begin{enumerate}
\item\label{i:classes}%
The classes of semialgebraic, PL, and subanalytic subsets of~$\cQ$ are
each closed under complements, finite intersections, and negation.

\item\label{i:sum-semialg}%
The Minkowski sum $S + \cQ_+$ of a semialgebraic set~$S$ with the
positive cone is semialgebraic if~$\cQ_+$ is semialgebraic.

\item\label{i:sum-PL}%
The Minkowski sum $S + \cQ_+$ of a PL set with the positive cone is
semialgebraic if~$\cQ_+$ is polyhedral.

\item\label{i:sum-subanalytic}%
The Minkowski sum $S + \cQ_+$ of a bounded subanalytic set~$S$ with
the positive cone is subanalytic if~$\cQ_+$ is subanalytic.
\end{enumerate}
\end{prop}
\begin{proof}
See \cite{shiota97} (for example) to treat the semialgebraic and
subanalytic cases of item~\ref{i:classes}.  The PL case reduces easily
to checking that the complement of a single polyhedron is PL, which in
turn follows because a real vector space is the union of the
(relatively open) faces in any finite hyperplane arrangement, so
removing a single one of these faces leaves a PL set remaining.

For item~\ref{i:sum-semialg}, use that the image of a semialgebraic
set under linear projection is a semialgebraic set, and then express
$S + \cQ_+$ as the image of $S \times \cQ_+$ under the projection $\cQ
\times \cQ \to \cQ$ that acts by $(\qq,\qq') \mapsto \qq + \qq'$.  The
same argument works for item~\ref{i:sum-PL}.  The same argument also
works for item~\ref{i:sum-subanalytic} but requires that the
restriction of the projection to the closure of $S \times \cQ_+$ be a
proper map, which always occurs when $S$ is bounded.
\end{proof}

\section{Fringe presentation by upsets and downsets}\label{s:fringe}

To define the concept of fringe presentation precisely requires some
elementary background on posets.  That includes upsets and downsets
and the modules constructed from them (Definition~\ref{d:indicator}).
Less obviously, notions of connectedness
(Definition~\ref{d:connected-poset}) play a key role, especially in
computing vector spaces of homomorphisms between upset and downset
modules (Proposition~\ref{p:U->D}).  Situations where these Hom sets
have dimension~$1$ (Corollary~\ref{c:U->D}) are particularly key,
leading to the notion of connected homomorphisms of interval modules
(Definition~\ref{d:connected-homomorphism}).  In general, the basic
poset material in Section~\ref{sub:indicator} should be useful as a
reference more widely than for the application to fringe presentation
here.  Section~\ref{sub:fringe} goes on to introduce fringe
presentation (Definition~\ref{d:fringe}) and monomial matrix
(Definition~\ref{d:monomial-matrix-fr}), along with some simple
examples.

\subsection{Upsets and downsets}\label{sub:indicator}

\begin{defn}\label{d:indicator}
The vector space $\kk[\cQ] = \bigoplus_{q\in\cQ} \kk$ that assigns
$\kk$ to every point of the poset~$\cQ$ is a $\cQ$-module with
identity maps on~$\kk$.  More generally,
\begin{enumerate}
\item\label{i:upset}%
an \emph{upset} (also called a \emph{dual order ideal}) $U \subseteq
\cQ$, meaning a subset closed under going upward in~$\cQ$ (so $U +
\RR_+^n = U$, when $\cQ = \RR^n$) determines an \emph{indicator
submodule} or \emph{upset module} $\kk[U] \subseteq \kk[\cQ]$; and
\item\label{i:downset}%
dually, a \emph{downset} (also called an \emph{order ideal}) $D
\subseteq \cQ$, meaning a subset closed under going downward in~$\cQ$
(so $D - \RR_+^n = D$, when $\cQ = \RR^n$) determines an
\emph{indicator quotient module} or \emph{downset module} $\kk[\cQ]
\onto \kk[D]$.
\end{enumerate}
When $\cQ$ is a subposet of a real partially ordered group
(Definition~\ref{d:pogroup}), an indicator module of either sort is
semialgebraic, PL, subanalytic, or of class~$\mathfrak X$ if the
corresponding upset or downset is of the same type
(Definition~\ref{d:auxiliary-hypotheses}).
\end{defn}

\begin{remark}\label{r:indicator}
Indicator submodules $\kk[U]$ and quotient modules $\kk[D]$ are
$\cQ$-modules, not merely $U$-modules or $D$-modules, by setting the
graded components indexed by elements outside of the ideals to~$0$.
It is only by viewing indicator modules as $\cQ$-modules that they are
forced to be submodules or quotients, respectively.  For relations
between these notions and those in Remark~\ref{r:curry}, again see
Curry's thesis \cite{curry-thesis}.  For example, upsets form the open
sets in the topology from Remark~\ref{r:curry}.
\end{remark}

\begin{example}\label{e:melting}
Ising crystals at zero temperature, with polygonal boundary conditions
and fixed mesh size, are semialgebraic upsets in~$\RR^n$.  That much
is by definition: fixing a mesh size means that the crystals in
question are (staircase surfaces of finitely generated) monomial
ideals in $n$ variables.  Remarkably, such crystals remain
semialgebraic in the limit of zero mesh size; see \cite{okounkov16}
for an exposition and references.
\end{example}

\begin{example}\label{e:asw}
Monomial ideals in polynomial rings with real exponents, which
correspond to upsets in $\RR^n_+$, are considered in
\cite{andersen--sather-wagstaff2015}, including aspects of primality,
irreducible decomposition, and Krull dimension.  Upsets in $\RR^n$ are
also considered in \cite{madden-mcguire2015}, where the combinatorics
of their lower boundaries, and topology of related simplicial
complexes, are investigated in cases with locally finite generating
sets.
\end{example}

\begin{defn}\label{d:connected-poset}
A poset~$\cQ$ is
\begin{enumerate}
\item\label{i:connected}%
\emph{connected} if every pair of elements $q,q' \in \cQ$ is joined by
a \emph{path} in~$\cQ$: a sequence $q = q_0 \preceq q'_0 \succeq q_1
\preceq q'_1 \succeq \dots \succeq q_k \preceq q'_k = q'$ in~$\cQ$;

\item%
\emph{upper-connected} if every pair of elements in~$\cQ$ has an
upper bound in~$\cQ$;

\item%
\emph{lower-connected} if every pair of elements in~$\cQ$ has a
lower bound in~$\cQ$; and

\item%
\emph{strongly connected} if $\cQ$ is upper-connected and
lower-connected.
\end{enumerate}
\end{defn}

\begin{example}\label{e:connected-poset}
$\RR^n$ is strongly connected.  The same is true of any partially
ordered abelian group.  (See \cite{prim-decomp-pogroup} for additional
basic theory of those posets.)
\end{example}

\begin{example}\label{e:bounded-poset}
A poset~$\cQ$ is upper-connected if (but not only if,
cf.~Example~\ref{e:connected-poset}) it has a maximum element---one
that is preceded by every element of~$\cQ$.  Similarly, $\cQ$ is
lower-connected if it has a minimum element---one that precedes every
element of~$\cQ$.
\end{example}

\begin{remark}\label{r:pi0}
The relation $q \sim q'$ defined by the existence of a path joining
$q$ to~$q'$ as in Definition~\ref{d:connected-poset}.\ref{i:connected}
is an equivalence relation.
\end{remark}

\begin{defn}\label{d:pi0}
Fix a poset~$\cQ$.  For any subset $S \subseteq \cQ$, write $\pi_0 S$
for the set of connected components of~$S$: the maximal connected
subsets of~$S$, or equivalently the classes under the relation from
Remark~\ref{r:pi0}.
\end{defn}

\begin{prop}\label{p:U->D}
Fix a poset~$\cQ$.
\begin{enumerate}
\item\label{i:U->D}%
For an upset~$U$ and a downset~$D$,
$$%
  \Hom_\cQ(\kk[U], \kk[D]) = \kk^{\pi_0(U \cap D)},
$$
a product of copies of\/~$\kk$, one for each connected component of $U
\cap D$.

\item\label{i:U'->U}%
For upsets $U$ and~$U'$,
$$%
  \Hom_\cQ(\kk[U'], \kk[U]) = \kk^{\{S\in\pi_0 U' \,\mid\, S \subseteq U\}},
$$
a product of copies of\/~$\kk$, one for each connected component
of~$U'$ contained in~$U$.

\item\label{i:D->D'}%
For downsets $D$ and~$D'$,
$$%
  \Hom_\cQ(\kk[D], \kk[D']) = \kk^{\{S\in\pi_0 D' \,\mid\, S \subseteq D\}},
$$
a product of copies of\/~$\kk$, one for each connected component
of~$D'$ contained in~$D$.
\end{enumerate}
\end{prop}
\begin{proof}
For the first claim, the action~$\phi_q$ of $\phi: \kk[U] \to \kk[D]$
on the copy of~$\kk$ in any degree $q \in U \minus D$ is~$0$ because
$\kk[D]_q = 0$, so assume $q \in U \cap D$.  Then $\phi_q = \phi_{q'}:
\kk \to \kk$ if $q \preceq q' \in U \cap D$ because $\kk[U]_q \to
\kk[U]_{q'}$ and $\kk[D]_q \to \kk[D]_{q'}$ are identity maps
on~$\kk$.  Similarly, $\phi_q = \phi_{q'}$ if $q \succeq q' \in U \cap
D$.  Induction on the length of the path in
Definition~\ref{d:connected-poset}.\ref{i:connected} shows that
$\phi_q = \phi_{q'}$ if $q$ and~$q'$ lie in the same connected
component of $U \cap D$.  Thus $\Hom_\cQ(\kk[U], \kk[D]) \subseteq
\kk^{\pi_0(U \cap D)}$.  On the other hand, specifying for each
component $S \in \pi_0(U \cap D)$ a scalar $\alpha_S \in \kk$ yields a
homomorphism $\phi: \kk[U] \to \kk[D]$, if $\phi_q$ is defined to be
multiplication by~$\alpha_S$ on the copies of $\kk = \kk[U]_q$ indexed
by $q \in S$ and $0$ for $q \in U \minus D$; that $\phi$ is indeed a
$\cQ$-module homomorphism follows because $\kk[D]_{q'} = 0$ (that is,
$q' \not\in D$) whenever $q' \succeq q \in D$ but $q'$ does not lie in
the connected component of~$U \cap D$ containing~$q$.  Said another
way, pairs of elements of $U \cap D$ either lie in the same connected
component of $U \cap D$ or they are incomparable.

The proofs of the last two claims are similar (and dual to one
another), particularly when it comes to showing that a homomorphism of
indicator modules of the same type---that is, source and target both
upset or both downset---is constant on the relevant connected
components.  The only point not already covered is that if $U'$ is a
connected upset and $U' \not\subseteq U$ then every homomorphism
$\kk[U'] \to \kk[U]$ is~$0$ because $q' \in U' \minus U$ implies
$\kk[U']_{q'} \to 0 = \kk[U]_{q'}$.
\end{proof}

The cases of interest in this paper and its sequels
\cite{essential-real,functorial-multiPH}, particulary real and
discrete partially ordered groups (Definition~\ref{d:pogroup}) such
as~$\RR^n$ and~$\ZZ^n$, have strong connectivity properties, thereby
simplifying the conclusions of Proposition~\ref{p:U->D}.  First, here
is a convenient notation.

\begin{cor}\label{c:U->D}
Fix a poset~$\cQ$ with upsets $U,U'$ and downsets $D,D'$.
\begin{enumerate}
\item\label{i:kk}%
$\Hom_\cQ(\kk[U], \kk[D]) = \kk$ if $U \cap D \neq \nothing$ and
either $U$ is lower-connected as a subposet of~$\cQ$ or $D$ is
upper-connected as a subposet of~$\cQ$.

\item\label{i:U}%
If $U$ and~$U'$ are upsets and $\cQ$ is upper-connected, then
$\Hom_\cQ(\kk[U'],\kk[U]) = \kk$\linebreak if $U' \subseteq U$
and\/~$0$ otherwise.

\item\label{i:D}%
If $D$ and~$D'$ are downsets and $\cQ$ is lower-connected, then
$\Hom_\cQ(\kk[D],\kk[D']) = \kk$ if $D \supseteq D'$ and\/~$0$
otherwise.\qed
\end{enumerate}
\end{cor}

\begin{example}\label{e:disconnected-homomorphism}
Consider the poset $\NN^2$, the upset $U = \NN^2 \minus \{\0\}$, and
the downset $D$ consisting of the origin and the two standard basis
vectors.  Then $\kk[U] = \mm = \<x,y\>$ is the graded maximal ideal of
$\kk[\NN^2] = \kk[x,y]$ and $\kk[D] = \kk[\NN^2]/\mm^2$.  Now
calculate
$$%
  \Hom_{\NN^2}(\kk[U],\kk[D])
  =
  \Hom_{\NN^2}(\mm,\kk[\NN^2]/\mm^2)
  =
  \kk^2,
$$
a vector space of dimension~$2$: one basis vector preserves the
monomial~$x$ while killing the monomial~$y$, and the other basis
vector preserves~$y$ while killing~$x$.
\end{example}

\begin{example}\label{e:totally-disconnected}
For an extreme case, consider the poset $\cQ = \RR^2$ with $U$ the
closed half-plane above the antidiagonal line $y = -x$ and $D = -U$,
so that $U \cap D$ is totally disconnected: $\pi_0(U \cap D) = U \cap
D$.  Then $\Hom_\cQ(\kk[U],\kk[D]) = \kk^\RR$ is a vector space of
beyond continuum dimension, the copy of~$\RR$ in the exponent being
the antidiagonal~line.
\end{example}

The proliferation of homomorphisms in
Examples~\ref{e:disconnected-homomorphism}
and~\ref{e:totally-disconnected} is undesirable for both computational
and theoretical purposes; it motivates the following concept.

\begin{defn}\label{d:connected-homomorphism}
Let each of $S$ and $S'$ be a nonempty intersection of an upset in a
poset~$\cQ$ with a downset in~$\cQ$, so $\kk[S]$ and $\kk[S']$ are
subquotients of~$\kk[\cQ]$.  A homomorphism $\phi: \kk[S] \to \kk[S']$
is \emph{connected} if there is a scalar $\lambda \in \kk$ such that
$\phi$ acts as multiplication by~$\lambda$ on the copy of~$\kk$ in
degree $q$ for all $q \in S \cap S'$.
\end{defn}

The cases of interest in the rest of this paper concern three
situations: both $S$ and~$S'$ are upsets, or both are downsets, or $S
= U$ is an upset and $S' = D$ is downset with $U \cap D \neq
\nothing$.  However, the full generality of
Definition~\ref{d:connected-homomorphism} is required in the sequel to
this work \cite{essential-real}.

\begin{remark}\label{r:U->D}
Corollary~\ref{c:U->D} says that homomorphisms among indicator modules
are automatically connected in the presence of appropriate upper- or
lower-connectedness.
\end{remark}

\subsection{Fringe presentations}\label{sub:fringe}\mbox{}

\begin{defn}\label{d:fringe}
Fix any poset~$\cQ$.  A \emph{fringe presentation} of a
$\cQ$-module~$\cM$ is
\begin{itemize}
\item%
a direct sum~$F$ of upset modules~$\kk[U]$,
\item%
a direct sum~$E$ of downset modules~$\kk[D]$, and
\item%
a homomorphism $F \to E$ of $\cQ$-modules with
\begin{itemize}
  \item%
  image isomorphic to~$\cM$ and
  \item%
  components $\kk[U] \to \kk[D]$ that are connected
  (Definition~\ref{d:connected-homomorphism}).
\end{itemize}
\end{itemize}
A fringe presentation
\begin{enumerate}
\item%
is \emph{finite} if the direct sums are finite;

\item\label{i:dominate}%
\emph{dominates} a constant subdivision of~$\cM$ if the subdivision is
subordinate to each summand $\kk[U]$ of~$F$ and~$\kk[D]$ of~$E$; and

\item\label{i:auxiliary-fringe}%
is \emph{semialgebraic}, \emph{PL}, \emph{subanalytic}, or \emph{of
class~$\mathfrak X$} if $\cQ$ is a subposet of a partially ordered
real vector space of finite dimension and the fringe presentation
dominates a constant subdivision of the corresponding type
(Definition~\ref{d:auxiliary-hypotheses}).
\end{enumerate}
\end{defn}

Fringe presentations are effective data structures via the following
notational trick.  Topologically, it highlights that births occur
along the lower boundaries of the upsets and deaths occur along the
upper boundaries of the downsets, with a linear map over the ground
field to relate them.

\begin{defn}\label{d:monomial-matrix-fr}
Fix a finite fringe presentation $\phi: \bigoplus_p \kk[U_p] = F \to E
= \bigoplus_q \kk[D_q]$.  A \emph{monomial matrix} for $\phi$ is an
array of \emph{scalar entries}~$\phi_{pq}$ whose columns are labeled
by the \emph{birth upsets}~$U_p$ and whose rows are labeled by the
\emph{death downsets}~$D_q$:
$$%
\begin{array}{ccc}
  &
  \monomialmatrix
	{U_1\\\vdots\ \\U_k\\}
	{\begin{array}{ccc}
		   D_1    & \cdots &    D_\ell   \\
		\phi_{11} & \cdots & \phi_{1\ell}\\
		\vdots    & \ddots &   \vdots    \\
		\phi_{k1} & \cdots & \phi_{k\ell}\\
	 \end{array}}
	{\\\\\\}
\\
  \kk[U_1] \oplus \dots \oplus \kk[U_k] = F
  & \fillrightmap
  & E = \kk[D_1] \oplus \dots \oplus \kk[D_\ell].
\end{array}
$$
\end{defn}

\begin{prop}\label{p:scalars}
With notation as in Definition~\ref{d:monomial-matrix-fr}, $\phi_{pq}
= 0$ unless $U_p \cap D_q \neq \nothing$.  Conversely, if an array of
scalars $\phi_{pq} \in \kk$ with rows labeled by upsets and columns
label\-ed by downsets has $\phi_{pq} = 0$ unless $U_p \cap D_q \neq
\nothing$, then it represents a fringe~\mbox{presentation}.
\end{prop}
\begin{proof}
Proposition~\ref{p:U->D}.\ref{i:U->D} and
Definition~\ref{d:connected-homomorphism}.
\end{proof}

\begin{example}\label{e:one-param-fringe}
Fringe presentation in one parameter reflects the usual matching
between left endpoints and right endpoints of a module, once it has
been decomposed as a direct sum of bars.  A single bar, say an
interval $[a,b)$ that is closed on the left and open on the right, has
fringe presentation
$$%
\hspace{8ex}
\psfrag{a}{\footnotesize\raisebox{-.2ex}{$a$}}
\psfrag{b}{\footnotesize\raisebox{-.2ex}{$b$}}
\psfrag{vert-to}{\small\raisebox{-.2ex}{$\downarrow$}}
\psfrag{vert-into}{\small\raisebox{-.2ex}{$\lhookdownarrow$}}
\psfrag{vert-onto}{\small\raisebox{-.2ex}{$\twoheaddownarrow$}}
\psfrag{has image}{}
\begin{array}{@{}l@{}}
\includegraphics[height=20mm]{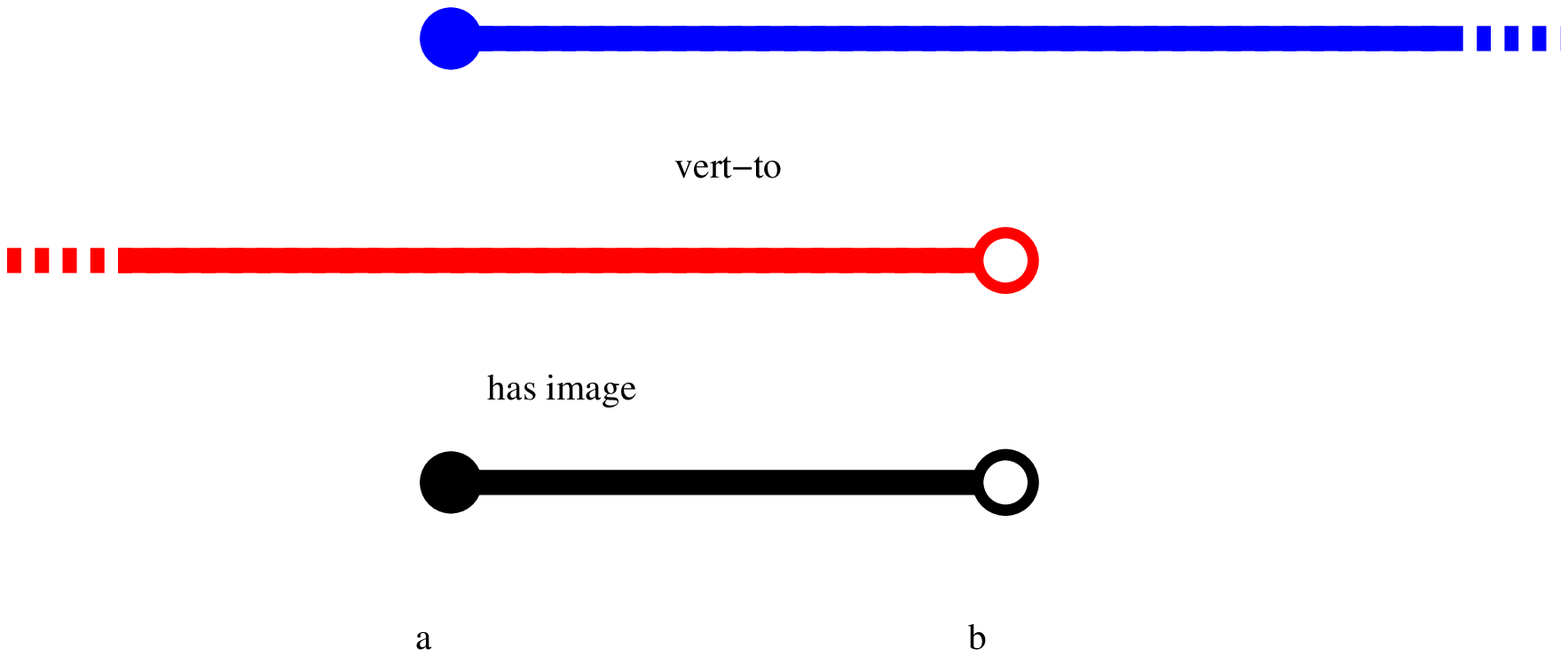}\\[-5ex]
\llap{with image\hspace{-1.5ex}}\\[2ex]
\end{array}
$$
in which a subset $S \subseteq \RR$ is drawn instead of writing
$\kk[S]$.  With multiple bars, the bijection from left to right
endpoints yields a monomial matrix whose scalar entries form the
identity, with rows labeled by positive rays having the specified left
endpoints (the ray is the whole real line when the left endpoint
is~$-\infty$) and columns labeled by negative rays having the
corresponding right endpoints (again, the whole line when the right
endpoint~is~$+\infty$).  In practical terms, the rows and columns can
be labeled simply by the endpoints themselves, with (say) a bar over a
closed endpoint and a circle over an open one.  Thus a standard bar
code, in monomial matrix notation, has the form
$$%
\begin{array}{c}
\monomialmatrix
	{ \blu \ol a_1 \\ \blu\vdots\,\ \\ \blu \ol a_k \\}
	{\begin{array}{ccc}
		\red \ob_1 & \red\cdots & \red \ob_k \\
		    1    &            &          \\
		         &  \ddots    &          \\
		         &            &     1    \\
	 \end{array}}
	{\\\\\\}.
\end{array}
$$
\end{example}

\begin{example}\label{e:two-param-fringe}
Although there are many opinions about what the multiparameter
analogue of a bar code should be, the analogue of a single bar is
generally accepted to be some kind of interval in the parameter
poset---that is, $\kk[U \cap D]$, where $U$ is an upset and $D$ is a
downset---sometimes with restrictions on the shape of the interval,
depending on context.  This case of a single bar explains the
terminology ``birth upset'' and ``death downset''.  For instance, a
fringe presentation of the yellow interval%
$$%
\hspace{-2.2ex}
\ \,\
\begin{array}{@{}r@{\hspace{-5.5pt}}|@{}l@{}}
\raisebox{-.2mm}{\includegraphics[height=25mm]{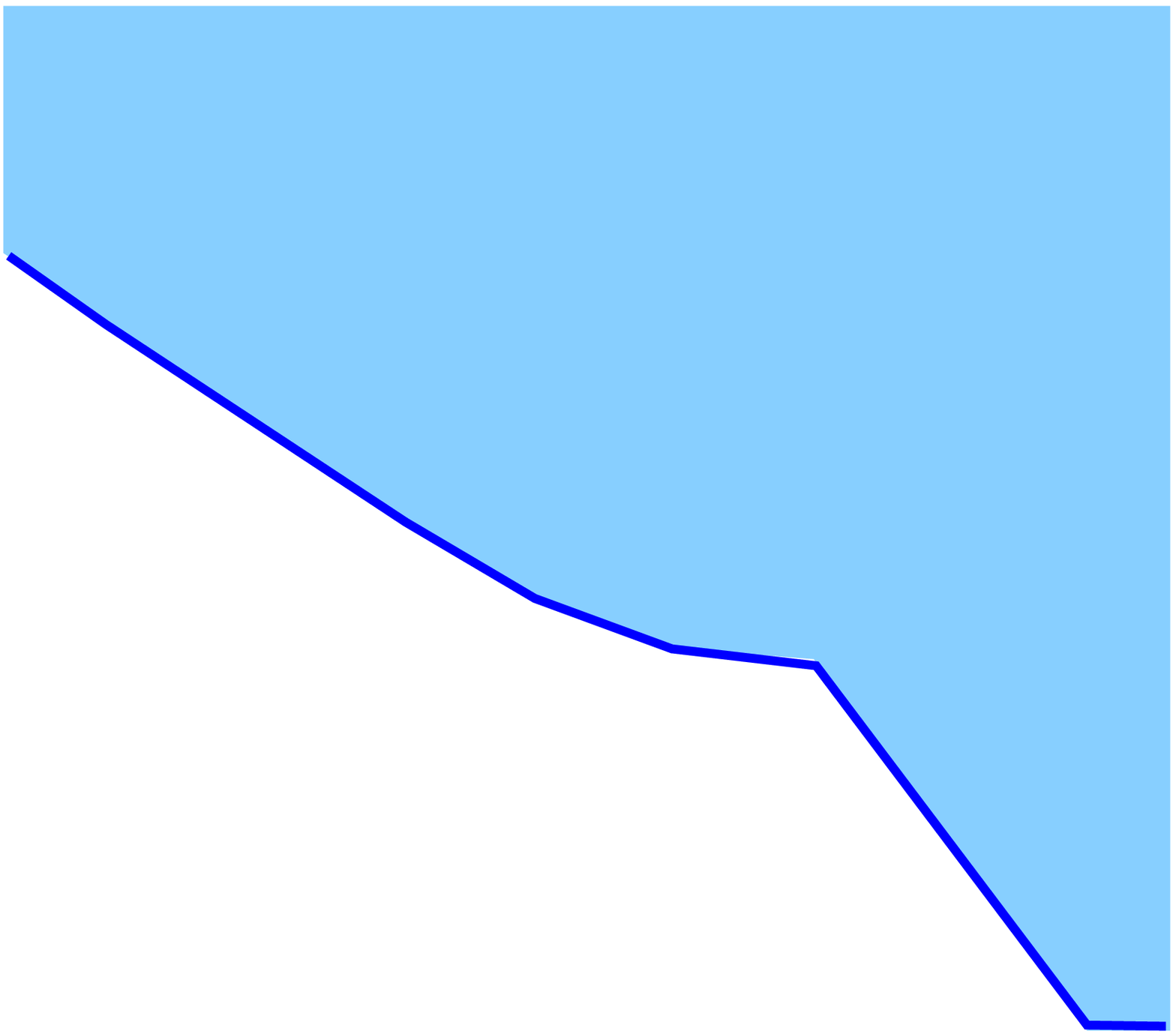}}
&\ \,\hspace{-.2pt}\\[-4.1pt]\hline
\end{array}
\ \to\
\begin{array}{@{}r@{\hspace{-.3pt}}|@{}l@{}}
\raisebox{-5mm}{\includegraphics[height=27mm]{red-downset}}
&\ \,\\[-6.3pt]\hline
\end{array}
\hspace{1.6ex}
\begin{array}{@{}c@{}}
\text{with image}\\[0ex]
\end{array}
\hspace{.7ex}
\hspace{.1pt}
\begin{array}{c}
\\[-1.5ex]
\begin{array}{@{}r@{\hspace{-.4pt}}|@{}l@{}}
\includegraphics[height=25mm]{semialgebraic}&\ \,\hspace{-.3pt}\\[-4.2pt]\hline
\end{array}
\\[-1.5ex]\mbox{}
\end{array}
\hspace{-6pt}
$$
locates the births along the lower boundary of the blue upset and the
deaths along the upper boundary of the red downset.  The scalar
entries relate the births to the deaths.  In this special case of one
bar, the monomial matrix is $1 \times 1$ with a single nonzero scalar
entry; choosing bases appropriately, this nonzero entry might as well
be~$1$.
\end{example}

\pagebreak

\begin{example}\label{e:flange-switching}
Consider an $\NN^2$-filtration of the following simplicial complex.
$$%
  \includegraphics[width=50mm]{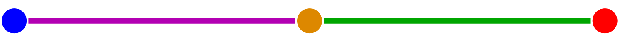}
$$
Each simplex is present above the correspondingly colored rectangular
curve in the following diagram, which theoretically should extend
infinitely far up and to the~right.
$$%
  \includegraphics[width=65mm]{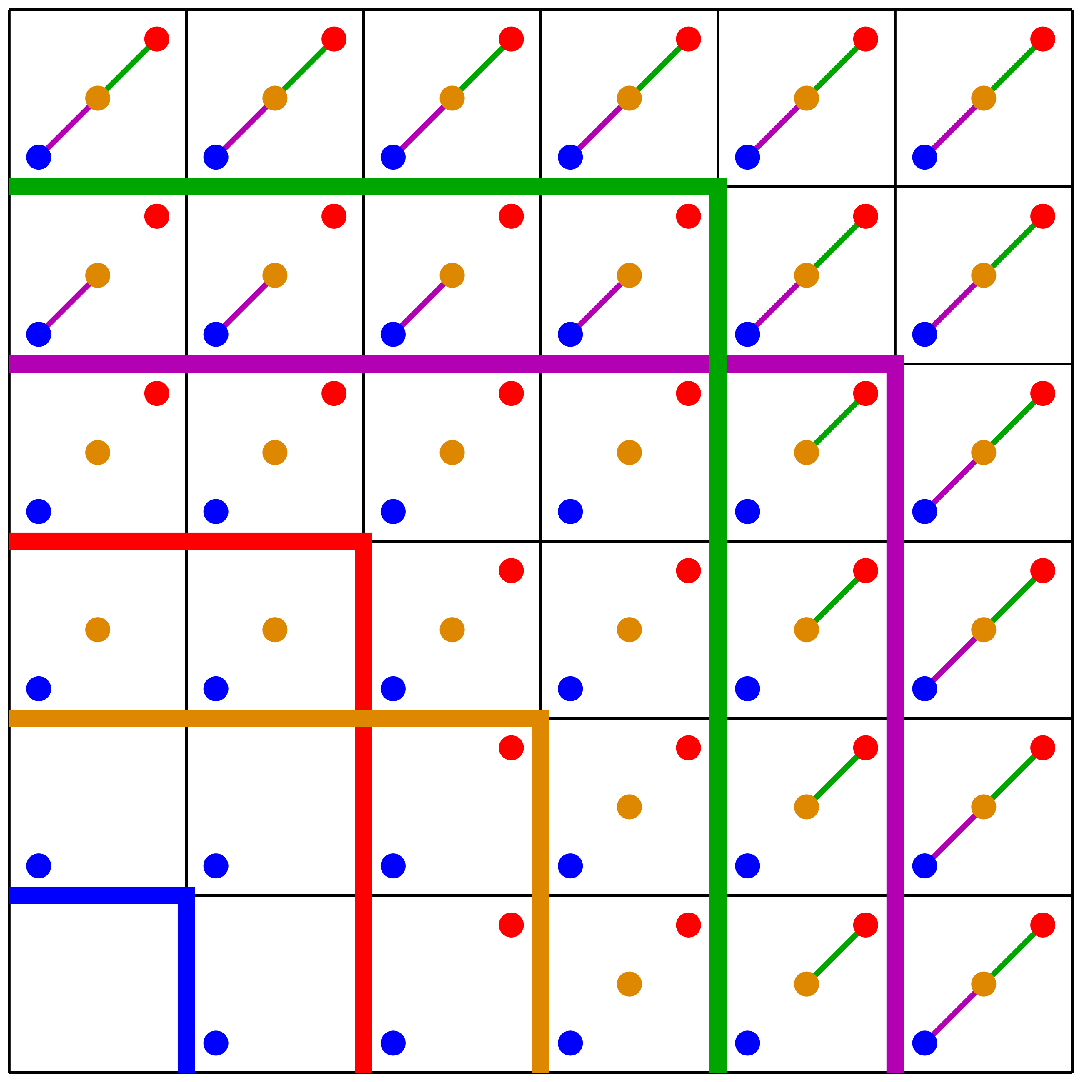}
$$
Each little square depicts the simplicial complex that is present at
the parameter occupying its lower-left corner.  Taking zeroth homology
yields an $\NN^2$-module that replaces the simplicial complex in each
box with the vector space spanned by its connected components.  A
fringe presentation for this $\NN^2$-module is
$$%
\begin{array}{c}
\\
\monomialmatrix
	{\includegraphics[width=10mm]{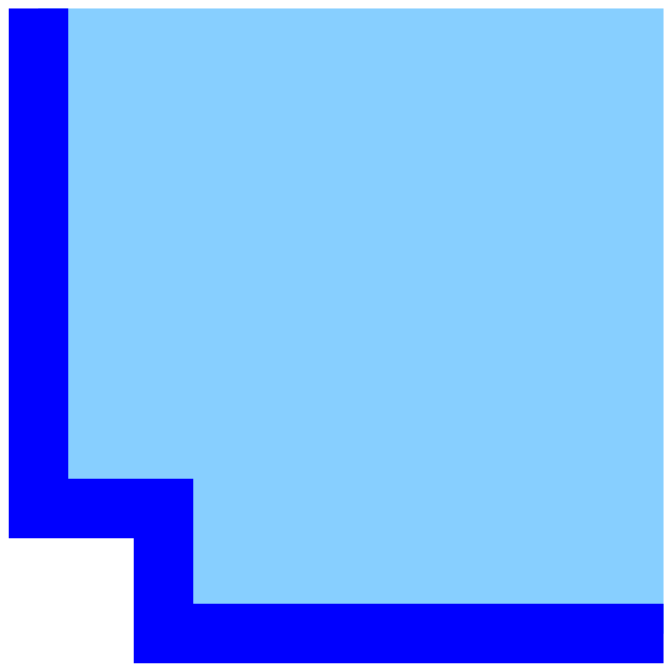}\\
	 \includegraphics[width=10mm]{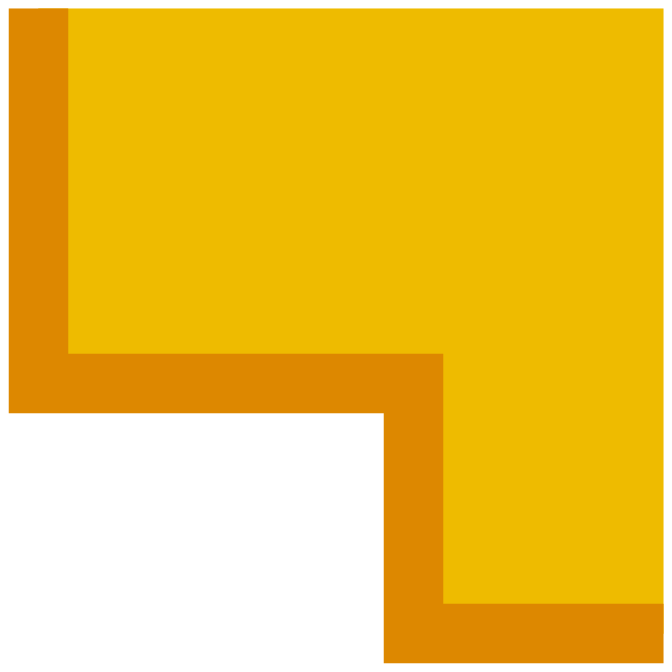}\\
	 \includegraphics[width=10mm]{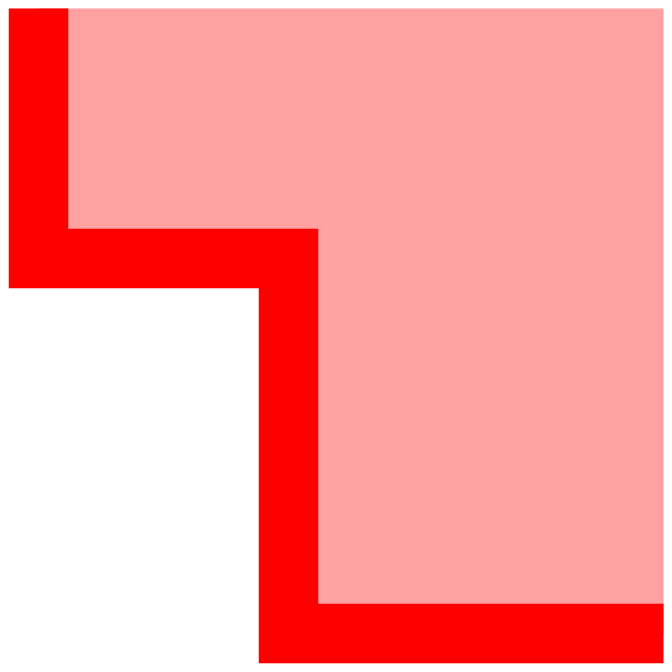}\\}
	{\begin{array}{ccc}
	\\[-8ex]
	\includegraphics[width=10mm]{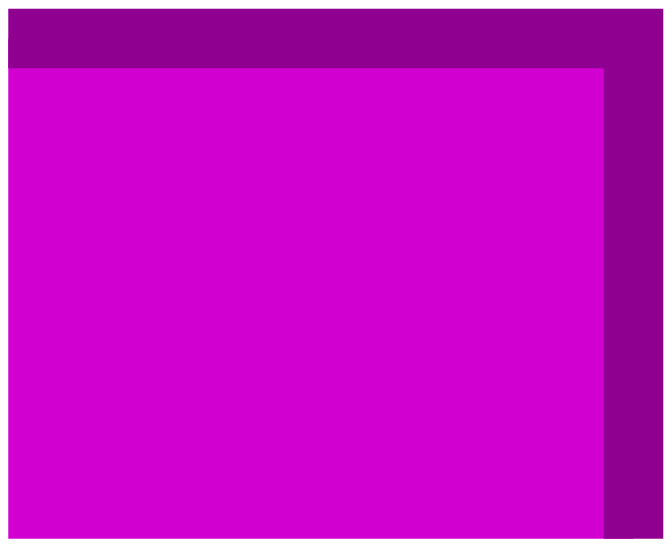}&
	\includegraphics[width=10mm]{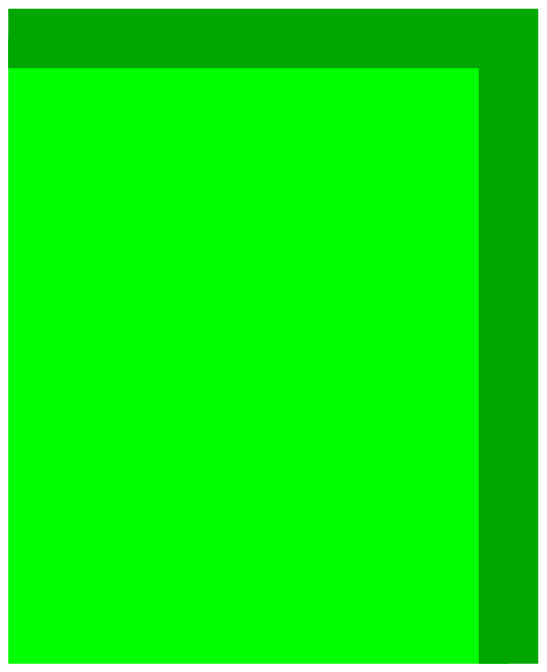}&
	\includegraphics[width=10mm]{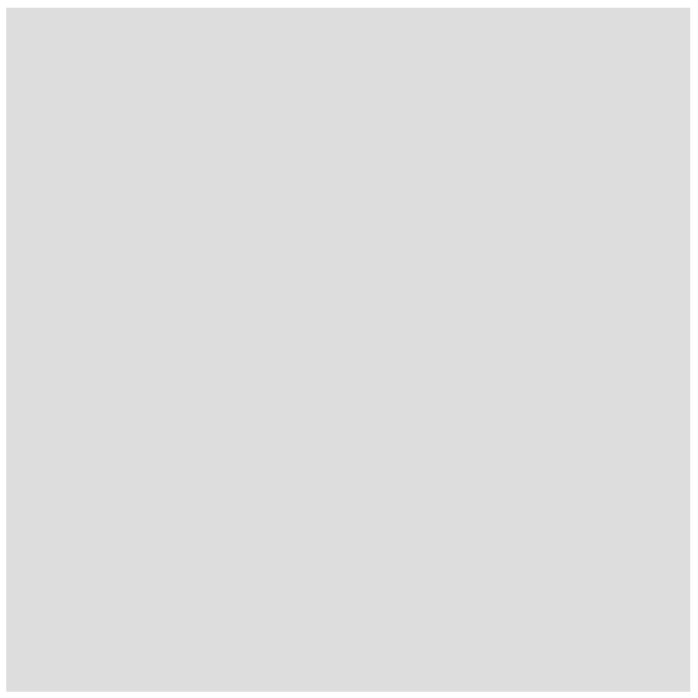}\\
		\\[-2ex] 1 & 0 &1 \\[2.5ex]
		\\[-2ex]-1 & 1 &1 \\[2.5ex]
		\\[-2ex] 0 &-1 &1 \\
	 \end{array}}
	{\\\\\\\\\\\\\\},
\end{array}
$$
where the grey square atop the third column represents the downset
that is all of~$\NN^2$.  This fringe presentation means that, for
example, the connected component that is the blue endpoint of the
simplicial complex is born along the union of the axes with the origin
removed but the point $\bigl[\twoline 11\bigr]$ appended.  The purple
downset, corresponding to the left edge, records the death---along the
upper purple boundary---of the homology class represented by the
difference of the blue (left) and gold (middle) vertices.
Computations and figures for this example were kindly provided by
Ashleigh Thomas.
\end{example}

\begin{remark}\label{r:portmanteau-fr}
The term ``fringe'' is a portmanteau of ``free'' and ``injective''
(that is, ``frinj''), the point being that it combines aspects of free
and injective resolutions while also conveying that the data structure
captures trailing topological features at both the birth and death
ends.
\end{remark}

\section{Encoding poset modules}\label{s:encoding}

Sections~\ref{s:tame} and~\ref{s:fringe} introduce two finiteness
conditions: a topological one (tameness, Definition~\ref{d:tame}),
which is the intuitive control of homological variation in a
filtration of a topological space, and an algebraic one (fringe
presentation, Definition~\ref{d:fringe}), which provides effective
data structures.  To interpolate between them, a third finiteness
condition, this one combinatorial in nature (finite encoding,
Definition~\ref{d:encoding}), serves as a theoretical tool whose
functorial essence supports much of the development in this paper; the
category of tame modules (Section~\ref{sub:cat}) is best dealt with
using this language, for instance.  The main result of
Section~\ref{s:encoding}, namely Theorem~\ref{t:tame}, says that tame
$\cQ$-modules can be encoded in the manner of
Definition~\ref{d:encoding}.  Theorems~\ref{t:constant-uptight}
and~\ref{t:tame} are a substantial portion of the main result of the
paper (Theorem~\ref{t:syzygy}), and their proofs contribute key
arguments not repeated there although their
statements~are~largely~\mbox{subsumed}.

\subsection{Finite encoding}\label{sub:encoding}\mbox{}

\begin{defn}\label{d:encoding}
Fix a poset~$\cQ$.  An \emph{encoding} of a $\cQ$-module $\cM$ by a
poset~$\cP$ is a poset morphism $\pi: \cQ \to \cP$ together with a
$\cP$-module $\cH$ such that $\cM \cong \pi^*\cH =
\bigoplus_{q\in\cQ}H_{\pi(q)}$, the \emph{pullback of~$\cH$
along~$\pi$}, which is naturally a $\cQ$-module.  The encoding is
\emph{finite} if
\begin{enumerate}
\item%
the poset $\cP$ is finite, and
\item%
the vector space $H_p$ has finite dimension for all $p \in \cP$.
\end{enumerate}
\end{defn}

\begin{example}\label{e:wing-subdivision}
Example~\ref{e:toy-model-fly-wing} shows a constant isotypic
subdivision of $\RR^2$ which happens to form a poset and therefore
produces an encoding.
\end{example}

\begin{example}\label{e:flange-switching'}
A finite encoding of the module in Example~\ref{e:flange-switching} is
as follows.
$$%
\psfrag{0-1}{\tiny$\left[\begin{array}{@{}c@{}}
	0\\[-.5ex]1\end{array}\right]$}
\psfrag{10-00-01}{\tiny$\left[\begin{array}{@{}c@{\ }c@{}}
	1&0\\[-.5ex]0&0\\[-.5ex]0&1\end{array}\right]$}
\psfrag{00-10-01}{\tiny$\left[\begin{array}{@{}c@{\ }c@{}}
	0&0\\[-.5ex]1&0\\[-.5ex]0&1\end{array}\right]$}
\psfrag{110-001}{\tiny$\left[\begin{array}{@{}c@{\ }c@{\ }c@{}}
	1&1&0\\[-.5ex]0&0&1\end{array}\right]$}
\psfrag{100-011}{\tiny$\left[\begin{array}{@{}c@{\ }c@{\ }c@{}}
	1&0&0\\[-.5ex]0&1&1\end{array}\right]$}
\psfrag{11}{\tiny$\left[\begin{array}{@{}c@{\ }c@{}}
	1&1\end{array}\right]$}
\psfrag{kk}{\footnotesize$\kk$}
\psfrag{kk2}{\footnotesize$\kk^2$}
\psfrag{kk3}{\footnotesize$\kk^3$}
\begin{array}{@{}c@{}}\includegraphics[height=63mm]{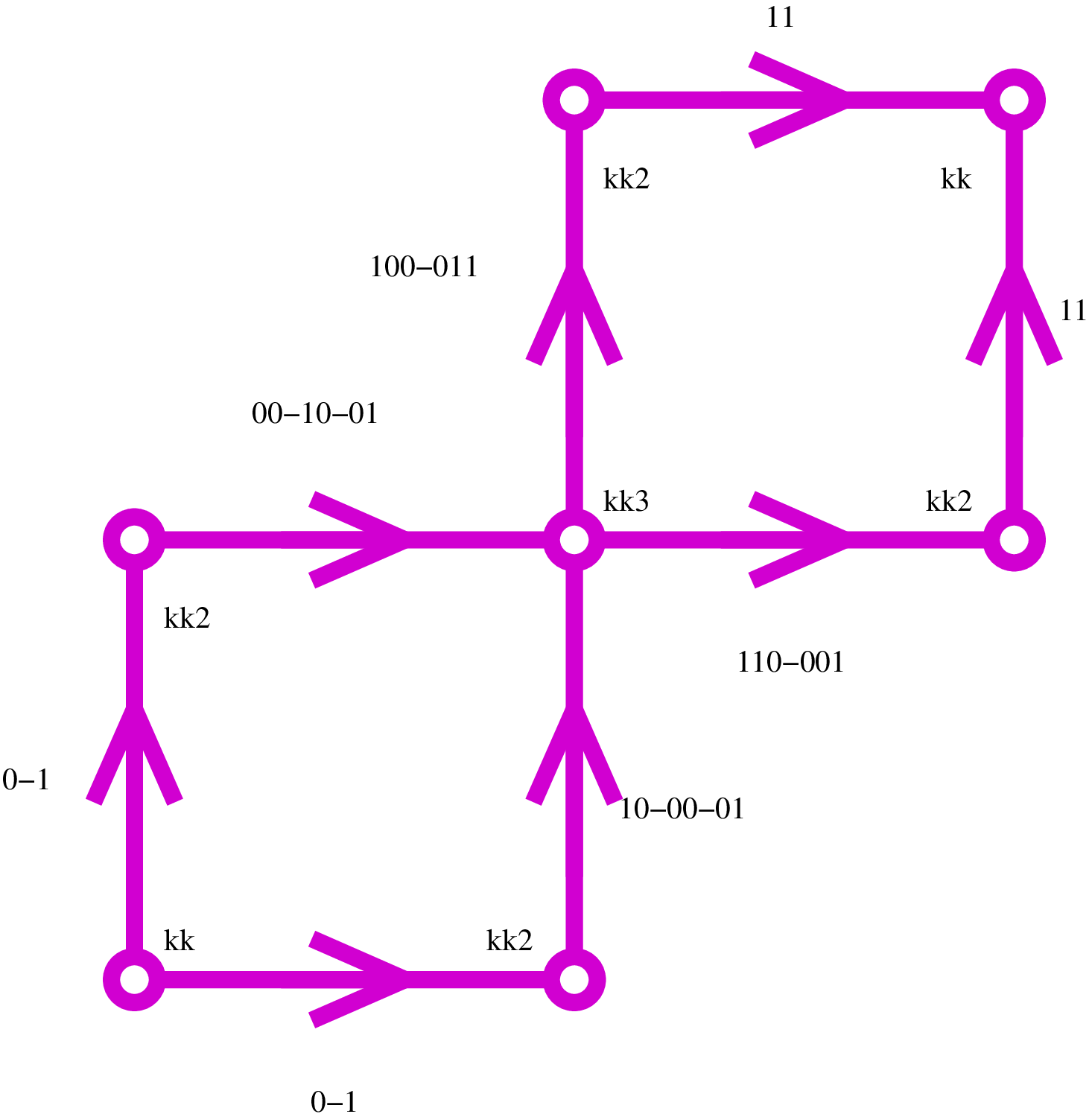}\end{array}
$$
\end{example}

\begin{example}\label{e:subdivide}
There is no natural way to impose a poset structure on the set of
regions in a constant subdivision.  Take, for example, $\cQ = \RR^2$
and $\cM = \kk_\0 \oplus \kk[\RR^2]$, where $\kk_\0$ is the
$\RR^2$-module whose only nonzero component is at the origin, where it
is a vector space of dimension~$1$.  This module~$\cM$ induces only
two isotypic regions, namely the origin and its complement, and they
constitute a constant subdivision.
$$%
\includegraphics[width=10mm]{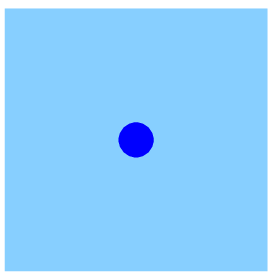}
\begin{array}{c}=\\[4ex]\end{array}
\includegraphics[width=10mm]{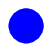}
\hspace{-1.25mm}\begin{array}{c}\cup\\[4ex]\end{array}\hspace{1.8mm}
\includegraphics[width=10mm]{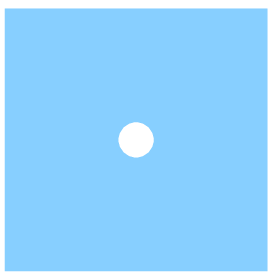}
\vspace{-3ex}
$$
Neither of the two regions has a stronger claim to precede the other,
but at the same time it would be difficult to justify forcing the
regions to be incomparable.
\end{example}

\begin{example}\label{e:convex-projection}
Take $\cQ = \ZZ^n$ and $\cP = \NN^n$.  The \emph{convex projection}
$\ZZ^n \to \NN^n$ sets to~$0$ every negative coordinate.  The pullback
under convex projection is the \v Cech hull
\cite[Definition~2.7]{alexdual}.  More generally, suppose $\aa \preceq
\bb$ in~$\ZZ^n$.  The interval $[\aa,\bb] \subseteq \ZZ^n$ is a box
(rectangular parallelepiped) with lower corner at~$\aa$ and upper
corner at~$\bb$.  The \emph{convex projection} $\pi: \ZZ^n \to
[\aa,\bb]$ takes every point in~$\ZZ^n$ to its closest point in the
box.  A $\ZZ^n$-module is \emph{finitely determined} if it is finitely
encoded by~$\pi$.
\end{example}

\begin{example}\label{e:indicator}
The indicator module $\kk[\cQ]$ is encoded by the morphism from $\cQ$
to the one-point poset with vector space $\cH = \kk$.  This
generalizes to other indicator modules.
\begin{enumerate}
\item%
Any upset module $\kk[U] \subseteq \kk[\cQ]$ is encoded by a morphism
from~$\cQ$ to the chain~$P$ of length~$1$, consisting of two points $0
< 1$, that sends $U$ to~$1$ and the complement $\ol U$ to~$0$.  The
$P$-module~$H$ that pulls back to~$\kk[U]$ has $H_0 =\nolinebreak
0$~and~$H_1 = \kk$.
\item%
Dually, any downset module $\kk[D]$ is also encoded by a morphism
from~$\cQ$ to the chain $P$ of length~$1$, but this one sends $D$
to~$0$ and the complement $\ol D$ to~$1$, and the $P$-module~$H$ that
pulls back to~$\kk[D]$ has $H_0 = \kk$~and~$H_1 = 0$.
\end{enumerate}
\end{example}

\begin{defn}\label{d:subordinate-encoding}
Fix a poset~$\cQ$ and a $\cQ$-module~$\cM$.
\begin{enumerate}
\item\label{i:morphism}%
A poset morphism $\pi: \cQ \to \cP$ or an encoding of a $\cQ$-module
(perhaps different from~$\cM$) is \emph{subordinate} to~$\cM$ if there
is a $\cP$-module~$\cH$ such that~$\cM \cong\nolinebreak \pi^*\cH$.

\item\label{i:auxiliary-encoding}%
When $\cQ$ is a subposet of a partially ordered real vector space, an
encoding of~$\cM$ is \emph{semialgebraic}, \emph{PL},
\emph{subanalytic}, or \emph{of class~$\mathfrak X$} if the partition
of~$\cQ$ formed by the fibers of~$\pi$ is of the corresponding type
(Definition~\ref{d:auxiliary-hypotheses}).
\end{enumerate}
\end{defn}

\begin{example}\label{e:antidiagonal'}
The ``antidiagonal'' $\RR^2$-module $\cM$ in
Example~\ref{e:antidiagonal} has a semialgebraic poset encoding by the
chain with three elements, where the fiber over the middle element is
the antidiagonal line, and the fibers over the top and bottom elements
are the open half-spaces above and below the line, respectively.
In contrast, using the diagonal line spanned by $\left[\twoline
11\right] \in \RR^2$ instead of the antidiagonal line yields a module
with no finite encoding; see Example~\ref{e:diagonal}.
\end{example}

\begin{lemma}\label{l:constant}
An indicator module is constant on every fiber of a poset morphism
$\pi: \cQ \to \cP$ if and only if the module is the pullback
along~$\pi$ of an indicator $\cP$-module.
\end{lemma}
\begin{proof}
The ``if'' direction is by definition.  For the ``only if'' direction,
observe that if $U \subseteq \cQ$ is an upset that is a union of
fibers of~$\cP$, then the image $\pi(U) \subseteq \cP$ is an upset
whose preimage equals~$U$.  The same argument works for downsets.
\end{proof}

\begin{example}[Pullbacks of flat and injective modules]\label{e:pullback}
An indecomposable flat $\ZZ^n$-module $\kk[\bb + \ZZ\tau + \NN^n]$ is
an upset module for the poset~$\ZZ^n$.  Pulling back to any poset
under a poset map to~$\ZZ^n$ therefore yields an upset module for the
given poset.  The dual statement holds for any indecomposable
injective module $\kk[\bb + \ZZ\tau - \NN^n]$: its pullback is a
downset module.
\end{example}

Pullbacks have particularly transparent monomial matrix
interpretations.

\begin{prop}\label{p:pullback-monomial-matrix}
Fix a poset~$\cQ$ and an encoding of a $\cQ$-module~$\cM$ via a poset
morphism $\pi: \cQ \to \cP$ and $\cP$-module~$\cH$.  Any monomial
matrix for a fringe presentation of~$\cH$ pulls back to a monomial
matrix for a fringe presentation that dominates the encoding by
replacing the row labels $U_1,\dots,U_k$ and column labels
$D_1,\dots,D_\ell$ with their preimages, namely $\pi^{-1}(U_1), \dots,
\pi^{-1}(U_k)$ and $\pi^{-1}(D_1), \dots, \pi^{-1}(D_\ell)$.
\hfill$\square$
\end{prop}

\subsection{Uptight posets}\label{sub:uptight}\mbox{}

\noindent
Constructing encodings from constant subdivisions uses general poset
combinatorics.

\begin{defn}\label{d:uptight-region}
Fix a poset~$\cQ$ and a set $\Upsilon$ of upsets.  For each poset
element $\aa \in \cQ$, let $\Upsilon_\aa \subseteq \Upsilon$ be the
set of upsets from~$\Upsilon$ that contain~$\aa$.  Two poset elements
$\aa,\bb \in \cQ$ lie in the same \emph{uptight region} if
$\Upsilon_\aa = \Upsilon_\bb$.
\end{defn}

\begin{remark}\label{r:common-refinement}
The partition of $\cQ$ into uptight regions in
Definition~\ref{d:uptight-region} is the common refinement of the
partitions $\cQ = U \cupdot (\cQ \minus U)$ for $U \in \Upsilon$.
\end{remark}

\begin{remark}\label{r:iso-uptight}
Every uptight region is the intersection of a single upset (not
necessarily one of the ones in~$\Upsilon$) with a single downset.
Indeed, the intersection of any family of upsets is an upset, the
complement of an upset is a downset, and the intersection of any
family of downsets is a downset.  Hence the uptight region
containing~$\aa$ equals $\bigl(\bigcap_{U \in \Upsilon_\aa} U\bigr)
\cap \bigl(\bigcap_{U \not\in \Upsilon_\aa} \ol U\bigr)$, the first
intersection being an upset and the second a~downset.
\end{remark}

\begin{prop}\label{p:posetQuotient}
In the situation of Definition~\ref{d:uptight-region}, the relation on
uptight regions given by $A \preceq B$ whenever $\aa \preceq \bb$ for
some $\aa \in A$ and $\bb \in B$ is reflexive and acyclic.
\end{prop}
\begin{proof}
The stipulated relation on the set of uptight regions is
\begin{itemize}
\item%
reflexive because $\aa \preceq \aa$ for any element~$\aa$ in any
uptight region~$A$; and

\item%
acyclic because going up from $\aa \in \cQ$ causes the set
$\Upsilon_\aa$ in Definition~\ref{d:uptight-region} to (weakly)
increase, so a directed cycle can only occur with a constant sequence
of sets~$\Upsilon_\aa$.\qedhere
\end{itemize}
\end{proof}

\begin{example}\label{e:puuska-nontransitive}
The relation in Proposition~\ref{p:posetQuotient} makes the set of
uptight regions into a directed acyclic graph, but the relation need
not be transitive.  An example in the poset $\cQ = \NN^2$, kindly
provided by Ville Puuska \cite{puuska18}, is as follows.
Notationally, it is easier to work with monomial ideals in $\kk[x,y] =
\kk[\NN^2]$, which correspond to upsets (see
Example~\ref{e:disconnected-homomorphism}).  Let $\Upsilon =
\{U_1,\dots,U_4\}$ consist of the upsets with indicator modules
$$%
  \kk[U_1] = \<x^2,y\>,\quad
  \kk[U_2] = \<x^3,y\>,\quad
  \kk[U_3] = \<xy\>,\quad
  \kk[U_4] = \<x^2y\>.
$$
Identifying each monomial $x^ay^b$ with the corresponding pair $(a,b)
\in \NN^2$, it follows that $\Upsilon_{\!x^2} = \{U_1\}$, and
$\Upsilon_{\!x^3} = \Upsilon_{\!y} = \{U_1,U_2\}$, and
$\Upsilon_{\!xy} = \{U_1,U_2,U_3\}$ represent three distinct uptight
regions; call them $A$, $B$, and $C$.  They satisfy $A \prec B \prec
C$ because $x^2 \prec x^3$ and $y \prec xy$.  However, $A \not \preceq
C$ because $A = \{x^2\}$ while $U_4$ forces $C = xy\kk[y]$ to consist
of all lattice points in a vertical ray starting at $xy$.
\end{example}

\begin{defn}\label{d:uptight-poset}
In the situation of Definition~\ref{d:uptight-region}, the
\emph{uptight poset} is the transitive closure $\cP_\Upsilon$ of the
directed acyclic graph of uptight regions in
Proposition~\ref{p:posetQuotient}.
\end{defn}

\subsection{Constant upsets}\label{sub:upsets}

\begin{defn}\label{d:uptight-constant}
Fix a constant subdivision of~$\cQ$ subordinate to~$\cM$.  A
\emph{constant upset} of~$\cQ$ is either
\begin{enumerate}
\item%
an upset $U_I$ generated by a constant region~$I$ or
\item%
the complement of a downset $D_I$ cogenerated by a constant
region~$I$.
\end{enumerate}
\end{defn}

\begin{thm}\label{t:constant-uptight}
Let $\Upsilon$ be the set of constant upsets from a~constant
subdivision of~$\cQ$ subordinate to~$\cM$.  The partition of~$\cQ$
into uptight regions for $\Upsilon$ forms another constant subdivision
subordinate to~$\cM$.
\end{thm}
\begin{proof}
Suppose that $A$ is an uptight region that contains points from
constant regions $I$ and~$J$.  Any point in $I \cap A$ witnesses the
containments $A \subseteq D_I$ and $A \subseteq U_I$ of~$A$ inside the
constant upset and downset generated and cogenerated by~$I$.  Any
point $\jj \in J \cap A$ is therefore sandwiched between elements
$\ii, \ii' \in I$, so $\ii \preceq \jj \preceq \ii'$, because $\jj \in
U_I$ (for~$\ii$) and $\jj \in D_I$ (for~$\ii'$).  By symmetry,
switching $I$ and~$J$, there exists $\jj' \in J$ with $\ii' \preceq
\jj'$.  The sequence
$$%
  M_I \to \cM_\ii \to \cM_\jj \to \cM_{\ii'} \to \cM_{\jj'} \to M_J,
$$
where the first and last isomorphisms come from
Definition~\ref{d:constant-subdivision} and the homomorphisms in
between are $\cQ$-module structure homomorphisms, induces isomorphisms
$\cM_\ii \to \cM_{\ii'}$ and $\cM_\jj \to \cM_{\jj'}$ by definition of
constant region.  Elementary homological algebra implies that $\cM_\ii
\to \cM_\jj$ is an isomorphism.  The induced isomorphism $M_I \to M_J$
is independent of the choices of $\ii$, $\jj$, $\ii'$, and~$\jj'$ (in
fact, merely considering independence of the choices of~$\ii$
and~$\jj'$ would suffice) because constant subdivisions
have~no~monodromy.

The previous paragraph need not imply that $I = J$, but it does imply
that all of the vector spaces $M_J$ for constant regions~$J$ that
intersect~$A$ are---viewing the data of the original constant
subdivision as given---canonically isomorphic to~$M_I$, thereby
allowing the choice of $M_A = M_I$.  This, plus the lack of monodromy
in constant subdivisions, ensures that $M_A \to M_\aa \to M_\bb \to
M_B$ is independent of the choices of $\aa \in A$ and $\bb \in B$ with
$\aa \preceq \bb$.  Thus the uptight subdivision is also constant
subordinate~to~$\cM$.
\end{proof}

\begin{example}\label{e:iso-uptight}
Theorem~\ref{t:constant-uptight} does not claim that $I = U_I \cap
D_I$, and in fact that claim is often not true, even if the isotypic
subdivision (Example~\ref{e:puuska-nonconstant-isotypic}) is already
constant.  Consider $\cQ = \RR^2$ and $\cM = \kk_\0 \oplus
\kk[\RR^2]$, as in Example~\ref{e:subdivide}, and take $I = \RR^2
\minus \{\0\}$.  Then $U_I = D_I = \RR^2$, so $U_I \cap D_I$ contains
the other isotypic region $J = \{\0\}$.  The uptight poset $\cP_\cM$
has precisely four elements:
\begin{enumerate}
\item%
the origin $\{\0\} = U_J \cap D_J$;
\item%
the complement $U_J \minus \{\0\}$ of the origin in $U_J$;
\item%
the complement $D_J \minus \{\0\}$ of the origin in $D_J$; and
\item%
the points $\RR^2 \minus (U_J \cup D_J)$ lying only in~$I$ and in
neither $U_J$ nor~$D_J$.
\end{enumerate}
Oddly, uptight region~4 has two connected components, the second and
fourth quadrants $A$ and~$B$, that are incomparable: any chain of
relations from Definition~\ref{d:constant-subdivision} that realizes
the equivalence $\aa \sim \bb$ for $\aa \in A$ and~$\bb \in B$ must
pass through the positive quadrant or the negative quadrant, each of
which accidentally becomes comparable to the other isotypic region~$J$
and hence lies in a different uptight region.
\end{example}

\subsection{Finite encoding from constant subdivisions}\label{sub:existence}

\begin{defn}\label{d:compactly-supported}
If $\cQ$ is a subposet of a partially ordered real vector space, then
a $\cQ$-module~$\cM$ has \emph{compact support} if~$\cM$ has nonzero
components~$\cM_q$ only in a bounded set of degrees $q \in \cQ$.  A
constant subdivision subordinate to such a module is \emph{compact} if
it has exactly one unbounded constant region (namely those $q \in \cQ$
for which $\cM_q = 0$).
\end{defn}

\begin{thm}\label{t:tame}
Fix a $\cQ$-finite module~$\cM$ over a poset~$\cQ$.
\begin{enumerate}
\item\label{i:admits-finite-encoding}%
$\cM$ admits a finite encoding if and only if there exists a finite
constant subdivision of~$\cQ$ subordinate to~$\cM$.  More precisely,

\item\label{i:uptight-encoding}%
the uptight poset of the set of constant upsets from any constant
subdivision yields an \emph{uptight encoding} of~$\cM$ that is finite
if the constant subdivision is finite.

\item\label{i:auxiliary-uptight}%
If $\cQ$ is a subposet of a partially ordered real vector space and
the constant subdivision in the previous item is
\begin{itemize}
\item%
semialgebraic, with $\cQ_{+\!}$ also semialgebraic; or

\item%
PL, with $\cQ_{+\!}$ also polyhedral; or

\item%
compact and subanalytic, with $\cQ_{+\!}$ also subanalytic; or

\item%
of class~$\mathfrak X$,
\end{itemize}
then the relevant uptight encoding is semialgebraic, PL, subanalytic,
or class~$\mathfrak X$.
\end{enumerate}
\end{thm}
\begin{proof}
One direction of item~\ref{i:admits-finite-encoding} is easy: a finite
encoding induces a constant subdivision almost by definition.  Indeed,
if $\pi: \cQ \to \cP$ is a poset encoding of~$\cM$ by a
$\cP$-module~$\cH$, then each fiber~$I$ of~$\pi$ is a constant region
with $M_I = H_{\pi(I)}$.  If $\ii \preceq \jj$ with $\ii \in I$ and
$\jj \in J$, then the composite homomorphism $M_I \to M_\ii \to M_\jj
\to M_J$ is merely the structure morphism $H_{\pi(I)} \to H_{\pi(J)}$
of the $\cP$-module~$\cH$.

The hard direction is producing a finite encoding from a constant
subdivision.  For that, it suffices to prove
item~\ref{i:uptight-encoding}.  Let $\Upsilon$ be the set of constant
upsets from a~constant subdivision of~$\cQ$ subordinate to~$\cM$.
Consider the quotient map $\cQ \to \cP_\Upsilon$ of sets that sends
each element of~$\cQ$ to the uptight region containing it.
Proposition~\ref{p:posetQuotient} and Definition~\ref{d:uptight-poset}
imply that this map of sets is a morphism of posets.  By
Definition~\ref{d:constant-subdivision} the vector spaces $M_A$
indexed by the uptight regions $A \in \cP_\Upsilon$ constitute a
$\cP_\Upsilon$-module~$H$ that is well defined by
Theorem~\ref{t:constant-uptight}.  The pullback of~$H$ to~$\cQ$ is
isomorphic to~$\cM$ by construction.  The claim about finiteness
follows because the number of uptight regions is bounded above by
$2^{2r}$, where $r$ is the number of constant regions in the original
constant subdivision: every element of~$\cQ$ lies inside or outside of
each constant upset and inside or outside of each constant downset.

For claim~\ref{i:auxiliary-uptight}, every constant upset is a
Minkowski sum $I + \cQ_+$ or the complement of $I - \cQ_+ = - (-I +
\cQ_+)$ by Definition~\ref{d:uptight-constant}.  These are
semialgebraic, PL, subanalytic, or of class~$\mathfrak X$,
respectively, by Proposition~\ref{p:auxiliary-hypotheses} (or
Definition~\ref{d:auxiliary-hypotheses} for class~$\mathfrak X$).
Note that in the compact subanalytic case, the unique unbounded
constant region~$I$ afforded by Definition~\ref{d:compactly-supported}
has $I + \cQ_+ = I - \cQ_+ = \cQ$, which is subanalytic.
\end{proof}

\begin{example}\label{e:antidiagonal''}
For the ``antidiagonal'' $\RR^2$-module $\cM$ in
Examples~\ref{e:antidiagonal} and~\ref{e:antidiagonal'}, every point
on the line is a singleton isotypic region, but these uncountably many
isotypic regions can be gathered together: the finite encoding there
is the uptight poset for the two upsets that are the closed and open
half-spaces bounded below by the antidiagonal.
\end{example}

\begin{example}\label{e:subdivide'}
In any encoding of the ``diagonal strip'' $\RR^2$-module $\cM$ in
Example~\ref{e:subdivide}, the poset must be uncountable by
Theorem~\ref{t:tame}.
\end{example}

\subsection{The category of tame modules}\label{sub:cat}

\begin{example}\label{e:ker}
The kernel of a homomorphism of tame modules need not be tame.  The
upset $U \subseteq \RR^2$ that is the closed half-space above the
antidiagonal line~$L$ given by $a + b = 1$ has interior $U^\circ$,
also an upset.  The quotient module $N = \kk[U]/\kk[U^\circ]$ is the
translate by one unit (up or to the right) of the antidiagonal module
in Examples~\ref{e:antidiagonal}, \ref{e:antidiagonal'},
and~\ref{e:antidiagonal''}.  Both $M = \kk[U] \oplus \kk[U]$ and $N$
are tame.  The surjection $\phi: M \onto N$ that acts in every degree
$\aa = \left[\twoline ab \right]$ along~$L$ by sending the basis
vectors of $M_\aa = \kk^2$ to $b$ and $-a$ in $N_\aa = \kk$ has
kernel~$K = \ker\phi$ that is the submodule of~$M$ with
\begin{itemize}
\item%
$\kk^2$ in every degree from~$U^\circ$, and
\item%
the line in~$\kk^2$ through $\left[\twoline 00 \right]$ and
$\left[\twoline ab \right]$ in every degree from~$L$.
\end{itemize}
$$%
\psfrag{x}{\tiny$x$}
\psfrag{y}{\tiny$y$}
0
\ \to\
\begin{array}{@{}c@{}}\includegraphics[height=30mm]{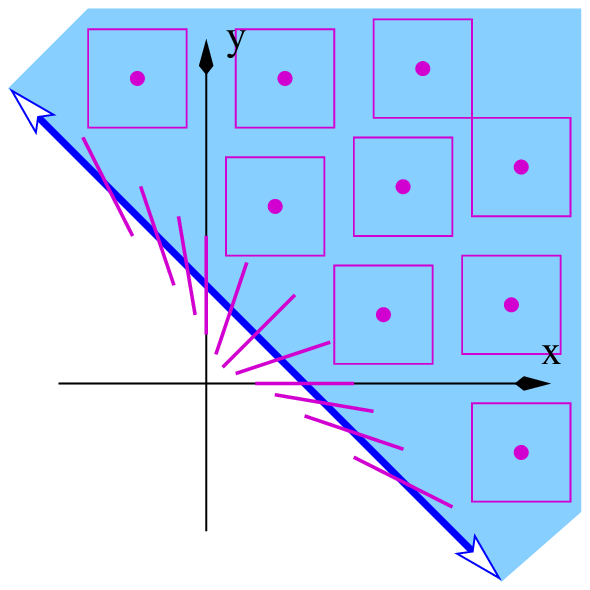}\end{array}
\ \too\
\begin{array}{@{}c@{}}\includegraphics[height=30mm]{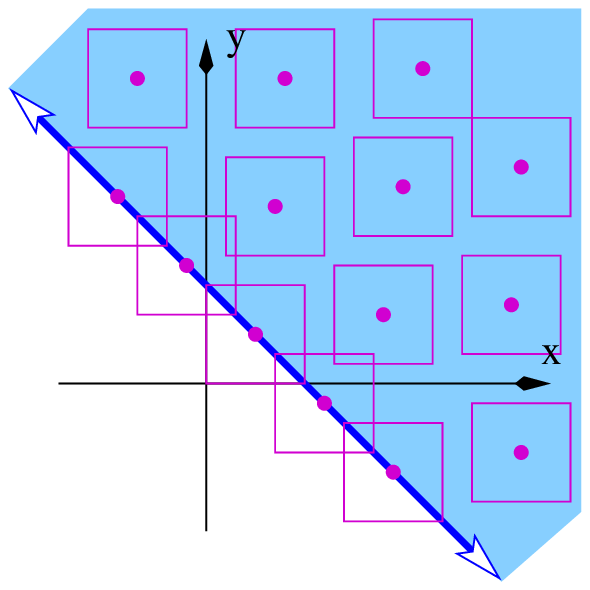}\end{array}
\ \too\
\begin{array}{@{}c@{}}\includegraphics[height=30mm]{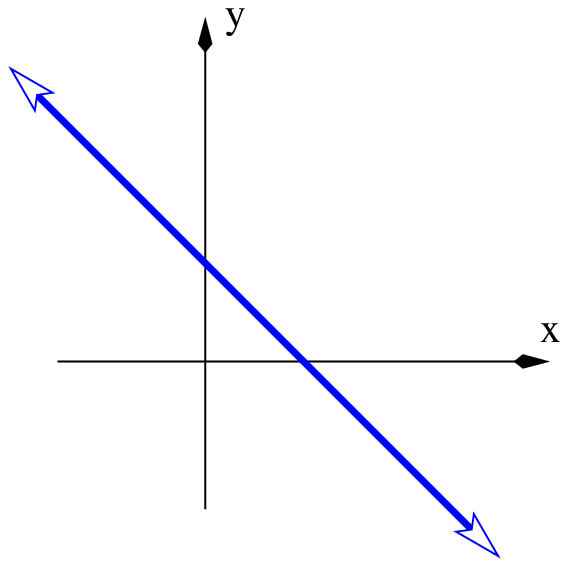}\end{array}
\ \to\
0
$$
That is, $K_\aa$ agrees with $M_\aa$ for degrees~$\aa$ outside of~$L$,
and $K_\aa$ is the line in~$M_\aa$ of slope $b/a$ through the origin
when $\aa$ lies on~$L$.  This kernel~$K$ is not tame.  Indeed, if
$\aa$ and~$\aa'$ are distinct points on~$L$, then the homomorphisms
$K_\aa \to K_{\aa\vee\aa'}$ and $K_{\aa'} \to K_{\aa\vee\aa'}$ have
different images, so $\aa$ and~$\aa'$ are forced to lie in different
constant regions in every constant subdivision of~$\RR^2$ subordinate
to~$K$.  (Note the relation between this example and
Proposition~\ref{p:U->D}.\ref{i:U->D} for $Q = U \subset \RR^2$ and $D
= L \subset Q$.)
\end{example}

\begin{remark}\label{r:lurie}
Encoding of a $\cQ$-module~$\cM$ by a poset morphism is related to
viewing~$\cM$ as a sheaf on~$\cQ$ with its Alexandrov topology that is
constructible in the sense of Lurie \cite[Definitions~A.5.1
and~A.5.2]{lurie2017}.
The difference is that poset encoding requires constancy (in the sense
of Definition~\ref{d:constant-subdivision}) on fibers of the encoding
morphism, whereas Alexandrov constructibility requires only local
constancy in the sense of sheaf theory.
This distinction is decisive for Example~\ref{e:ker}, where the
kernel~$K$ is constructible but not finitely encoded.
\end{remark}

Because of Remark~\ref{r:lurie}, allowing arbitrary homomorphisms
between tame modules would step outside of the tame class.  More
formally, inside the category of $\cQ$-modules, the full subcategory
generated by the tame modules contains modules that are not tame.  To
preserve tameness, it is thus necessary to restrict the allowable
morphisms.

\begin{defn}\label{d:tame-morphism}
A homomorphism $\phi: M \to N$ of $\cQ$-modules is \emph{tame} if
$\cQ$ admits a finite constant subdivision subordinate to both $M$
and~$N$ such that for each constant region~$I$ the composite
isomorphism $M_I \to M_\ii \to N_\ii \to N_I$ does not depend~on~$\ii
\in I$.  The map $\phi$ is semialgebraic, PL, subanalytic, or
class~$\mathfrak X$ if this constant subdivision~is.
\end{defn}

\begin{lemma}\label{l:ker-coker}
The kernel and cokernel of any tame homomorphism of $\cQ$-modules are
tame morphisms of tame modules.  The same is true when tameness is
replaced by semialgebraic, PL, subanalytic, or class~$\mathfrak X$.
\end{lemma}
\begin{proof}
Any constant subdivision as in Definition~\ref{d:tame-morphism} is
subordinate to both the kernel and cokernel of $M \to N$, with the
vector spaces assocated to any constant region~$I$ being $\ker(M_I \to
N_I)$ and $\coker(M_I \to N_I)$.
\end{proof}

\begin{defn}\label{d:abelian-category}
The \emph{category of tame modules} is the subcategory of
$\cQ$-modules whose objects are the tame modules and whose morphisms
are the tame homomorphisms.
\end{defn}

\begin{remark}\label{r:morphism}
To be precise with language, a \emph{morphism} of tame modules is
required to be tame, whereas a \emph{homomorphism} of tame modules is
not.  That is, morphisms in the category of tame modules are called
morphisms, whereas morphisms in the category of $\cQ$-modules are
called homomorphisms.  To avoid confusion, the set of tame morphisms
from a tame module~$M$ to another tame module~$N$ is denoted
$\Mor(M,N)$ instead of $\Hom(M,N)$.
\end{remark}

\begin{prop}\label{p:abelian-category}
Over any poset~$\cQ$, the category of tame $\cQ$-modules is~abelian.
If~$\cQ$ is a subposet of a partially ordered real vector space of
finite dimension, then the category of semialgebraic, PL, subanalytic,
or class~$\mathfrak X$ modules is~abelian.
\end{prop}
\begin{proof}
Over any poset, the category in question is a subcategory of the
category of $\cQ$-modules, which is abelian.  The subcategory is not
full, but $\Mor(M,N)$ is an abelian subgroup of $\Hom(M,N)$; this is
most easily seen via Theorem~\ref{t:tame}, for if $\phi: M \to N$
and $\phi': M \to N'$ are finitely encoded by $\pi: \cQ \to P$ and
$\pi': \cQ \to P'$, respectively, then $\phi + \phi'$ is finitely
encoded by $\pi \times \pi': \cQ \to P \times P'$.  The same
construction, but with the source of~$\pi'$ being a new module~$M'$
instead of~$M$, shows that the ordinary product and direct sum of a
pair of finitely encoded modules serves as a product and coproduct in
the tame category.  Kernels and cokernels of morphisms in the tame
category exist by Lemma~\ref{l:ker-coker}, which also implies that
every monomorphism is a kernel (it is the kernel of its cokernel in
the category of $\cQ$-modules) and every epimorphism is a cokernel (it
is the cokernel of its kernel in the category of $\cQ$-modules).

The semialgebraic, PL, and class~$\mathfrak X$ cases have the same
proof, noting that $\pi \times \pi'$ has fibers of the desired type if
$\pi$ and~$\pi'$ both do.  The subanalytic case only follows from this
argument when restricted to the category of modules whose nonzero
graded pieces lie in a bounded subset of~$\cQ$ (the subset is allowed
to depend on the module).  However, the argument in the previous
paragraph can be done directly with common refinements of pairs of
constant subdivisions, so reducing to Theorem~\ref{t:tame} is not
necessary.
\end{proof}

\section{Finitely determined \texorpdfstring{$\ZZ^n$}{Zn}-modules}\label{s:ZZn}

Unless otherwise stated, this section is presented over the discrete
partially ordered group $\cQ = \ZZ^n$ with $\cQ_+ = \NN^n$.  It begins
by reviewing the structure of finitely determined $\ZZ^n$-modules
(Section~\ref{sub:def-finitely-determined}), including (minimal)
injective and flat resolutions (Sections~\ref{sub:inj}
and~\ref{sub:flat}), before getting to flange presentations
(Section~\ref{sub:flange}) and the syzygy theorem
(Section~\ref{sub:Zsyzygy}).  These latter two underlie the general
syzygy theorem (Section~\ref{sub:syzygy}), including existence of
fringe presentations.  They also serve as models for the concepts of
socle, cogenerator, and downset hull over real polyhedral groups,
covered in the sequel \cite{essential-real}, as well as dual notions
of top, generator, and upset~cover.

The main references for $\ZZ^n$-modules used here are
\cite{alexdual,cca}.  The development of homological theory for
injective and flat resolutions in the context of finitely determined
modules is functorially equivalent to that for finitely generated
modules, by \cite[Theorem~2.11]{alexdual}, but it is convenient to
have on hand statements in the finitely determined case.  Flange
presentation (Section~\ref{sub:flange}) and the characterization of
finitely determined modules in Proposition~\ref{p:determined} and
(hence) Theorem~\ref{t:finitely-determined} are apparently new.

\subsection{Definitions}\label{sub:def-finitely-determined}\mbox{}

\noindent
The essence of the finiteness here is that all of the relevant
information about the relevant modules should be recoverable from what
happens in a bounded box in~$\ZZ^n$.

\begin{defn}\label{d:determined}
A $\ZZ^n$-finite module~$\cN$ is \emph{finitely determined} if for
each $i = 1,\dots,n$ the multiplication map $\cdot x_i: N_\bb \to
N_{\bb+\ee_i}$ is an isomorphism whenever $b_i$ lies outside of some
bounded interval.  For notation, $\kk[\NN^n] = \kk[\xx]$, where $\xx =
x_1,\dots,x_n$ is a sequence of variables, and $\ee_i$ is the standard
basis vector whose only nonzero entry is $1$ in slot~$i$.
\end{defn}

\begin{remark}\label{r:determined}
This notion of finitely determined is the same notion as in
Example~\ref{e:convex-projection}.  A module is finitely determined if
and only if, after perhaps translating its $\ZZ^n$-grading, it is
\emph{$\aa$-determined} for some $\aa \in \NN^n$, as defined in
\cite[Definition~2.1]{alexdual}.
\end{remark}

\begin{remark}\label{r:fg/ZZ^n}
For $\ZZ^n$-modules, the finitely determined condition is
weaker---that is, more inclusive---than finitely generated, but it is
much stronger than tame or (equivalently, by Theorem~\ref{t:tame})
finitely encoded.  The reason is essentially
Example~\ref{e:convex-projection}, where the encoding has a very
special nature.  For a generic sort of example, the restriction
to~$\ZZ^n$ of any $\RR^n$-finite $\RR^n$-module with finitely many
constant regions of sufficient width is a tame $\ZZ^n$-module, and
there is simply no reason why the constant regions should be
commensurable with the coordinate directions in~$\ZZ^n$.  Already the
toy-model fly wing modules in Examples~\ref{e:toy-model-fly-wing}
and~\ref{e:encoding} yield infinitely generated but tame
$\ZZ^n$-modules, and this remains true when the discretization $\ZZ^n$
of~$\RR^n$ is rescaled by any factor.
\end{remark}

\begin{example}\label{e:local-cohomology}
The local cohomology of an affine semigroup ring is tame but usually
not finitely determined; see \cite{injAlg} and \cite[Chapter~13]{cca},
particularly Theorem~13.20, Example~13.17, and Example~13.4 in the
latter.
\end{example}

\subsection{Injective hulls and resolutions}\label{sub:inj}\mbox{}

\begin{remark}\label{r:injective}
Every $\ZZ^n$-finite module that is injective in the category of
$\ZZ^n$-modules is, by \cite[Theorem~11.30]{cca}, a direct sum of
downset modules $\kk[D]$ for \emph{coprincipal} downsets $D = \aa +
\tau - \NN^n$, said to be \emph{cogenerated} by~$\aa$ \emph{along} the
face~$\tau$ of~$\NN^n$.  Note that faces of~$\NN^n$ correspond to
subsets of~$[n] = \{1,\dots,n\}$ via $\tau \leftrightarrow \{i \in [n]
\mid \ee_i \in \tau\}$.
\end{remark}

Minimal injective resolutions work for finitely determined modules
just as they do for finitely generated modules.  The standard
definitions are as follows.

\begin{defn}\label{d:inj}
Fix a $\ZZ^n$-module~$\cN$.
\begin{enumerate}
\item%
An \emph{injective hull} of~$\cN$ is an injective homomorphism $\cN
\to E$ in which $E$ is an injective $\ZZ^n$-module (see
Remark~\ref{r:injective}).  This injective hull is
\begin{itemize}
\item%
\emph{finite} if $E$ has finitely many indecomposable summands and
\item%
\emph{minimal} if the number of such summands is minimal.
\end{itemize}
\item%
An \emph{injective resolution} of~$\cN$ is a complex~$E^\spot$ of
injective $\ZZ^n$-modules whose differential $E^i \to E^{i+1}$ for $i
\geq 0$ has only one nonzero homology $H^0(E^\spot) \cong\nolinebreak
\cN$ (so $\cN \into E^0$ and $\coker(E^{i-1} \to E^i) \into E^{i+1}$
are injective hulls for all $i \geq 1$).  \nolinebreak$E^\spot$
\begin{itemize}
\item%
has \emph{length~$\ell$} if $E^i = 0$ for $i > \ell$ and $E^\ell \neq
0$;
\item%
is \emph{finite} if $E^\spot = \bigoplus_i E^i$ has finitely many
indecomposable summands; and
\item%
is \emph{minimal} if $\cN \into E^0$ and $\coker(E^{i-1} \to E^i)
\into E^{i+1}$ are minimal injective hulls for all $i \geq 1$.
\end{itemize}
\end{enumerate}
\end{defn}

\begin{prop}\label{p:determined}
The following are equivalent for a $\ZZ^n$-module~$\cN$.
\begin{enumerate}
\item%
$\cN$ is finitely determined.
\item%
$\cN$ admits a finite injective resolution.
\item%
$\cN$ admits a finite minimal injective resolution.
\end{enumerate}
Any finite minimal resolution is unique up to isomorphism and has
length~at~most~$n$.
\end{prop}
\begin{proof}
The proof is based on existence of finite minimal injective hulls and
resolutions for finitely generated $\ZZ^n$-modules, along with
uniqueness and length~$n$ given minimality, as proved by Goto and
Watanabe \cite{goto-watanabe1978}.

First assume $\cN$ is finitely determined.  Translating the
$\ZZ^n$-grading affects nothing about existence of a finite injective
resolution.  Therefore, using Remark~\ref{r:determined}, assume that
$\cN$ is $\aa$-determined.  Truncate by taking the $\NN^n$-graded part
of~$\cN$ to get a positively $\aa$-determined---and hence finitely
generated---module~$\cN_{\succeq\0}$; see
\cite[Definition~2.1]{alexdual}.  Take any minimal injective
resolution $\cN_{\succeq\0} \to E^\spot$.  Extend backward using the
\v Cech hull \cite[Definition~2.7]{alexdual}, which is exact
\cite[Lemma~2.9]{alexdual}, to get a finite minimal injective
resolution $\vC(\cN_{\succeq\0} \to E^\spot) = (\cN \to \vC E^\spot)$,
noting that $\vC$ fixes indecomposable injective modules whose
$\NN^n$-graded parts are nonzero and is zero on all other
indecomposable injective modules \cite[Lemma~4.25]{alexdual}.  This
proves 1~$\implies$~3.

That 3 $\implies$~2 is trivial.  The remaining implication,
2~$\implies$~1, follows because every indecomposable injective is
finitely determined and the category of finitely determined modules is
abelian.  (The category of $\ZZ^n$-modules each of which is nonzero
only in a bounded set of degrees is abelian, and constructions such as
kernels, cokernels, or direct sums in the category of finitely
determined modules are pulled back from there.)
\end{proof}

\subsection{Flat covers and resolutions}\label{sub:flat}\mbox{}

\noindent
Minimal flat resolutions are not commonplace, but the notion is Matlis
dual to that of minimal injective resolution.  In the context of
finitely determined modules, flat resolutions work as well as
injective resolutions.  The definitions are as follows.

\begin{defn}\label{d:flat}
Fix a $\ZZ^n$-module~$\cN$.
\begin{enumerate}
\item%
A \emph{flat cover} of~$\cN$ is a surjective homomorphism $F \to \cN$
in which $F$ is a flat $\ZZ^n$-module (see Remark~\ref{r:flat}).  This
flat cover is
\begin{itemize}
\item%
\emph{finite} if $F$ has finitely many indecomposable summands and
\item%
\emph{minimal} if the number of such summands is minimal.
\end{itemize}
\item%
A \emph{flat resolution} of~$\cN$ is a complex~$F_\spot$ of flat
$\ZZ^n$-modules whose differential $F_{i+1} \to F_i$ for $i \geq 0$
has only one nonzero homology $H_0(F_\spot) \cong \cN$ (so $F_0 \onto
\cN$ and $F_{i+1} \onto \ker(F_i \to F_{i-1})$ are flat covers for all
$i \geq 1$).  The flat resolution~$F_\spot$
\begin{itemize}
\item%
has \emph{length~$\ell$} if $F_i = 0$ for $i > \ell$ and $F_\ell \neq
0$;
\item%
is \emph{finite} if $F_\spot = \bigoplus_i F_i$ has finitely many
indecomposable summands; and
\item%
is \emph{minimal} if $F_0 \onto \cN$ and $F_{i+1} \onto \ker(F_i \to
F_{i-1})$ are minimal flat covers for all $i \geq 1$.
\end{itemize}
\end{enumerate}
\end{defn}

\begin{defn}\label{d:matlis}
The \emph{Matlis dual} of a $\ZZ^n$-module~$\cM$ is the
$\ZZ^n$-module~$\cM^\vee$ defined by
$$%
  (\cM^\vee)_\aa = \Hom_\kk(\cM_{-\aa},\kk),
$$
so the homomorphism $(\cM^\vee)_\aa \to (M^\vee)_\bb$ is transpose to
$\cM_{-\bb} \to \cM_{-\aa}$.
\end{defn}

\begin{lemma}\label{l:vee-vee}
$(\cM^\vee)^\vee\!$ is canonically isomorphic to~$\cM$ for any
$\ZZ^n$-finite module~$\cM$.~\hspace{1ex}$\square$
\end{lemma}

\begin{remark}\label{r:flat}
By the adjunction between Hom and $\otimes$, a module is flat if and
only its Matlis dual is injective (see \cite[\S1.2]{alexdual}, for
example).  The Matlis dual of Remark~\ref{r:injective} therefore says
that every $\ZZ^n$-finite flat $\ZZ^n$-module is isomorphic to a
direct sum of upset modules~$\kk[U]$ for upsets of the form $U = \bb -
\tau + \NN^n = \bb + \NN^n + \ZZ\tau$.  These upset modules are the
graded translates of localizations of~$\kk[\NN^n]$ along faces.
\end{remark}

\subsection{Flange presentations}\label{sub:flange}\mbox{}

\begin{defn}\label{d:flange}
Fix a $\ZZ^n$-module~$\cN$.
\begin{enumerate}
\item%
A \emph{flange presentation} of~$\cN$ is a $\ZZ^n$-module morphism
$\phi: F \to E$, with image isomorphic to~$\cN$, where $F$ is flat and
$E$ is injective in the category of \mbox{$\ZZ^n$-modules}.
\item%
If $F$ and~$E$ are expressed as direct sums of indecomposables, then
$\phi$ is \emph{based}.
\item%
If $F$ and~$E$ are finite direct sums of indecomposables, then $\phi$
is \emph{finite}.
\item%
If the number of indecomposable summands of~$F$ and~$E$ are
simultaneously minimized then $\phi$ is \emph{minimal}.
\end{enumerate}
\end{defn}

\begin{remark}\label{r:portmanteau-fl}
The term \emph{flange} is a portmanteau of \emph{flat} and
\emph{injective} (i.e., ``flainj'') because a flange presentation is
the composite of a flat cover and an injective hull.
\end{remark}

The same notational trick to make fringe presentations effective data
structures (Definition~\ref{d:monomial-matrix-fr}) works on flange
presentations.

\begin{defn}\label{d:monomial-matrix-fl}
Fix a based finite flange presentation $\phi:
\bigoplus_p\hspace{-.2pt} F_p = F \to E =
\nolinebreak\bigoplus_q\hspace{-.2pt} E_q$.  A \emph{monomial matrix}
for $\phi$ is an array of \emph{scalar entries}~$\phi_{qp}$ whose
columns are labeled by the indecomposable flat summands~$F_p$ and
whose rows are labeled by the indecomposable injective summands~$E_q$:
$$%
\begin{array}{ccc}
  &
  \monomialmatrix
	{F_1\\\vdots\ \\F_k\\}
	{\begin{array}{ccc}
		   E_1    & \cdots &    E_\ell   \\
		\phi_{11} & \cdots & \phi_{1\ell}\\
		\vdots    & \ddots &   \vdots    \\
		\phi_{k1} & \cdots & \phi_{k\ell}\\
	 \end{array}}
	{\\\\\\}
\\
  F_1 \oplus \dots \oplus F_k = F
  & \fillrightmap
  & E = E_1 \oplus \dots \oplus E_\ell.
\end{array}
$$
\end{defn}

The entries of the matrix $\phi_{\spot\spot}$ correspond to
homomorphisms $F_p \to E_q$.

\begin{lemma}\label{l:F->E}
If $F = \kk[\aa + \ZZ\tau' + \NN^n]$ is an indecomposable flat
$\ZZ^n$-module and $E = \kk[\bb + \ZZ\tau - \NN^n]$ is an
indecomposable injective $\ZZ^n$-module, then $\Hom_{\ZZ^n}(F, E) = 0$
unless $(\aa + \ZZ\tau' + \NN^n) \cap (\bb + \ZZ\tau - \NN^n) \neq
\nothing$, in which case $\Hom_{\ZZ^n}(F, E) = \kk$.
\end{lemma}
\begin{proof}
Corollary~\ref{c:U->D}.\ref{i:kk}.
\end{proof}

\begin{defn}\label{d:F<E}
In the situation of Lemma~\ref{l:F->E}, write $F \preceq E$ if their
degree sets have nonempty intersection: $(\aa + \ZZ\tau' + \NN^n) \cap
(\bb + \ZZ\tau - \NN^n) \neq \nothing$.
\end{defn}

\begin{prop}\label{p:scalars-fl}
With notation as in Definition~\ref{d:monomial-matrix-fl}, $\phi_{pq} =
0$ unless $F_p \preceq E_q$.  Conversely, if an array of scalars
$\phi_{qp} \in \kk$ with rows labeled by indecomposable flat modules
and columns labeled by indecomposable injectives has $\phi_{pq} = 0$
unless $F_q \preceq E_q$, then it represents a flange presentation.
\end{prop}
\begin{proof}
Lemma~\ref{l:F->E} and Definition~\ref{d:F<E}.
\end{proof}

The unnatural hypothesis that a persistence module be finitely
generated results in data types and structure theory that are
asymmetric regarding births as opposed to deaths.  In contrast, the
notion of flange presentation is self-dual: their duality interchanges
the roles of births~($F$) and deaths~($E$).

\begin{prop}\label{p:duality}
A $\ZZ^n$-module $\cN$ has a finite flange presentation $F \to E$ if
and only if the Matlis dual $E^\vee \to F^\vee$ is a finite flange
presentation of the Matlis dual $\cN^\vee$.
\end{prop}
\begin{proof}
Matlis duality is an exact, contravariant functor on~$\ZZ^n$-modules
that takes the subcategory of finitely determined $\ZZ^n$-modules to
itself (these properties are immediate from the definitions),
interchanges flat and injective objects therein, and has the property
that the natural map $(\cN^\vee)^\vee \to \cN$ is an isomorphism for
finitely determined~$\cN$ (Lemma~\ref{l:vee-vee}); see
\cite[\S1.2]{alexdual} for a discussion of these properties.
\end{proof}

\subsection{Syzygy theorem for \texorpdfstring{$\ZZ^n$}{Zn}-modules}\label{sub:Zsyzygy}

\begin{thm}\label{t:finitely-determined}
A $\ZZ^n$-module is finitely determined if and only if it admits one,
and hence all, of the following:
\begin{enumerate}
\item\label{i:flange}%
a finite flange presentation; or
\item\label{i:flat-presentation}%
a finite flat presentation; or
\item\label{i:injective-copresentation}%
a finite injective copresentation; or
\item\label{i:flat-res}%
a finite flat resolution; or
\item\label{i:injective-res}%
a finite injective resolution; or
\item\label{i:minimal}%
a minimal one of any of the above.
\end{enumerate}
Any minimal one of these objects is unique up to noncanonical
isomorphism, and the resolutions have length at most~$n$.
\end{thm}
\begin{proof}
The hard work is done by Proposition~\ref{p:determined}.  It implies
that $\cN$ is finitely determined $\iff \cN^\vee$ has a minimal
injective resolution $\iff \cN$ has a minimal flat resolution of
length at most~$n$, since the Matlis dual of any finitely determined
module~$\cN$ is finitely determined.  Having both a minimal injective
resolution and a minimal flat resolution is stronger than having any
of items~\ref{i:flange}--\ref{i:injective-copresentation}, minimal or
otherwise, so it suffices to show that $\cN$ is finitely determined if
$\cN$ has any of
items~\ref{i:flange}--\ref{i:injective-copresentation}.  This follows,
using that the category of finitely determined modules~is~abelian as
in the proof of Proposition~\ref{p:determined}, from the fact that
every indecomposable injective or flat $\ZZ^n$-module is finitely
determined.
\end{proof}

\begin{remark}\label{r:finitely-determined}
Conditions~\ref{i:flange}--\ref{i:minimal} in
Theorem~\ref{t:finitely-determined} remain equivalent for
$\RR^n$-modules, with the standard positive cone $\RR^n_+$, assuming
that the finite flat and injective modules in question are finite
direct sums of localizations of~$\RR^n$ along faces and their Matlis
duals.  (The equivalence, including minimality, is a consequence of
the generator and cogenerator theory over real polyhedral groups
\cite{essential-real}.)  The equivalent conditions do not characterize
$\RR^n$-modules that are pulled back under convex projection from
arbitrary modules over an interval in~$\RR^n$, though, because all
sorts of infinite things can can happen inside of a box, such as
having generators~along~a~curve.
\end{remark}

\section{Homological algebra of poset modules}\label{s:syzygy}

\subsection{Indicator resolutions}\label{sub:indicator-res}

\begin{defn}\label{d:resolutions}
Fix any poset~$\cQ$ and a $\cQ$-module~$\cM$.
\begin{enumerate}
\item%
An \emph{upset resolution} of~$\cM$ is a complex~$F_\spot$ of
$\cQ$-modules, each a direct sum of upset submodules of~$\kk[\cQ]$,
whose differential $F_i \to F_{i-1}$ decreases homological degrees,
has components $\kk[U] \to \kk[U']$ that are connected
(Definition~\ref{d:connected-homomorphism}), and has only one nonzero
homology $H_0(F_\spot) \cong \cM$.

\item%
A \emph{downset resolution} of~$\cM$ is a complex~$E^\spot$ of
$\cQ$-modules, each a direct sum of downset quotient modules
of~$\kk[\cQ]$, whose differential $E^i \to E^{i+1}$ increases
cohomological degrees, has components $\kk[D'] \to \kk[D]$ that are
connected, and has only one nonzero homology $H^0(E^\spot)
\cong\nolinebreak \cM$.
\end{enumerate}\setcounter{separated}{\value{enumi}}
An upset or downset resolution is called an \emph{indicator
resolution} if the ``up-'' or ``down-'' nature is unspecified.  The
\emph{length} of an indicator resolution is the largest
(\hspace{-.1pt}co\hspace{-.1pt})homological degree in which the
complex is nonzero.  An indicator resolution
\begin{enumerate}\setcounter{enumi}{\value{separated}}
\item%
is \emph{finite} if the number of indicator module summands is finite,

\item%
\emph{dominates} a constant subdivision or encoding of~$\cM$ if the
subdivision or encoding is subordinate to each indicator summand, and

\item\label{i:auxiliary-resolution}%
is \emph{semialgebraic}, \emph{PL}, \emph{subanalytic}, or \emph{of
class~$\mathfrak X$} if $\cQ$ is a subposet of a real partially
ordered group and the resolution dominates a constant subdivision or
encoding of the corresponding type.
\end{enumerate}
\end{defn}

\begin{defn}\label{d:resolution-monomial-matrix}
Monomial matrices for indicator resolutions are defined similarly to
those for fringe presentations in
Definition~\ref{d:monomial-matrix-fr}, except that for the
cohomological case the row and column labels are source and target
downsets, respectively, while in the homological case the row and
column labels are target and source upsets, respectively:
$$%
\begin{array}{ccc}
  &
  \monomialmatrix
	{\vdots\ \\D^i_p\\\vdots\ }
	{\begin{array}{ccc}
		 \cdots & D^{i+1}_q & \cdots \\
		        &           &        \\
		        & \phi_{pq} &        \\
		        &           &        \\
	 \end{array}}
	{\\\\\\}
\\
    E^i
  & \fillrightmap
  & E^{i+1}
\end{array}
\qquad\text{and}\qquad
\begin{array}{ccc}
  &
  \monomialmatrix
	{\vdots\ \\U_i^p\\\vdots\ }
	{\begin{array}{ccc}
		 \cdots & U_{i+1}^q & \cdots \\
		        &           &        \\
		        & \phi_{pq} &        \\
		        &           &        \\
	 \end{array}}
	{\\\\\\}
\\
    F_i
  & \filleftmap
  & F_{i+1}.
\end{array}
$$
(Note the switch of source and target from cohomological to
homological, so the map goes from right to left in the homological
case, with decreasing homological indices.)
\end{defn}

As in Proposition~\ref{p:pullback-monomial-matrix}, pullbacks have
transparent monomial matrix interpretations.

\begin{prop}\label{p:syzygy-monomial-matrix}
Fix a poset~$\cQ$ and an encoding of a $\cQ$-module~$\cM$ by a poset
morphism $\pi: \cQ \to \cP$ and $\cP$-module~$\cH$.  Monomial matrices
for any indicator resolution of~$\cH$ pull back to monomial matrices
for an indicator resolution of~$\cM$ that dominates the encoding by
replacing the row and column labels with their preimages under~$\pi$.
\hfill$\square$
\end{prop}

\begin{defn}\label{d:indicator-(co)presentation}
Fix any poset~$\cQ$ and a $\cQ$-module~$\cM$.
\begin{enumerate}
\item%
An \emph{upset presentation} of~$\cM$ is an expression of~$\cM$ as the
cokernel of a homomorphism $F_1 \to F_0$ such that each $F_i$ is a
direct sum of upset modules and every component $\kk[U'] \to \kk[U]$
of the homomorphism is connected
(Definition~\ref{d:connected-homomorphism}).

\item%
A \emph{downset copresentation} of~$\cM$ is an expression of~$\cM$ as
the kernel of a homomorphism $E^0 \to E^1$ such that each $E^i$ is a
direct sum of downset modules and every component $\kk[D] \to \kk[D']$
of the homomorphism is connected.
\end{enumerate}\setcounter{separated}{\value{enumi}}
These \emph{indicator presentations} are \emph{finite}, or
\emph{dominate} a constant subdivision or encoding of~$\cM$, or are
\emph{semialgebraic}, \emph{PL}, \emph{subanalytic}, or \emph{of
class~$\mathfrak X$} as in Definition~\ref{d:resolutions}.
\end{defn}

\begin{example}\label{e:one-param-upset}
In one parameter, the bar $[a,b)$ in Example~\ref{e:one-param-fringe},
has upset presentation
$$%
\hspace{8ex}
\psfrag{a}{\footnotesize\raisebox{-.2ex}{$a$}}
\psfrag{b}{\footnotesize\raisebox{-.2ex}{$b$}}
\psfrag{vert-to}{\small\raisebox{-.2ex}{$\downarrow$}}
\psfrag{vert-into}{\small\raisebox{-.2ex}{$\lhookdownarrow$}}
\psfrag{vert-onto}{\small\raisebox{-.2ex}{$\twoheaddownarrow$}}
\psfrag{has image}{}
\begin{array}{@{}c@{}}
\includegraphics[height=20mm]{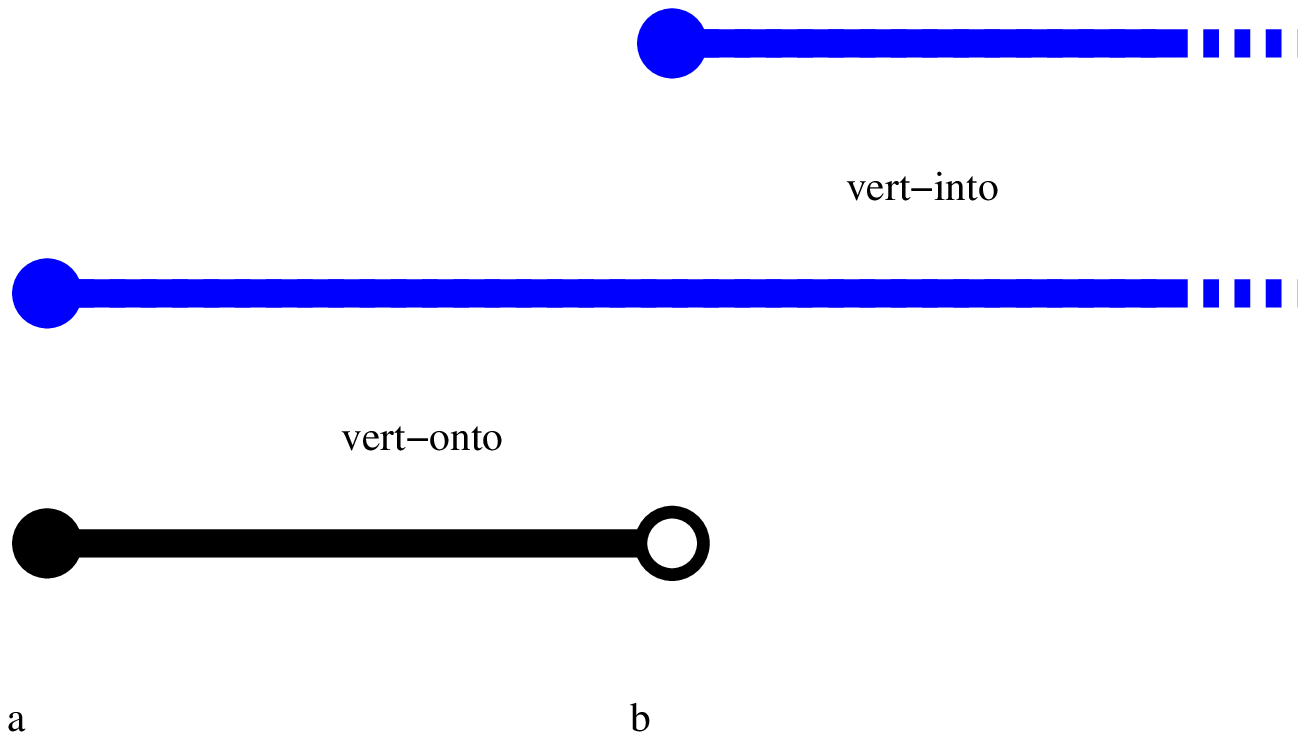}\\[-5ex]
\llap{with cokernel\hspace{10ex}}\\[2ex]
\end{array}
$$
isomorphic to the single bar.  When there are multiple bars, the
bijection from left to right endpoints yields a monomial matrix whose
scalar entries again form an identity matrix, with rows labeled by
positive rays having the specified left endpoints (the ray is the
whole real line when the left endpoint is~$-\infty$) and columns
labeled by positive rays having the corresponding right
endpoints---but with their open or closed nature reversed---as left
endpoints (the ray is empty when the specified right
endpoint~is~$+\infty$).
\end{example}

\begin{example}\label{e:two-param-upset}
$$%
\begin{array}{c}
\\[-1.48ex]
\begin{array}{@{}r@{\hspace{-.4pt}}|@{}l@{}}
\includegraphics[height=25mm]{semialgebraic}&\ \,\hspace{-.3pt}\\[-4.2pt]\hline
\end{array}
\\[-1.52ex]\mbox{}
\end{array}
\ \
\begin{array}{@{}c@{}}
\text{is the cokernel of}\\[2ex]
\end{array}
\qquad
\begin{array}{@{}r@{\hspace{-6.3pt}}|@{}l@{}}
\raisebox{-.2mm}{\includegraphics[height=25mm]{upset-blue}}
&\ \,\\[-4pt]\hline
\end{array}
\,\otni\!
\begin{array}{@{}r@{\hspace{-5.6pt}}|@{}l@{}}
\raisebox{-.4mm}{\includegraphics[height=25mm]{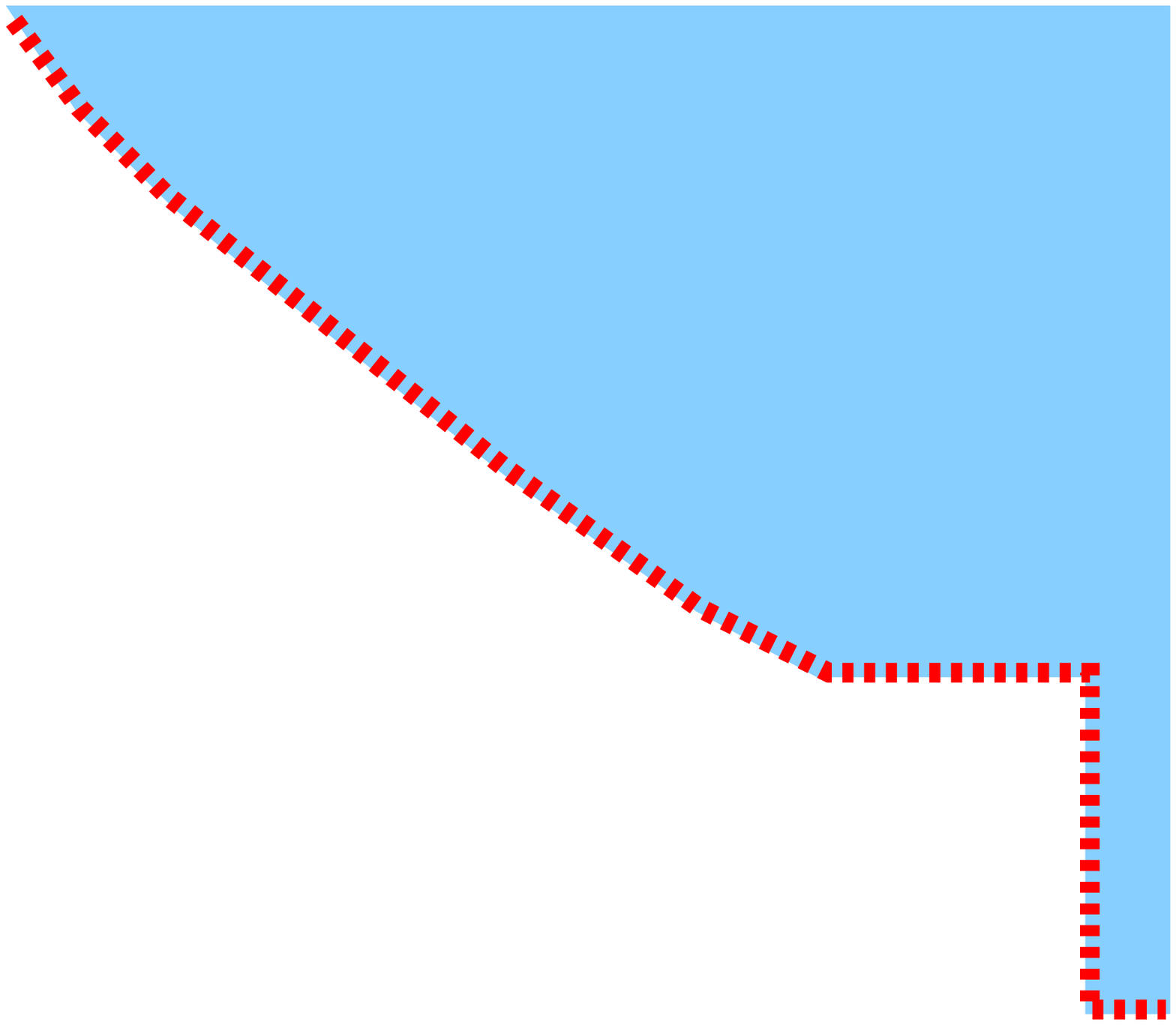}}
&\ \,\hspace{-.2pt}\\[-4pt]\hline
\end{array}
$$
\end{example}

\begin{lemma}\label{l:morphisms}
The homomorphisms in indicator presentations and resolutions are tame,
so their kernels and cokernels are tame.  If the indicator modules in
question are semialgebraic, PL, subanalytic, or of class~$\mathfrak X$
then the morphisms are, as well.
\end{lemma}
\begin{proof}
Any connected homomorphism among indicator modules is tame---and
satisfies one of the auxiliary hypotheses, if the source and target
do---by Definition~\ref{d:tame-morphism}, so the conclusion follows
from Proposition~\ref{p:abelian-category}.
\end{proof}

\begin{example}\label{e:puuska-nonconstant-isotypic'}
The poset module in Example~\ref{e:puuska-nonconstant-isotypic} has an
upset presentation
$$%
\begin{array}{ccc}
  &
  \monomialmatrix
	{L\\R\\}
	{\begin{array}{rr}
		 T &  B \\
		 2 &  1 \\
		-1 & -1 \\
	 \end{array}}
	{\\\\}
\\
  \kk[L] \oplus \kk[R]
  & \filleftmap
  & \kk[T]
\end{array}
$$
in which the monomial matrix has row and column labels
\begin{itemize}
\item%
$L$, the upset generated by the leftmost element;
\item%
$R$, the upset generated by the rightmost element;
\item%
$T$, the upset consisting solely of the maximal element depicted on
top; and
\item%
$B$, the upset consisting solely of the maximal element depicted on
the bottom.
\end{itemize}
Although the disjoint union of $T$ and~$B$ is an upset, and there is a
homomorphism $\phi: \kk[T \cup B] \to \kk[L] \oplus \kk[R]$ whose
cokernel is the desired poset module, there is no way to arrange for
the homomorphism~$\phi$ to be connected.
\end{example}

\begin{remark}\label{r:augmentation}
It is tempting to think that a fringe presentation is nothing more
than the concatenation of the augmentation map of an upset resolution
(that is, the surjection at the end) with the augmentation map of a
downset resolution (that is, the injection at the beginning), but
there is no guarantee that the components $F_i \to E_j$ of the
homomorphism thus produced are connected
(Definition~\ref{d:connected-homomorphism}).  In contrast, a flange
presentation (Definition~\ref{d:flange}) is in fact nothing more than
the concatenation of the augmentation maps of a flat resolution and an
injective resolution, since connected homomorphisms are forced by
Lemma~\ref{l:F->E}.
\end{remark}

\subsection{Syzygy theorem for modules over posets}\label{sub:syzygy}

\begin{prop}\label{p:pushforward}
For any inclusion $\iota: \cP \to \cZ$ of posets and
$\cP$-module~$\cH$ there is a $\cZ$-module $\iota_*\cH$, the
\emph{pushforward to~$\cZ$}, whose restriction to~$\iota(\cP)$
is~$\cH$ and is universally repelling: $\iota_*\cH$ has a canonical
map to every $\cZ$-module whose restriction to~$\iota(\cP)$ is~$\cH$.
\end{prop}
\begin{proof}
At $z \in \cZ$ the pushforward $\iota_*\cH$ places the colimit
$\dirlim\cH_{\preceq z}$ of the diagram of vector spaces indexed by
the elements of~$\cP$ whose images precede~$z$.  The universal
pro\-perty of colimits implies that $\iota_*\cH$ is a $\cZ$-module
with the desired universal property.
\end{proof}

\begin{remark}\label{r:kan-extension}
With perspectives as in Remark~\ref{r:curry}, the pushforward is a
left Kan extension \cite[Remark~4.2.9]{curry-thesis}.  This instance
is a special case of~\cite[Example~4.4]{curry2019}.
\end{remark}

\pagebreak

\begin{thm}[Syzygy theorem]\label{t:syzygy}
A module~$\cM$ over a poset~$\cQ$ is tame if and only if it admits
one, and hence all, of the following:
\begin{enumerate}
\item\label{i:syzygy-tame}%
a finite constant subdivision of~$\cQ$ subordinate to~$\cM$; or

\item\label{i:syzygy-encoding}%
a finite poset encoding subordinate to~$\cM$; or

\item\label{i:fringe}%
a finite fringe presentation; or

\item\label{i:upset-presentation}%
a finite upset presentation; or

\item\label{i:downset-copresentation}%
a finite downset copresentation; or

\item\label{i:upset-res}%
a finite upset resolution; or

\item\label{i:downset-res}%
a finite downset resolution; or

\item\label{i:dominating}%
any of the above dominating any given finite encoding; or

\item\label{i:subordinate-encoding}%
a finite encoding subordinate to any given one of
items~\ref{i:syzygy-tame}--\ref{i:downset-res}; or

\item\label{i:subordinate-constant}%
a finite constant subdivision subordinate to any given one of
items~\ref{i:syzygy-tame}--\ref{i:downset-res}.
\end{enumerate}
The statement remains true over any subposet of a real partially
ordered group if ``tame'' and all occurrences of ``finite'' are
replaced by ``semialgebraic'', ``PL'', or ``class~$\mathfrak X$''.
Moreover, any tame or semialgebraic, PL, or class~$\mathfrak X$
morphism $\cM \to \cM'$ lifts to a similarly well behaved morphism of
presentations or resolutions as in
parts~\ref{i:fringe}--\ref{i:downset-res}.  All of these results
except item~\ref{i:subordinate-encoding} hold in the subanalytic case
if~$\cM$ has compact support.
\end{thm}
\begin{proof}
Tame is equivalent to item~\ref{i:syzygy-tame} without auxiliary
hypotheses by Definition~\ref{d:tame} and with auxiliary hypotheses by
Definition~\ref{d:auxiliary-hypotheses}.  Tame is equivalent to
item~\ref{i:syzygy-encoding} by
Theorem~\ref{t:tame}.\ref{i:admits-finite-encoding}.  With auxiliary
hypotheses, \ref{i:syzygy-tame} $\implies$ \ref{i:syzygy-encoding} by
Theorem~\ref{t:tame}.\ref{i:auxiliary-uptight}; to apply that result
in the subanalytic case starting from an arbitrary subanalytic finite
constant subdivision subordinate to~$\cM$, construct a compact such
subdivision by keeping the bounded constant regions as they are and
taking the union of all unbounded constant regions to get a single
unbounded one.  The implication \ref{i:syzygy-encoding} $\implies$
\ref{i:syzygy-tame} holds because the fibers of the encoding poset
morphism form a constant subdivision of the relevant type.

The necessity to construct an auxiliary compact subdivision from the
given one is the reason to exclude item~\ref{i:subordinate-encoding}
from the subanalytic case, as the upcoming argument produces constant
subdivisions, not directly encodings.  For all of the other cases,
item~\ref{i:subordinate-encoding} proceeds via
item~\ref{i:subordinate-constant}, given the uptight constructions in
the previous paragraph.  For item~\ref{i:subordinate-constant}, to
produce a subordinate finite constant subdivision given a finite
fringe presentation, take the common refinement of the canonical
constant subdivision subordinate to each of its indicator summands.
The same construction works if indicator presentations or resolutions
are given, and it preserves auxiliary hypotheses by
Proposition~\ref{p:auxiliary-hypotheses}.\ref{i:classes}.

What remains is item~\ref{i:dominating}: a finitely encoded
$\cQ$-module~$\cM$ has finite upset and downset resolutions and
(co)presentations, as well as a finite fringe presentation, all
dominating the given encoding.  (As noted in the first paragraph, the
fibers of the encoding morphism are already a constant subdivision of
the relevant type.)  The domination takes care of the cases with
auxiliary hypotheses by
Definitions~\ref{d:fringe}.\ref{i:auxiliary-fringe},
\ref{d:subordinate-encoding}.\ref{i:auxiliary-encoding},
\ref{d:resolutions}.\ref{i:auxiliary-resolution},
and~\ref{d:indicator-(co)presentation}.

Fix a $\cQ$-module~$\cM$ finitely encoded by a poset morphism $\pi:
\cQ \hspace{-.2ex}\to\hspace{-.2ex} \cP$ and
\mbox{$\cP$-module}~$\cH$.  The finite poset~$\cP$ has order
dimension~$n$ for some positive integer~$n$; as such $\cP$ has an
embedding $\iota: \cP \into \ZZ^n$.  The pushforward $\iota_*\cH$
(Proposition~\ref{p:pushforward}) is finitely determined
(Definition~\ref{d:determined}; see also
Example~\ref{e:convex-projection}) as it is pulled back from any box
containing~$\iota(\cP)$.  The desired presentation or resolution is
pulled back to~$\cQ$ (via $\iota \circ \pi: \cQ \to \ZZ^n$) from the
corresponding flange, flat, or injective presentation or resolution
of~$\iota_*\cH$ afforded by Theorem~\ref{t:finitely-determined}.
These pullbacks are finite indicator resolutions of~$\cM$
dominating~$\pi$ by Example~\ref{e:pullback} and
Lemma~\ref{l:constant}.  The component homomorphisms are connected
because, by Corollary~\ref{c:U->D} and Example~\ref{e:connected-poset}
(see Definition~\ref{d:connected-poset}), components of flange
presentations, flat resolutions, and injective resolutions
over~$\ZZ^n$ are~\mbox{automatically}~\mbox{connected}.

The preceding argument proves the claim about a morphism $\cM \to
\cM'$, as well, since
\begin{itemize}
\item%
only one poset morphism is required to encode the morphism $\cM \to
\cM'$;
\item%
the push-pull constructions are functorial; and
\item%
morphisms of finitely determined modules can be lifted to the relevant
presentations and resolutions, since the relevant covers,
presentations, and resolutions are free or injective in the category
of finitely determined modules.\qedhere
\end{itemize}
\end{proof}


\begin{remark}
Comparing Theorems~\ref{t:syzygy} and~\ref{t:finitely-determined},
what happened to minimality?  It is not clear in what generality
minimality can be characterized.  The sequel \cite{essential-real} to
this paper can be seen as a case study for posets arising from abelian
groups that are either finitely generated and free
or real vector spaces of finite dimension.
The answer is much more nuanced in the real case, obscuring how
minimality might generalize beyond these~cases.
\end{remark}

\begin{remark}\label{r:pullback-twice}
In the situation of the proof of Theorem~\ref{t:syzygy}, composing two
applications of Proposition~\ref{p:pullback-monomial-matrix}---one for
the encoding $\pi: \cQ \to \cP$ and one for the embedding $\iota: \cP
\into \ZZ^n$---yields a monomial matrix for a fringe presentation
of~$\cM$ directly from a monomial matrix for a flange presentation.
\end{remark}

\begin{remark}\label{r:RRn-mod}
Lesnick and Wright consider \mbox{$\RR^n$-modules}
\cite[\S2]{lesnick-wright2015} in finitely presented cases.  As
they indicate,
homological algebra of such $\RR^n$-modules is no different than
finitely generated \mbox{$\ZZ^n$-modules}.  This can be seen by finite
encoding: any finite poset in~$\RR^n$ is embeddable in~$\ZZ^n$,
because a product of finite chains is all that~is~needed.
\end{remark}

\subsection{Syzygy theorem for complexes of modules}\label{sub:complexes}\mbox{}

\noindent
Theorem~\ref{t:syzygy} is stated for individual modules, but the proof
works just as well for complexes, in a sense recorded here for
reference in the proof of a version in the language of derived
categories of constructible sheaves \cite[Theorem~4.5]{strat-conical}.

\begin{defn}\label{d:tame-complex}
Fix a complex $M^\spot$ of modules over a poset~$\cQ$.
\begin{enumerate}
\item%
$\cM^\spot$ is \emph{tame} if its modules and morphisms are tame
(Definitions~\ref{d:tame} and~\ref{d:tame-morphism}).

\item%
A constant subdivision or poset encoding is \emph{subordinate}
to~$\cM^\spot$ if it is subordinate to all of the modules and
morphisms therein, and in that case $\cM^\spot$ is said to
\emph{dominate} the subdivision or encoding.

\item%
An \emph{upset resolution} of~$\cM^\spot$ is a complex of
$\cQ$-modules in which each $F_i$ is a direct sum of upset modules and
the components $\kk[U] \to \kk[U']$ are connected, with a homomorphism
$F^\spot \to \cM^\spot$ of complexes inducing an isomorphism on
homology.

\item%
A \emph{downset resolution} of~$\cM^\spot$ is a complex of
$\cQ$-modules in which each $E_i$ is a direct sum of downset modules
and the components $\kk[D] \to \kk[D']$ are connected, with a
homomorphism $\cM^\spot \to E^\spot$ of complexes inducing an
isomorphism on homology.
\end{enumerate}
These resolutions are \emph{finite}, or \emph{dominate} a constant
subdivision or encoding, or are \emph{semialgebraic}, \emph{PL},
\emph{subanalytic}, or \emph{of class~$\mathfrak X$} as in
Definition~\ref{d:resolutions}.
\end{defn}

\begin{thm}[Syzygy theorem for complexes]\label{t:syzygy-complexes}
Theorem~\ref{t:syzygy} holds verbatim for a bounded complex $M^\spot$
in place of the module~$M$ as long as items~\ref{i:fringe},
\ref{i:upset-presentation}, and \ref{i:downset-copresentation} are
ignored.
\end{thm}
\begin{proof}
As already noted, the proof is the same.  It bears mentioning that
finite injective and flat resolutions of complexes exist in the
category of finitely determined $\ZZ^n$-modules because finite
injective resolutions do (Proposition~\ref{p:determined}): any of the
standard constructions that produce injective resolutions of complexes
given that modules have injective resolutions works in this setting,
and then Matlis duality (Definition~\ref{d:matlis}) produces finite
flat resolutions (see Remark~\ref{r:flat}).
\end{proof}

\addtocontents{toc}{\protect\setcounter{tocdepth}{2}}

\vspace{-1.8ex}

\end{document}